\documentclass[11pt]{amsart}
\usepackage{amsmath}
\usepackage{amssymb}
\usepackage{amsthm}
\usepackage{amscd}
\usepackage{amsfonts}
\usepackage{graphicx}
\usepackage{fancyhdr}
\usepackage[utf8]{inputenc}
\usepackage{mathrsfs}
\usepackage{tikz-cd}
\usepackage[all,cmtip]{xy}
\usepackage{hyperref}
\numberwithin{equation}{section}
\newtheorem{theorem}[equation]{Theorem}
\newtheorem*{thm}{Theorem}
\newtheorem{Thm}{Theorem}
\theoremstyle{plain}
\newtheorem{assumption}[equation]{Assumption}
\newtheorem{lemma}[equation]{Lemma}
\newtheorem{proposition}[equation]{Proposition}
\newtheorem{definition}[equation]{Definition}
\newtheorem{corollary}[equation]{Corollary}

\newtheorem*{corollary*}{Corollary}

\theoremstyle{definition}
\newtheorem{remark}[equation]{Remark}

\newenvironment{myproof}[2] {\emph{Proof of {#1} {#2}.}}{\hfill$\square$}

\def\Gal{\mathrm{Gal}}
\def\GL{\mathrm{GL}}
\def\GSp{\mathrm{GSp}}

\def\Sp{\mathrm{Sp}}

\def\GU{\mathrm{GU}}

\def\Nilp{\mathrm{Nilp}}

\def\det{\mathrm{det}}
\def\ord{\mathrm{ord}}
\def\Lie{\mathrm{Lie}}

\def\Ker{\mathrm{Ker}}
\def\Trd{\mathrm{Trd}}
\def\Nrd{\mathrm{Nrd}}
\def\Pa{\mathrm{Pa}}
\def\Kl{\mathrm{Kl}}
\def\sing{\mathrm{sin}}

\def\rmc{\mathrm{c}}

\def\Coker{\mathrm{Coker}}
\def\Ker{\mathrm{Ker}}
\def\ss{\mathrm{ss}}

\def\Iw{\mathrm{Iw}}
\def\Spf{\mathrm{Spf}}
\def\Spec{\mathrm{Spec}}

\def\et{\'{e}tale}

\def\min{\mathrm{min}}
\def\lr{\mathrm{lr}}
\def\Frob{\mathrm{Frob}}

\DeclareMathOperator{\Hom}{Hom}

\DeclareMathOperator{\End}{End}

\def\calC{\mathcal{C}}

\def\calL{\mathcal{L}}

\def\calM{\mathcal{M}}
\def\calO{\mathcal{O}}

\def\calT{\mathcal{T}}

\def\frakm{\mathfrak{m}}

\def\CC{\mathbb{C}}

\def\FF{\mathbb{F}}
\def\GG{\mathbb{G}}

\def\PP{\mathbb{P}}
\def\QQ{\mathbb{Q}}
\def\RR{\mathbb{R}}
\def\TT{\mathbb{T}}
\def\XX{\mathbb{X}}
\def\VV{\mathbb{V}}
\def\ZZ{\mathbb{Z}}

\def\sfA{\mathsf{A}}
\def\sfX{\mathsf{X}}

\def\rme{\mathrm{e}}

\def\rmu{\mathrm{u}}
\def\rms{\mathrm{s}}

\def\rmB{\mathrm{B}}
\def\rmC{\mathrm{C}}
\def\rmD{\mathrm{D}}
\def\rmF{\mathrm{F}}
\def\rmG{\mathrm{G}}
\def\rmH{\mathrm{H}}
\def\rmI{\mathrm{I}}
\def\rmJ{\mathrm{J}}
\def\rmK{\mathrm{K}}
\def\rmL{\mathrm{L}}
\def\rmN{\mathrm{N}}
\def\rmM{\mathrm{M}}
\def\rmP{\mathrm{P}}
\def\rmQ{\mathrm{Q}}
\def\rmU{\mathrm{U}}
\def\rmV{\mathrm{V}}
\def\rmW{\mathrm{W}}
\def\rmX{\mathrm{X}}
\def\rmY{\mathrm{Y}}
\def\rmR{\mathrm{R}}
\def\rmS{\mathrm{S}}
\def\rmT{\mathrm{T}}
\def\rmZ{\mathrm{Z}}
\def\sfH{\mathsf{H}}

\newcommand{\hooklongrightarrow}{\lhook\joinrel\longrightarrow}

\usepackage[top=1.5in, bottom=1.5in, left=1.3in, right=1.3in]{geometry}
\author{Haining Wang }
\address{\parbox{\linewidth}{Haining Wang\\Shanghai Center for Mathematical Sciences,\\ Fudan University,\\No.2005 Songhu Road,\\Shanghai, 200438, China.~ }}
\email{wanghaining1121@outlook.com}
\date{\today}

\begin{document}
\title[Level lowering for paramodular level]{Level lowering on Siegel modular threefold of paramodular level}


\begin{abstract}
In this article we prove several level lowering results for cuspidal automorphic representations occurring in the cohomology of the Siegel modular threefold with paramodular level structure by adapting a method of Ribet in his proof of the Serre's epsilon conjecture. The proof is purely geometric and relies on the description of the supersingular locus of certain quaternionic unitary Shimura variety and an arithmetic level raising result on this Shimura variety. The heart of the proof is a comparison of the dimension of the space of vanishing cycles on the Siegel modular threefold with paramodular level structure with that on the quaternionic unitary Shimura variety. 
\end{abstract}

\maketitle
\tableofcontents
\section{Introduction}
\subsection{Motivations} In this article we prove three level lowering principles for Galois representations attached to  cuspidal automorphic representation of the symplectic similitude group $\GSp_{4}$ of degree $4$ over $\QQ$. These results are inspired by the celebrated work of Ribet on Serre's epsilon conjecture \cite{Rib90}, \cite{Rib91}.  To motivate our results, we first recall these results in the setting of modular forms. Let $l\geq 3$ be a prime and we fix an isomorphism $\iota_{l}: \CC\cong \overline{\QQ}_{l}$. Denote by $\rmG_{\QQ}=\Gal(\overline{\QQ}/\QQ)$.  Let 
\begin{equation*}
\rho: \rmG_{\QQ}\rightarrow \GL_{2}(k)
\end{equation*}
be an absolutely irreducible representation of $\rmG_{\QQ}$ valued in a finite field $k$ of characteristic $l$. Suppose that $\rho$ is modular of level $\rmN$ which means $\rho$ comes as the residual representation of the $\lambda$-adic representation 
\begin{equation*}
\rho_{f, \lambda}:  \rmG_{\QQ}\rightarrow \GL_{2}(E_{\lambda})
\end{equation*}
attached to a weight $2$ newform $f$ of level $\Gamma_{0}(\rmN)$. Here $E$ is the Hecke field of $f$ and $\lambda$ is the place in $E$ over $l$ induced by $\iota_{l}$. Let $p$ be a prime which divides $\rmN$ exactly once and suppose that $\rho$ is unramified at $p$. Then Serre \cite{Ser87} conjectured that $\rho$ is modular of level $\rmN/p$. This conjecture is known as the Serre's epsilon conjecture.  The following theorem is the main result of \cite{Rib90}.
\begin{thm}[Serre's epsilon conjecture]\label{epsilon}
Assume $\rho$ is modular of level $\rmN$ and is unramified at $p$ with $p\mid\mid \rmN$. Then $\rho$ is modular of level $\rmN/p$ if $\rmN$ is prime to $l$. 
\end{thm}

The proof of the above theorem in fact breaks into two parts. First one proves the following theorem which is known as the Mazur's principle.
\begin{thm}[Mazur's principle]
Assume $\rho$ is modular of level $\rmN$ and unramified at $p$ with $p\mid\mid \rmN$. Then $\rho$ is modular of level $\rmN/p$ if $\rmN$ is prime to $l$ and $p\not\equiv 1 \mod l$.
\end{thm}

The assumption that $p\not\equiv 1 \mod l$ means that the two Frobenius eigenvalues of $\rho$ at $p$ are distinct modulo $l$. In this case, the geometric argument simplifies significantly and one can prove this theorem by analyzing the bad reduction of a suitable modular curve. To treat the case when $p\equiv 1\mod l$, one chooses an auxiliary prime $q$ such that $q\not\equiv 1\mod l$ and raise the level of $f$ to $\Gamma_{0}(q\rmN)$. Then one can regard $f$ as a modular form on the indefinite quaternion algebra $\rmB=\rmB_{pq}$ that is ramified at $pq$ by the Jacquet--Langlands correspondence. Then the so-called $(p, q)$-switch trick allows one to drop the level at $p$ and the Mazur's principle allows one to drop the level at $q$. The original proof in  \cite{Rib90} relies on certain multiplicity one results on the Galois representations appearing in the torison points of the Jacobian of the modular and Shimura curves which are hard to generalize to other Shimura varieties. The reliance on these multiplicity one results are lifted in  \cite{Rib91} and in the same article Ribet proves the following theorem which we will refer to as the Ribet's principle.
 
\begin{thm}[Ribet's principle]
Let $\rmN$ be a positive integer. Let $p, q$ be two distinct primes which are coprime to $l\rmN$. Assume $\rho$ is modular of level $pq\rmN$.  Assume that 
$\rho$ is unramified at $p$ but ramified at $q$. Then $\rho$ is modular of level $q\rmN$. 
\end{thm}

The existence of the auxiliary prime $q$ allows Ribet to use the forementioned $(p,q)$-switch trick in a simplified way. The crucial observation for the $(p,q)$-switch trick is that the singular locus of a Shimura curve associated to a suitable indefinite quaternion algebra and the singular locus of a suitable modular curve can be parametrized by the same discrete Shimura set. We make this observation more precise now. Let $\mathbb{A}$ be the ring of ad\`{e}les over $\QQ$.  Let $\rmB=\rmB_{pq}$ be an indefinite quaternion algebra with discriminant $pq$. Let $\rmX(\rmB)$ (for the moment) be the Shimura curve associated to $\rmB$ with level $\Gamma_{0}(\rmN)$ and let $\rmX_{0}(pq)$ be the usual modular curve with $\Gamma_{0}(pq\rmN)$ level structure. Consider the special fiber  
$\overline{\rmX}(\rmB)$ of $\rmX(\rmB)$ at $p$. The singular locus of $\overline{\rmX}(\rmB)$ can be identified with the double coset
\begin{equation*}
\rmZ_{\Iw}(\overline{\rmB})=\overline{\rmB}^{\times}(\QQ)\backslash \overline{\rmB}^{\times}(\mathbb{A}^{(\infty)})/\rmU^{\prime}(pq)\rmU_{0}(\rmN) 
\end{equation*}
where $\overline{\rmB}=\rmB_{q\infty}$ is the definite quaternion algebra with discriminant $q$, $\rmU_{0}(\rmN)$ is the tame level structure corresponding to $\Gamma_{0}(\rmN)$ and $\rmU^{\prime}(pq)={\rm{Iw}}(p)\calO^{\times}_{\rmD^{\prime}}$ where ${\rm{Iw}}(p)$ is the Iwahori subgroup of $\GL_{2}(\QQ_{p})$ and $\calO_{\rmD^{\prime}}$ is the maximal order of $\rmD^{\prime}=\overline{\rmB}_{q}$.  On the other hand, the singular locus of $\overline{\rmX}_{0}(pq)_{\overline{\FF}_{q}}$ at the prime $q$ can be identified with the same double coset. Therefore the vanishing cycles on $\overline{\rmX}(\rmB)_{\overline{\FF}_{p}}$ and on $\overline{\rmX}_{0}(pq)_{\overline{\FF}_{q}}$ can be identified. The starting point of this article is the observation that the same phenomenon happens for the paramodular Siegel threefold and its quaternionic analogue. 

\subsection{Main results} Let $\pi$ be a cuspidal automorphic representation of $\GSp_{4}(\mathbb{A})$ which is of general type in the sense of Arthur \cite{Arthur-GSp} and whose weight is $(k_{1}, k_{2})$ satisfying $k_{1}\geq k_{2}\geq 3$ and $k_{1}\equiv k_{2}\mod 2$. Then one can attach a Galois representation
\begin{equation*}
\rho_{\pi, \lambda}: \rmG_{\QQ}\rightarrow \GSp_{4}(E_{\lambda})
\end{equation*}
for a finite extension $E_{\lambda}$ over $\QQ_{l}$. We will denote by 
\begin{equation*}
\overline{\rho}_{\pi,\lambda}:  \rmG_{\QQ}\rightarrow \GSp_{4}(k)
\end{equation*}
its associated residual representation valued in the residue field $k$ of $E_{\lambda}$. In fact, this Galois representation is realized in the middle degree cohomology of the Siegel modular threefold. More precisely, it is realized in $\Hom_{\GSp_{4}(\mathbb{A}^{(\infty)})}(\pi_{f}, \rmH^{3}_{\rmc}(\rmX\otimes{\overline{\QQ}}, \VV))$ where $\rmH^{3}_{\rmc}(\rmX\otimes{\overline{\QQ}}, \VV)=\varinjlim_{\rmU}\rmH^{3}_{\rmc}(\rmX_{\rmU}\otimes{\overline{\QQ}}, \VV)$ is the compact support cohomology of the Siegel modular threefold indexed by the open compact subgroups $\rmU$ of $\GSp_{4}(\mathbb{A}^{(\infty)})$ and $\VV$ is a suitable \'etale local system which we will make it more precise in the main body of the article. Recall that there are two maximal parahoric subgroups of $\GSp_{4}(\QQ_{p})$ up to isomorphism, one is known as the hyperspecial subgroup which we will denote by $\rmH$ and the other is known as the paramodular subgroup which we will denote by $\Pa(p)$. 
Suppose that the local component of $\pi$ at $p$ is paraspherical and ramified which means that $\pi$ is of type $\mathrm{IIa}$ in the table of Schmidt \cite{Schmi05}.  We will let $\rmU=\Pa(p)\rmU^{(p)}$ where $\rmU^{(p)}\subset \GSp_{4}(\mathbb{A}^{(\infty p)})$ such that $\rmU$ is a neat open compact subgroup such that $\pi^{\rmU}\neq 0$. Let $\rmX_{\Pa}(p)=\rmX_{\rmU}$ be the Siegel modular threefold with paramodular level at $p$. Let $\TT$ be a suitable Hecke algebra acting on the middle degree cohomology of $\rmX_{\Pa}(p)$ and $\frakm$ be the maximal ideal in $\TT$ corresponding to $\overline{\rho}_{\pi, \lambda}$. Our first result is the analogue of Mazur's principle for Galois representations occurring in the middle degree cohomology of $\rmX_{\Pa}(p)$.
\begin{Thm}\label{Mz-principle}
Let $\pi$ be a cuspidal automorphic representation of $\GSp_{4}(\mathbb{A})$ which is of general type and whose weight is $(k_{1}, k_{2})$ satisfying $k_{1}\geq k_{2} \geq 3$ and $k_{1}\equiv k_{2}\mod 2$. We assume that $l+1\geq k_{1}+k_{2}$.  Let $p$ be a prime distinct from $l$ and such that $p\not\equiv 1 \mod l$. Let $\rmU=\Pa(p)\rmU^{p}\subset \GSp_{4}(\QQ_{p})\GSp_{4}(\mathbb{A}^{(\infty p)})$ be a neat open compact subgroup such that $\pi^{\rmU}\neq 0$. We assume that $\pi_{p}$ is ramified of type $\mathrm{IIa}$ in the table of Schmidt. Let $\frakm$ be the maximal ideal associated to $\overline{\rho}_{\pi,\lambda}$. Suppose the residual Galois representation $\overline{\rho}_{\pi, \lambda}$ satisfies the following assumptions.
\begin{enumerate}
\item $\overline{\rho}_{\pi, \lambda}$ is unramified at $p$;
\item $\overline{\rho}_{\pi, \lambda}$ is absolutely irreducible and the image of $\overline{\rho}_{\pi, \lambda}(\rmG_{\QQ})$ contains $\GSp_{4}(\FF_{l})$;
\item $\rmH^{i}_{\rmc}(\rmX_{\Pa}(p)\otimes \overline{\QQ}, \VV)_{\frakm}$ is concentrated in degree $3$;
\item $\rmH^{3}_{\rmc}(\rmX_{\Pa}(p)\otimes \overline{\QQ}, \VV)_{\frakm}\otimes k$ is semisimple as a Galois module.
\end{enumerate}
Then there exists a cuspidal automorphic representation $\breve{\pi}$ with the same weight as $\pi$ such that 
\begin{equation*}
\overline{\rho}_{\pi, \lambda}\cong \overline{\rho}_{\breve{\pi}, \lambda}
\end{equation*}
and such that $\breve{\pi}_{p}$ is an unramified principal series.
\end{Thm} 

We remark that in a pioneering work \cite{VH19},  van Hoften was able to prove the Mazur's principle in this setting with different assumptions. In particular we are not relying on the vanishing results of Lan--Suh \cite{LS13} and thus we allow more general weights for the \'etale local system $\VV$. In addition, we do not require that the Frobenius eigenvalues of $\overline{\rho}_{\pi, \lambda}$ at $p$ to be distinct instead we only assume that $p\not\equiv 1 \mod l$. However we introduce the additional assumption that  $\rmH^{3}_{\rmc}(\rmX_{\Pa}(p)\otimes\overline{\QQ}, \VV)_{\frakm}\otimes k$ is semisimple.  The assumption is the analogue of the main result of \cite{BLR91}. The proof of \cite{BLR91} relies on the fact the representation of interest is two dimensional and therefore their methods do not generalize directly to our setting. We also introduce the additional assumption $\rmH^{i}_{\rmc}(\rmX_{\Pa}(p)\otimes \overline{\QQ}, \VV)_{\frakm}$ is concentrated in degree $3$. This holds true if we impose further conditions on the maximal ideal $\frakm$ or on the weights of $\VV$.  We will discuss this point in Remark \ref{CS-van}. The most important step in proving the Mazur's principle is to show that the cohomology group $\rmH^{2}_{\rmc}(\overline{\rmX}_{\Pa}(p)\otimes\overline{\FF}_{p}, \VV)_{\frakm}$ vanishes. For this, we follow the strategy of \cite{VH19} closely. Comparing to our work \cite{Wang} in quaternionic setting, the reason for this vanishing seems to be the lack of Tate cycles coming from the supersingular locus. Indeed, the supersingular locus of $\overline{\rmX}_{\Pa}(p)$ is one dimensional and thus can not contribute to the degree $2$ cohomology. This is in contrast to the case of the quaternionic unitary Shimura variety which we will discuss later.

Next we state our result on Ribet's principle. For this, let $p,q$ be two distinct primes which are different from $l$ but now $p$ could be congruent to $1$ modulo $l$. We let $\rmU^{pq}$ be an open compact subgroup of $\GSp_{4}(\mathbb{A}^{(\infty pq)})$ such that $\rmU=\Pa(p)\Pa(q)\rmU^{(pq)}$ is a neat open compact subgroup of  $\GSp_{4}(\mathbb{A}^{(\infty)})$ such that $\pi^{\rmU}\neq 0$.   We consider the Siegel modular threefold $\rmX_{\Pa}(pq)=\rmX_{\rmU}$ with paramodular level at $p$ and $q$ . We will refer to following theorem as the Ribet's principle.

\begin{Thm}
Let $\pi$ be a cuspidal automorphic representation of $\GSp_{4}(\mathbb{A})$ of general type and whose weight is $(k_{1}, k_{2})$ satisfying $k_{1}\geq k_{2} \geq 3$ and $k_{1}\equiv k_{2}\mod 2$. We assume that $l+1\geq k_{1}+k_{2}$. Let $p, q$ be two distinct primes different from $l$. Let $\rmU=\Pa(p)\Pa(q)\rmU^{pq}$ be a neat open compact subgroup such that $\pi^{\rmU}\neq 0$. We assume that $\pi_{p}$ and $\pi_{q}$ are both ramified of type $\mathrm{IIa}$. Suppose the residual Galois representation $\overline{\rho}_{\pi, \lambda}$ satisfies the following assumptions.
\begin{enumerate}
\item $\overline{\rho}_{\pi, \lambda}$ is unramified at $p$ with generic Hecke parameters;
\item $\overline{\rho}_{\pi, \lambda}$ is ramified at $q$;
\item $\overline{\rho}_{\pi, \lambda}$ is absolutely irreducible and $\overline{\rho}_{\pi,\lambda}(\rmG_{\QQ})$ contains $\GSp_{4}(\FF_{l})$;
\item $\rmH^{i}_{\rmc}(\rmX_{\Pa}(pq)\otimes\overline{\QQ}, \VV)_{\frakm}$ and $\rmH^{i}_{\rmc}(\rmX(\rmB)\otimes\overline{\QQ}, \VV)_{\frakm}$ are concentrated in degree $3$;
\item $\rmH^{3}_{\rmc}(\rmX_{\Pa}(pq)\otimes\overline{\QQ}, \VV)_{\frakm}\otimes k$ is a semisimple Galois module. 
\end{enumerate}
Then there exists a cuspidal automorphic representation $\breve{\pi}$ with the same weight as $\pi$ and such that 
\begin{equation*}
\overline{\rho}_{\pi, \lambda}\cong \overline{\rho}_{\breve{\pi}, \lambda} 
\end{equation*}
and $\breve{\pi}_{p}$ is an unramified principal series.
\end{Thm} 

Although we allow $p\equiv 1\mod l$, we need to impose a mild assumption called generic on the Hecke parameters of $\pi$ at $p$. The precise definition of this assumption will be given in Definition \ref{generic-at-p}. The proof of this theorem relies heavily on the interplay between the geometry of the Siegel modular threefold $\rmX_{\Pa}(pq)$ and its quaternionic analogue $\rmX(\rmB)$. Here the quaternion algebra $\rmB$ is again the indefinite quaternion algebra over $\QQ$ with discriminant $pq$ as in the $\GL_{2}$-case but $\rmX(\rmB)$ is now the Shimura variety associated to the quaternionic unitary group $\GU_{2}(\rmB)$ of degree $2$. We will rely on the description of the supersingular locus of $\overline{\rmX}(\rmB)_{\overline{\FF}_{p}}$ in \cite{Wanga} which is also obtained by Oki \cite{Oki} independently. In particular the starting point of this work is the observation that the singular locus of $\rmX_{\Pa}(pq)_{\overline{\FF}_{p}}$ and the singular locus $\rmX(\rmB)_{\overline{\FF}_{q}}$ are parametrized by the same discrete Shimura set and therefore one can identify the space of vanishing cycles on $\rmX_{\Pa}(pq)_{\overline{\FF}_{p}}$ and the space of vanishing cycles on $\rmX(\rmB)_{\overline{\FF}_{q}}$. Under this observation, we assume that the level of $\overline{\rho}_{\pi,\lambda}$ can not be lowered and we compare the dimension of the space of vanishing cycles on $\rmX_{\Pa}(pq)$ and $\rmX(\rmB)$ at the primes $p$ and $q$ to derive a contradiction. We will also use the fact that in this case the Tate cycles coming from the supersingular locus indeed generate the space $\rmH^{2}_{\rmc}(\overline{\rmX}(\rmB)\otimes \overline{\FF}_{p}, \VV)_{\frakm}$ once the Hecke parameters of $\pi$ is generic. This is in contrast to the situation of $\overline{\rmX}_{\Pa}(p)$ mentioned before. Indeed, the supersingular locus of $\overline{\rmX}(\rmB)$ is two dimensional and hence does contribute to $\rmH^{2}_{\rmc}(\overline{\rmX}(\rmB)\otimes \overline{\FF}_{p}, \VV)$.

Note that the auxiliary assumption that $\overline{\rho}_{\pi, \lambda}$ is ramified at $q$ simplify the situation significantly but can be removed once we impose further assumptions on the residual Galois representation.  This is the content of the main level lowering theorem in this article. 

\begin{Thm}
Let $\pi$ be a cuspidal automorphic representation of $\GSp_{4}(\mathbb{A})$ of general type and whose weight is $(k_{1}, k_{2})$ with  $k_{1}\geq k_{2}\geq 3$ and $k_{1}\equiv k_{2}\mod 2$. Suppose $l+1\geq k_{1}+k_{2}$. Let $p$ be a prime different from $l$. Let $\rmU=\Pa(p)\rmU^{(p)}$ with $\rmU^{(p)}\subset \GSp_{4}(\mathbb{A}^{(\infty p)})$ such that $\rmU$ is a neat open compact subgroup such that $\pi^{\rmU}\neq 0$. Suppose that $\pi_{p}$ is ramified of type $\mathrm{IIa}$. Suppose also that the residual Galois representation $\overline{\rho}_{\pi, \lambda}$ satisfies the following assumptions.
\begin{enumerate}
\item  $\overline{\rho}_{\pi, \lambda}$ is unramified with generic Hecke parameters at $p$;
\item  $\overline{\rho}_{\pi,\lambda}$ is rigid as in Definition \ref{rigid};
\item  $\rmH^{i}_{\rmc}(\rmX_{\Pa}(pq)\otimes\overline{\QQ}, \VV)_{\frakm}$ and $\rmH^{i}_{\rmc}(\rmX(\rmB)\otimes\overline{\QQ}, \VV)_{\frakm}$ are concentrated in degree $i=3$;
\item $\rmH^{3}_{\rmc}(\rmX_{\Pa}(pq)\otimes\overline{\QQ}, \VV)_{\frakm}\otimes k$ is a semisimple Galois module; 
\item $\overline{\rho}_{\pi,\lambda}$ is absolutely irreducible and the image of $\overline{\rho}_{\pi,\lambda}(\rmG_{\QQ})$ contains $\GSp_{4}(\FF_{l})$.
\end{enumerate}
Then there exists a cuspidal automorphic representation $\breve{\pi}$of general type with the same weight as $\pi$ such that 
\begin{equation*}
\overline{\rho}_{\pi, \lambda}\cong \overline{\rho}_{\breve{\pi}, \lambda} 
\end{equation*}
and $\breve{\pi}_{p}$ is an unramified principal series.
\end{Thm} 

To prove this theorem, first we exploit the level raising theorem in \cite[Theorem 9.9]{Wang} to obtain an automorphic representation which is congruent to $\pi$ and occurs in the cohomology of $\rmX(\rmB)$ as well as in the cohomology of 
$\rmX_{\Pa}(pq)$. Then we proceed similarly as in the proof of the Ribet's principle to calculate the dimensions of the spaces of vanishing cycles at $p$ and $q$ for these Shimura varieties. This calculation is complicated by the fact that the maximal ideal $\frakm$ corresponding to $\overline{\rho}_{\pi, \lambda}$ occurs both in the space of $q$-newforms and $q$-oldforms. Fortunately, our main arithmetic level raising result on $\rmX(\rmB)$ \cite{Wang} which is finer than the level raising theorem exactly enables us to overcome this difficulty. 

\subsection{Survey of level lowering results} 
We survey  some results in the literature on level lowering theorems. We do not attempt to exhaust all the results in this direction but only those the author is aware of.
\begin{itemize}
\item In the case of Hilbert modular forms,  Jarvis \cite{Jar99} proved Mazur's principle for Shimura curves over totally real field and  Rajaei \cite{Raj01} generalizes Ribet's method in this setting. See also the work of Fujiwara \cite{Fuji} in this direction. Note that there are also other techniques to prove level lowering results for Hilbert modular forms, notably the work of Skinner--Wiles \cite{SW01} proves a potential level lowering theorem. 
\item In the setting of Kottwitz--Harris--Taylor type Shimura varieties, the Mazur's principle is proved for the $\rmU(2, 1)$-case by Helm in \cite{Helm} and by Boyer \cite{Boyer} for the $\rmU(n-1, 1)$-case.
\item In the setting of $\GSp_{4}$, Sorensen \cite{Sor09a} is able to prove a potential level lowering result using the techniques of \cite{SW01}. More precisely, he establishes a congruence between representations that admits Iwahori fixed vectors and certain tamely ramified principal series. 
\item Using techniques from modularity lifting theorems, one can also prove level raising and level lowering theorems. For automorphic representations of $\GSp_{4}$ over totally real fields, one can deduce level lowering results from the main theorem of \cite{GG}. See also \cite{Gee} for the case of Hilbert modular forms and the case for $\GL_{n}$.
\end{itemize}

\subsection{Notations and conventions}  We will use common notations and conventions in algebraic number theory and algebraic geometry. The cohomology  of schemes appear in this article will be understood as computed over the \'{e}tale sites. 

For a field $\rmK$, we denote by $\overline{\rmK}$ a separable closure of $\rmK$ and put $\rmG_{\rmK}:=\Gal(\overline{\rmK}/\rmK)$ the Galois group of $\rmK$. Suppose $\rmK$ is a number field and $v$ is a place of $\rmK$, then $\rmK_{v}$ is the completion of $\rmK$ at $v$ with valuation ring $\calO_{v}$ whose maximal ideal is also written as $v$ if no danger of confusion could arise. Let $\mathrm{Art}:\rmK^{\times}_{v}\rightarrow \mathrm{W}^{\mathrm{ab}}_{\rmK}$ be the local Artin map sending the uniformizers to the geometric Frobenius element $\Frob_{p}$. We let $\phi_{v}$ be the arithmetic Frobenius element at $v$. We let $\mathbb{A}_{\rmK}$ be the ring of ad\`{e}les over $\rmK$ and $\mathbb{A}^{(\infty)}_{\rmK}$ be the subring of finite ad\`{e}les.  Let $l\geq 3$ be a prime and we fix an isomorphism $\iota_{l}: \CC\cong \overline{\QQ}_{l}$. We will denote by $\epsilon_{l}$ the $l$-adic cyclotomic character. The prime $l$ will be the residue characteristic of the coefficient

Let $\rmK$ be a local field with ring of integers $\calO_{\rmK}$ and residue field $k$. We let $\rmI_{\rmK}$ be the inertia subgroup of $\rmG_{K}$. Suppose $\rmM$ is a $\rmG_{\rmK}$-module. Then the finite part $\rmH^{1}_{\mathrm{fin}}(\rmK, \rmM)$ of $\rmH^{1}(\rmK, \rmM)$ is defined to be $\rmH^{1}(k, \rmM^{\rmI_{\rmK}})$ and the singular part $\rmH^{1}_{\mathrm{sing}}(\rmK, \rmM)$ of $\rmH^{1}(\rmK, \rmM)$ is defined to be the quotient of $\rmH^{1}(\rmK, \rmM)$ by the image of $\rmH^{1}_{\mathrm{fin}}(\rmK, \rmM)$. 

Let $p$ be a prime and let $\FF$ be an algebraically closed field containing $\FF_{p}$. We denote by $\sigma$ be the Frobenius element on $\FF$. Let $\rmW_{0}=\rmW(\FF)$ be the ring of Witt vectors of $\FF$ and $\rmK_{0}=\rmW(\FF)_{\QQ}$ its fraction field, then $\sigma$ extends to an action on $\rmW(\FF)$. Let $\FF_{p^{d}}$ be the finite field of $p^{d}$ elements. Then $\ZZ_{p^{d}}$ is defined to be $\rmW(\FF_{p^{d}})$. Let $\rmM_{1}\subset \rmM_{2}$ be two $\rmW_{0}$-modules we wrire $\rmM_{1}\subset^{d} \rmM_{2}$ if the $\rmW_{0}$ colength of the inclusion is $d$.

If $\rmR$ is ring and $\rmL$ is an $\rmR$-module and $\rmR^{\prime}$ is an $\rmR$-algebra, we define $\rmL_{\rmR^{\prime}}=\rmL\otimes_{\rmR} \rmR^{\prime}$. Let $\rmX$ be a scheme or a formal scheme over $\rmR$, we write $\rmX_{\rmR^{\prime}}$ its base change to $\rmR^{\prime}$. Let $\rmM_{1}\subset \rmM_{2}$ be two $\rmR$-modules, then we write $\rmM_{1}\subset_{\mathrm{Gr}\rmM}\rmM_{2}$ if $\rmM_{2}/\rmM_{1}=\mathrm{Gr}\rmM$.

For our convention, $\GSp_{4}$ be the reductive group over $\ZZ$ defined by 
\begin{equation*}
\GSp_{4}=\{g\in\GL_{4}: g\rmJ g^{t}=\rmc(g)\rmJ\}
\end{equation*}
where $\rmJ$ is the antisymmetric matrix given by $\begin{pmatrix}0&s\\-s&0\\ \end{pmatrix}$
for $s=\begin{pmatrix}0&1\\ 1&0 \\ \end{pmatrix}$ and where $\rmc$ is the similitude character of $\GSp_{4}$. By definition, we have an exact sequence 
\begin{equation*}
1\rightarrow \Sp_{4}\rightarrow \GSp_{4}\xrightarrow{\rmc}\GG_{m}\rightarrow 1.
\end{equation*}

\subsection{Acknowledgement.} This work started when the author was a postdoctoral fellow at McGill university and he would like to thank Henri Darmon and Pengfei Guan for their generous support during a difficult time. The author would like to thank Marc-Hubert Nicole for informing him the work of 
van Hoften \cite{VH19} and discussions about the Siegel modular threefold with paramodular level structure. The author is grateful to Kai-Wen Lan, Yifeng Liu and Liang Xiao for helpful conversations and correspondences. It is needless to say that any remaining errors or inaccuracies are due to the author himself.

\section{Review of nearby and vanishing cycles}\label{review-cyc}
\subsection{Nearby and vanishing cycles}
Let $S=\Spec(R)$ be the spectrum of a henselian DVR. We choose a uniformizer $\pi$ of $R$. We denote by $\rms$ the closed point of $S$ and by $\eta$ the generic point of $\rmS$. We assume the residue field $k(s)$ at $s$ is of characteristic $p$. Let $\overline{s}$ be a geometric point of $S$ over $s$ and $\overline{\eta}$ be a seperable closure of $\eta$. For a morphsim $f: \rmX\rightarrow S$, we obtain the following maps
\begin{equation*}
\rmX_{\overline{s}}\xrightarrow{\overline{i}} \rmX_{\overline{S}} \xleftarrow{\overline{j}} \rmX_{\overline{\eta}}\\
\end{equation*}
by base change.
Let $\rmK\in \rmD^{+}(\rmX_{\overline{\eta}}, \Lambda)$ with coefficient in $\Lambda=\calO_{\lambda}/ \lambda^{m}$ or $\calO_{\lambda}$, then we define the \emph{nearby cycle complex} $\rmR\Psi(K)\in \rmD^{+}(\rmX_{\overline{s}}, \Lambda)$ by
\begin{equation*}
\rmR\Psi(\rmK)= \overline{i}^{*}\rmR\overline{j}_{*}(\rmK\vert\rmX_{\overline{\eta}}).
\end{equation*}
For $\rmK\in \rmD^{+}(\rmX, \Lambda)$, the adjunction map defines the following distinguished triangle
\begin{equation*}
\rmK\vert\rmX_{\overline{s}}\rightarrow \rmR\Psi(\rmK\vert\rmX_{\overline{\eta}})\rightarrow \rmR\Phi(\rmK)\rightarrow
\end{equation*}
The complex $\rmR\Phi(\rmK)$ is known as the \emph{vanishing cycle complex} and is the cone of the previous map. If $f$ is proper, then we have 
\begin{equation*}
\rmR\Gamma(\rmX_{\overline{\eta}}, \rmK\vert\rmX_{\overline{\eta}})= \rmR\Gamma(\rmX_{\overline{s}}, \rmR\Psi(\rmK\vert\rmX_{\overline{\eta}}))
\end{equation*}
and in general we always have the following long exact sequence 
\begin{equation*}
\cdots \rightarrow \mathrm{H}^{i}(\rmX_{\overline{s}}, \rmK\vert\rmX_{\overline{s}})\xrightarrow{\rm{sp}}\mathrm{ H}^{i}(\rmX_{\overline{s}}, \rmR\Psi(\rmK\vert\rmX_{\overline{\eta}})) \rightarrow  \mathrm{H}^{i}(\rmX_{\overline{s}}, \rmR\Phi(\rmK))\rightarrow \cdots.
\end{equation*}
The first map $\rm{sp}$ is called the specialization map and the cohomology of vanishing cycles  $\mathrm{H}^{i}(\rmX_{\overline{s}}, \rmR\Phi(\rmK))$ measures the defect of $\rm{sp}$ from being an isomorphism. If we assume that $f$ is smooth, then  $\mathrm{H}^{i}(\rmX_{\overline{s}}, \rmR\Phi(\rmK))$ vanishes and the specialization map is an isomorphism. 
\subsection{Isolated singularities} Suppose that $f: \rmX\rightarrow S$ is a regular, flat, finite type morphism of relative dimension $n$ which is smooth outside a finite collection $\Sigma$ of closed points in $\rmX_{\overline{s}}$. Let $\VV$ be a smooth $\calO_{\lambda}$-sheaf. In this case we have \begin{equation*} 
\rmR\Phi(\VV)|\rmX_{\overline{s}}-\Sigma=0
\end{equation*}
and moreover
\begin{equation*}
\rmR\Phi(\VV)=\bigoplus_{\sigma\in \Sigma} \rmR^{n}\Phi_{\sigma}(\VV)
\end{equation*}
is concentrated at degree $n$. Therefore we have the following  exact sequence
\begin{equation}\label{van-cycle-ext}
\begin{split}
&0\rightarrow \mathrm{H}^{n}(\rmX_{\overline{s}}, \VV)\xrightarrow{\rm{sp}}  \mathrm{H}^{n}(\rmX_{\overline{s}}, \rmR\Psi(\VV)) \xrightarrow{\mathrm{ca}} \bigoplus_{\sigma\in \Sigma} \rmR^{n}\Phi_{\sigma}(\VV)\xrightarrow{\alpha}\\ 
&\mathrm{H}^{n+1}(\rmX_{\overline{s}},\VV)\rightarrow \mathrm{H}^{n+1}(\rmX_{\overline{s}}, \rmR\Psi(\VV))\rightarrow 0.
\end{split}
\end{equation}

Assume next that for every $\sigma\in \Sigma$ is an ordinary quadratic singularity. This means $\rmX$ is {\et} locally near $x$ isomorphic to 
\begin{itemize}
\item  $\rmV(\Sigma_{1\leq i\leq m}\rmZ_{i}\rmZ_{i+m}+\pi)\subset \mathbb{A}^{2m}_{S}$ if $n=2m-1$;
\item  $\rmV(\Sigma_{1\leq i\leq m}\rmZ_{i}\rmZ_{i+m}+\rmZ^{2}_{2m+1}+\pi)\subset \mathbb{A}^{2m+1}_{S}$ if $n=2m$.
\end{itemize}
Then in this case 
\begin{equation*}
\rmR^{n}\Phi_{\sigma}(\VV)\cong\VV
\end{equation*} noncanonically. Let $I_{\eta}\subset \Gal(\overline{\eta}/\eta)$ be  the inertia group and let $\xi\in I_{\eta}$. Then we have the  variation morphism
$\mathrm{Var}_{\sigma}(\xi): \rmR^{n}\Phi_{\sigma}(\VV)\rightarrow \mathrm{H}^{n}_{\{\sigma\}}(\rmX_{\overline{s}}, \rmR\Psi(\VV) )$ and the action of $\xi-1$ on $\rmH^{n}(\rmX_{\overline{\eta}}, \VV)$ can be factored as 
\begin{equation*}
\begin{tikzcd}
\mathrm{H}^{n}(\rmX_{\overline{\rms}}, \rmR\Psi(\VV)) \arrow[r] \arrow[d, "\xi-1"] & \bigoplus\limits_{\sigma\in\Sigma}\rmR^{n}\Phi_{\sigma}(\VV) \arrow[d, "\mathrm{Var}_{\sigma}(\xi)"]\arrow[r, "\alpha"]  &\mathrm{H}^{n+1}(\rmX_{\overline{s}}, \VV))\\
 \mathrm{H}^{n}_{\rmc}(\rmX_{\overline{\rms}}, \rmR\Psi(\VV))              & \bigoplus\limits_{\sigma\in\Sigma}\rmH^{n}_{\{\sigma\}}(\rmX_{\overline{s}}, \rmR\Psi(\VV))\arrow[l]     &\mathrm{H}^{n-1}_{\rmc}(\rmX_{\overline{s}}, \VV)\arrow[l, "\beta"]       
\end{tikzcd}
\end{equation*}

Here the map in lowerline is the composite 
\begin{equation*}
\bigoplus_{\sigma\in\Sigma}\rmH^{n}_{\{\sigma\}}(\rmX_{\overline{s}}, \rmR\Psi(\VV))\rightarrow \rmH^{n}_{\rmc}(\rmX_{\overline{s}}, \VV)\rightarrow  \rmH^{n}_{\rmc}(\rmX_{\overline{s}}, \rmR\Psi(\VV))
\end{equation*} 
where the first map is the Gysin map under the identification 
\begin{equation*}
\bigoplus_{\sigma\in\Sigma}\rmH^{n}_{\{\sigma\}}(\rmX_{\overline{s}}, \rmR\Psi(\VV))=\bigoplus_{\sigma\in\Sigma}\rmH^{n}_{\{\sigma\}}(\rmX_{\overline{s}},\VV)
\end{equation*}
and the second map is the specialization map. One has also a Frobenius equivariant version
\begin{equation*}
\begin{tikzcd}
\mathrm{H}^{n}(\rmX_{\overline{\rms}}, \rmR\Psi(\VV)(1)) \arrow[r] \arrow[d, "{\rmN}"] & \bigoplus\limits_{\sigma\in\Sigma}\rmR^{n}\Phi_{\sigma}(\VV(1)) \arrow[d, "\rmN_{\Sigma}"] \arrow[r, "\alpha"]  &\mathrm{H}^{n+1}(\rmX_{\overline{s}}, \VV)) \\
\mathrm{H}^{n}_{\rmc}(\rmX_{\overline{\rms}}, \rmR\Psi(\VV))           &\bigoplus\limits_{\sigma\in\Sigma}\rmH^{n}_{\{\sigma\}}(\rmX_{\overline{s}}, \rmR\Psi(\VV))\arrow[l] &\mathrm{H}^{n-1}_{\rmc}(\rmX_{\overline{s}}, \VV)\arrow[l, "\beta"]       
\end{tikzcd}
\end{equation*}
where $\rmN$ is the monodromy operator and $\rmN_{\Sigma}=\oplus_{\sigma}\rmN_{\sigma}$ is the sum of the local monodromy operator $\rmN_{\sigma}: \rmR^{n}\Phi_{\sigma}(\VV(1))\rightarrow \rmH^{n}_{\{\sigma\}}(\rmX_{\overline{s}}, \rmR\Psi(\VV))$.  We will loosely refer to this diagram as the Picard--Lefschetz formula for $\rmH^{n}_{\rmc}(\rmX_{\overline{s}}, \rmR\Psi(\VV))$. Note that the monodromy filtration of  $\mathrm{H}^{n}_{\rmc}(\rmX_{\overline{s}}, \rmR\Psi(\VV))$ is then determined by the above Picard--Lefschetz formula
\begin{equation*}
0\subset_{\mathrm{Gr}_{-1}}\mathrm{F}_{-1}\mathrm{H}^{n}_{\rmc}(\rmX_{\overline{s}}, \rmR\Psi(\VV)) \subset_{\mathrm{Gr}_{0}} \mathrm{F}_{0}\mathrm{H}^{n}_{\rmc}(\rmX_{\overline{s}}, \rmR\Psi(\VV)) \subset_{\mathrm{Gr}_{1}} \mathrm{F}_{1}\mathrm{H}^{n}_{\rmc}(\rmX_{\overline{s}}, \rmR\Psi(\VV)).
\end{equation*}
The successive quotient of the filtration is given by
\begin{equation}\label{mono-fil}
\begin{aligned}
&\mathrm{Gr}_{-1}=\mathrm{Gr}_{-1}\mathrm{H}^{n}_{\rmc}(\rmX_{\overline{s}}, \rmR\Psi(\VV))= \Coker(\beta)\\
&\mathrm{Gr}_{0}=\mathrm{Gr}_{0}\mathrm{H}^{n}_{\rmc}(\rmX_{\overline{s}}, \rmR\Psi(\VV))= \frac{\mathrm{H}^{n}_{\rmc}(\rmX_{\overline{s}}, \VV)}{\bigoplus_{\sigma\in\Sigma}\rmH^{n}_{\{\sigma\}}(\rmX_{\overline{s}}, \rmR\Psi(\VV))}\\
&\mathrm{Gr}_{1}=\mathrm{Gr}_{1}\mathrm{H}^{n}_{\rmc}(\rmX_{\overline{s}}, \rmR\Psi(\VV))=\Ker(\alpha)(-1).\\
\end{aligned}
\end{equation}
The monodromy map $\mathrm{N}$ factor through 
\begin{equation*}
\mathrm{H}^{n}(\rmX_{\overline{s}}, \rmR\Psi(\VV)(1))\twoheadrightarrow\mathrm{Gr}_{1}\xrightarrow{\mathrm{N}} \mathrm{Gr}_{-1} \hookrightarrow \mathrm{H}^{n}_{\rmc}(X_{\overline{s}}, \rmR\Psi(\VV)).
\end{equation*}

\subsection{Nearby cycles of automorphic \'etale sheaves} Suppose $f:\rmX\rightarrow S$ is proper. Then the proper base change theorem implies that we have an isomorphism 
\begin{equation*}
\rmH^{i}(\rmX_{\overline{\eta}}, \VV)\cong \rmH^{i}(\rmX_{\overline{s}}, \rmR\Psi(\VV)).
\end{equation*}
However if $f: \rmX\rightarrow S$ is not proper, then this does not always hold.
In the setting of Shimura varieties with good compactification, we do have such an isomorphism. In particular all the Shimura varieties we will consider in this article have good compactifications and belong to the case: 
\begin{itemize}
\item[] (Nm) A flat integral model defined by taking normalization of a characteristic $0$ PEL type moduli problem over a product of  good reduction integral models of smooth PEL type moduli problem. 
\end{itemize}
in the classification of \cite{LS18b} and \cite{LS18c}. The following theorem summarizes the results we need. 
\begin{theorem}[{\cite[Corollary 4.6]{LS18b}}]
Suppose $\rmX$ is a Shimura variety that is in the case of $(\mathrm{Nm})$ and let $\VV$ be an automorphic \'etale sheaf defined as in \cite[\S 3]{LS18b}. Then the canonical adjunction morphisms
\begin{equation*}
\rmH^{i}(\rmX_{\overline{\eta}}, \VV)\rightarrow \rmH^{i}(\rmX_{\overline{s}}, \rmR\Psi(\VV))
\end{equation*}
and
\begin{equation*}
\rmH^{i}_{\rmc}(\rmX_{\overline{s}}, \rmR\Psi(\VV))\rightarrow \rmH^{i}_{\rmc}(\rmX_{\overline{\eta}}, \VV)
\end{equation*}
are isomorphisms for all $i$. 
\end{theorem}

\section{Automorphic and Galois representations for $\GSp_{4}$}
\subsection{The group $\GSp_{4}$}
Let $\GSp_{4}$ be the symplectic similitude group defined by the set of matrices $g$ in $\GL_{4}$ that satisfy $g^{t}\rmJ g=c(g)\rmJ$ for 
\begin{equation*}
\rmJ=\begin{pmatrix}
 & & &1\\
 & & 1&\\
 & -1& &\\
-1 & & &\\
\end{pmatrix}
\end{equation*}
and some $c(g)\in \mathbb{G}_{m}$.  We define the map $c: \GSp_{4}\rightarrow \mathbb{G}_{m}$ by sending $g\in \GSp_{4}$ to $c(g)$ and refer to it as the similitude map. The kernel of this map is by definition the symplectic group $\Sp(4)$.  There are two conjugacy classes of maximal parabolic subgroups of $\GSp_{4}$
given by the \emph{Siegel parabolic subgroup} $\rmP$, whose Levi factor is 
\begin{equation*}
\rmM_{\rmP}=\{\begin{pmatrix}A&\\&u A^{\prime}\\\end{pmatrix}: u\in \GL_{1}, A\in \GL_{2}\}
\end{equation*}
where $A^{\prime}=\begin{pmatrix}&1\\1&\\ \end{pmatrix}(A^{-1})^{t}\begin{pmatrix}&1\\1&\\ \end{pmatrix}$. It is clear that $\rmM_{\rmP}\cong \GL_{1}\times \GL_{2}$. And we have the \emph{Klingen parabolic subgroup} $\rmQ$, whose Levi factor is
\begin{equation*}
\rmM_{\rmQ}=\{\begin{pmatrix}u&\\&A&\\ &&u^{-1}\det(A)\\ \end{pmatrix}: u\in \GL_{1}, A\in \GL_{2}\}.
\end{equation*}
It is clear that $\rmM_{\rmQ}\cong \GL_{1}\times \GL_{2}$.

Let $\rmT$ be the diagonal torus in $\GSp_{4}$ and $\rmX^{*}(\rmT)$ be its character group which we identify with subset $\{(a,b;c)\in\ZZ^{3}: a+b\equiv c \mod 2\}$ of $\ZZ^{3}$ by associating a triple $(a,b;c)$ the character 
\begin{equation}\begin{pmatrix}t_{1}&&&\\&t_{2}&&&\\&&vt^{-1}_{2}&\\&&&vt^{-1}_{1}\\\end{pmatrix}\rightarrow t^{a}_{1}t^{b}_{2}v^{\frac{(c-a-b)}{2}}.\end{equation}
Let $\rmB$ be the Borel subgroup of upper triangular matrices in $\GSp_{4}$. Then the set of dominant weights $\rmX^{*}(\rmT)^{+}$ with respect to $\rmB$ is given explicitly by $\{(a,b;c)\in \rmX^{*}(\rmT): a\geq b\geq 0\}.$

\subsection{Parahoric subgroups of $\GSp_{4}(\QQ_{p})$} The affine Dynkin diagram of type $\tilde{\mathrm{C}}_{2}$ is given by 
\begin{displaymath}
 \xymatrix{\underset{0}\bullet \ar@2{->}[r] &\underset{1}\bullet &\underset{2}\bullet\ar@2{->}[l]}
\end{displaymath}
and correspondingly we have the parahoric subgroups corresponding to each non-empty subsets of $\{0,1,2\}$.
\begin{itemize}
\item The \emph{Iwahori subgroup} $\rmK_{\{0,1,2\}}$: the Iwahori subgroup will also be denoted by $\Iw(p)$. It contains the pro-$p$ Iwahori subgroup which we will denote by $\Iw_{1}(p)$;
\item The \emph{Siegel parahoric subgroup} $\rmK_{\{0,2\}}$: this is obtained from the Iwahori subgroup by adding the affine root group corresponding to $1$; Its reduction is the Siegel parabolic $\rmP(\FF_{p})$ of $\GSp_{4}(\FF_{p})$. The Siegel parhoric subgroup will also be denoted by $\mathrm{Sie}(p)$;
\item The \emph{Kilngen parahoric subgroup} $\rmK_{\{0,1\}}$: this is obtained from the Iwahori subgroup by adding the affine root group corresponding to $2$; Its reduction is the Klingen parabolic $\rmQ(\FF_{p})$ of $\GSp_{4}(\FF_{p})$. The Klingen parahoric subgroup will also be denoted by $\Kl(p)$; 
\item The \emph{paramodular subgroup} $\rmK_{\{1\}}$: this is obtained from the Iwahori subgroup by adding the affine root groups corresponding to $0$ and $2$. The paramodular subgroup will also be denoted by $\Pa(p)$;
\item The \emph{hyperspecial subgroup} $\rmK_{\{0\}}$ and $\rmK_{\{2\}}$: these are obtained from the Iwahori subgroup by adding the affine root groups corresponding to $1, 2$ or $1, 0$. Note that they are conjugate to each other and if we need to identify them we will denote them by $\rmH$.
\end{itemize}

Note that the hyperspecial subgroup and the paramodular subgroup are the maxiamal parahoric subgroups in $\GSp_{4}(\QQ_{p})$, therefore there is no map between paramodular Siegel threefold to the Siegel threefold of hyperspecial level. This is one of the difference between our case and the case of modular curves or Shimura curves. 

\subsection{Some inner forms of $\GSp_{4}$}\label{inner_form} Let $\rmB$ be a quaternion algebra over $\QQ$. If $\rmB$ split at $\infty$, then we call $\rmB$ \emph{indefinite} and otherwise we call $\rmB$ \emph{definite}. For an element $b\in \rmB$, we denote by $\overline{b}$ the image of $b$ under the main involution of $\rmB$. We choose an element $\tau\in \rmB$ such that $\overline{\tau}=-\tau$ if $B$ is indefinite and $\tau=1$ if $\rmB$ is definite. We define a new involution $*$ on $\rmB$ by putting $b^{*}=\tau \overline{b}\tau^{-1}$. Then consider the form $(\cdot,\cdot)$ on $\rmV=\rmB\oplus \rmB$ defined by 
\begin{equation}\label{qua-herm}
(x, y)= \mathrm{tr}^{0}(\tau^{-1}(x_{1}y^{*}_{1}+x_{2}y^{*}_{2}))
\end{equation}
where $x=(x_{1}, x_{2})$ and $y=(y_{1}, y_{2})$ are elements in $\rmV$. 
Then we define the \emph{quaternionic unitary similitude group of degree $2$} by 
\begin{equation}\label{Global-quauni}
\GU_{2}(\rmB)(\QQ)=\{g\in \GL_{2}(\rmB)(\QQ): (gx, gy)=c(g)(x, y), c(g)\in \QQ^{\times}\}. 
\end{equation}
 We write $\rmD$ the quaternion division algebra over $\QQ_{p}$. If $\rmB_{p}=\rmD$ at a place $p$, then $\GU_{2}(\rmB)(\QQ_{p})=\GU_{2}(\rmD)$ where $\GU_{2}(\rmD)$ is defined similarly as in \eqref{Global-quauni} for $\rmV=\rmD\oplus \rmD$. If  If $\rmB_{p}=\mathrm{M}_{2}(\QQ_{p})$,  then $\GU_{2}(\rmB)(\QQ_{p})=\GSp_{4}(\QQ_{p}).$ Since $\rmD$ splits over $\QQ_{p^{2}}$, we also have an identification 
\begin{equation*}
\GU_{2}(\rmD)(\QQ_{p^{2}})\cong \GSp_{4}(\QQ_{p^{2}}).
\end{equation*}
There are three kinds of parahoric subgroups of $\GU_{2}(\rmD)$. This can be explained using the affine Dynkin diagram of $\GU_{2}(\rmD)$. The affine Dynkin diagram is still of type $\tilde{\mathrm{C}}_{2}$ as shown below 
\begin{center}
\begin{displaymath}
 \xymatrix{\underset{0}\bullet \ar@/^1pc/[rr]\ar@2{->}[r] &\underset{1}\bullet&\underset{2}\bullet\ar@2{->}[l] \ar@/_1pc/[ll]}
\end{displaymath}
\end{center}
but with Frobenius acting non-trivially on the diagram by switching the nodes $0$ and $2$ while fixing the node $1$. We have the following correspondence between the parahoric subgroups of $\GU_{2}(\rmD)$ and parahoric subgroups of $\GSp_{4}(\QQ_{p})$ .
\begin{itemize}
\item The \emph{Iwahori subgroup} $\Iw_{\rmD}(p)$ which corresponds to the Iwahori subgroup $\Iw(p)$ in $\GSp_{4}$ when base change to $\ZZ_{p^{2}}$;
\item The \emph{Siegel parahoric subgroup} $\mathrm{Sie}_{\rmD}=\rmK^{\prime}_{\{02\}}$ which corresponds to the Siegel parahoric $\rmK_{\{02\}}$ in $\GSp_{4}$ when base changed to $\ZZ_{p^{2}}$;
\item The \emph{paramodular subgroup} $\Pa_{\rmD}(p)=\rmK^{\prime}_{\{1\}}$ which corresponds to the paramodular parahoric $\Pa(p)=\rmK_{\{1\}}$ in $\GSp_{4}$ when base changed to $\ZZ_{p^{2}}$.
\end{itemize}

\subsection{Unramified principal series representation} Let $p$ be  a prime number. We let $\mathcal{H}_{p}$ be the spherical Hecke algebra over $\mathbb{Z}$. This is a commutative algebra isomorphic to $\mathbb{Z}[\rmT_{p,0}, \rmT^{-1}_{p,0}, \rmT_{p,1}, \rmT_{p,2}]$ where 
\begin{equation}
\begin{split}
&\rmT_{p,0}=\mathrm{char}(\GSp_{4}(\ZZ_{p})\begin{pmatrix}p&&&\\&p&&\\&&p&\\&&&p\\\end{pmatrix}\GSp_{4}(\ZZ_{p})), \\
&\rmT_{p,1}=\mathrm{char}(\GSp_{4}(\ZZ_{p})\begin{pmatrix}p^{2}&&&\\& p&&\\&&1&\\&&&1\\\end{pmatrix}\GSp_{4}(\ZZ_{p})), \\
&\rmT_{p,2}=\mathrm{char}(\GSp_{4}(\ZZ_{p})\begin{pmatrix}p &&&\\&p&&\\&&1&\\&&& 1\\\end{pmatrix}\GSp_{4}(\ZZ_{p}))\\
\end{split}
\end{equation}
and where $\mathrm{char}(\cdot)$ is the characteristic function.
The Hecke polynomial is by definition given by
\begin{equation*}
\rmQ_{p}(\rmX)=1- \rmT_{p,2}\rmX+ p(\rmT_{p,1}+(p^{2}+1)\rmT_{p,0})\rmX^{2}-p^{3}\rmT_{p,2}\rmT_{p,0}\rmX^{3}+p^{6}\rmT^{2}_{p,0}\rmX^{6}.
\end{equation*}
To any $\pi_{p}$ irreducible unramified admissible representation of $\GSp_{4}(\QQ_{p})$, we associate a character 
\begin{equation}\label{local-Hecke-map}
\phi_{\pi_{p}}: \mathcal{H}_{p}\rightarrow \End(\pi_{p}^{\GSp_{4}(\ZZ_{p})})=\CC
\end{equation}
and a Langlands parameter 
\begin{equation*}
\begin{pmatrix}
\alpha&&&\\
&\beta&&&\\
&&\gamma&&\\
&&&\delta\\
\end{pmatrix}
\end{equation*}
considered as an element in $\GSp_{4}(\CC)={^{L}\GSp_{4}}$. Then we have the following identity
\begin{equation}\label{Hecke-poly}
\phi_{\pi_{p}}(\rmQ_{p}(\rmX))=(1-p^{2/3}\alpha \rmX)(1-p^{2/3}\beta \rmX)(1-p^{2/3}\gamma \rmX)(1-p^{2/3}\delta \rmX).
\end{equation}

We will freely use the classification of irreducible constitutes of parabolic induced representation for $\GSp_{4}(\QQ_{p})$ by \cite{ST} and the table of Schmidt \cite{Schmi05}.  Let $\chi_{1}, \chi_{2}, \sigma$ be characters of $\QQ^{\times}_{p}$, we consider the principal series representation of $\GSp_{4}(\QQ_{p})$ given by
\begin{equation*}
\chi_{1}\times\chi_{2}\rtimes \sigma:= \mathrm{Ind}^{\GSp_{4}(\QQ_{p})}_{\mathrm{\rmB}(\QQ_{p})}\chi_{1}\otimes\chi_{2}\otimes \sigma.
\end{equation*}
This is defined by the normalized induction from the Borel subgroup for the character given by
\begin{equation*}
\chi_{1}\otimes\chi_{2}\rtimes\sigma: \begin{pmatrix}a&\ast&\ast&\ast\\ &b&\ast&\ast\\  &&cb^{-1}&\ast\\  &&&ca^{-1}\\ \end{pmatrix}\mapsto \chi_{1}(a)\chi_{2}(b)\sigma(c).
\end{equation*}
This is irreducible if and only if none of the characters $\chi_{1}, \chi_{2}, \chi_{1}\chi^{\pm1}_{2}$ is equal to $\vert\cdot\vert^{\pm1}$ where $\vert\cdot\vert$ is the normalized valuation such that $\vert p\vert=p^{-1}$. This is the type $\rmI$ representation in the classification of \cite{ST} and in the table of Schmidt \cite{Schmi05}. If the characters $\chi_{1}, \chi_{2}, \sigma$ are unramified, then we say $\pi_{p}$ is an \emph{unramified principal series}. It is natural to consider the twist $\pi_{p}\otimes \vert\cdot\vert^{-3/2}$ and we call the parameters 
\begin{equation*}
[\alpha_{p}, \beta_{p}, p^{3}\beta^{-1}_{p}, p^{3}\alpha^{-1}_{p}]=[\vert\cdot\vert^{-3/2}\chi_{1}\chi_{2}\sigma(p), \vert\cdot\vert^{-3/2}\chi_{1}\sigma(p), \vert\cdot\vert^{-3/2}\chi_{2}\sigma(p), \vert\cdot\vert^{-3/2}\sigma(p)] 
\end{equation*}
the \emph{Hecke parameters} of $\pi_{p}=\chi_{1}\times\chi_{2}\rtimes \sigma$. Note this notion is only well-defined modulo the Weyl group action.

\subsection{Representations with paramodular fixed vector} Let $\pi_{p}$ be an irreducible smooth representation of $\GSp_{4}(\QQ_{p})$ such that $\pi^{\Pa(p)}_{p}\neq0$. The table of \cite{Schmi05} shows that there are five types of representations (IIa, IVc, Vb, Vc, VIc) and only IIa is generic. Let $|\cdot|: \QQ^{\times}_{p}\rightarrow\overline{\QQ}_{l}$ be the normalized absolute value such that $|p|=\frac{1}{p}$ and let $\sigma$ be an unramified character $\QQ^{\times}_{p}$ and $\chi$ be character of $\QQ^{\times}_{p}$ such that $\chi\neq |\cdot|^{\pm}, |\cdot|^{\pm 3/2}$. Then the type IIa irreducible representation $\chi\mathrm{St}_{\GL_{2}}\rtimes \sigma$ corresponds to the Weil--Deligne representation given by 
\begin{equation}\label{weil-del}
\begin{pmatrix}
\chi^{2}\sigma&&&\\
&|\cdot|^{1/2}\chi\sigma&&\\
&&|\cdot|^{-1/2}\chi\sigma&\\
&&&\sigma\\
\end{pmatrix} 
\end{equation}
 with monodromy operator of rank one given by 
 \begin{equation}\label{mono}
\rmN= 
\begin{pmatrix}
0&&&\\
&0&1&\\
&&0&\\
&&&0.\\
\end{pmatrix}
 \end{equation}
under the local Langlands correspondence for $\GSp_{4}$ established in \cite{GT11}.  In other words, the $L$-parameter is given by the $\chi^{2}\sigma\oplus |\cdot|^{-1/2}\chi\sigma\otimes \mathrm{Sp}_{2}\oplus \sigma$. We will refer to $ |\cdot|^{-1/2}\chi\sigma\otimes \mathrm{Sp}_{2}$ as the twisted Steinberg part of the Weil--Deligne representation. 

Suppose that $\pi_{p}=\chi\mathrm{St}_{\GL_{2}}\rtimes \sigma$. It is natural to consider the twist $\pi_{p}\otimes \vert\cdot\vert^{-3/2}$ and we call the parameters defined by
\begin{equation*}
[\alpha_{p}, \beta_{p}, p^{3}\beta^{-1}_{p}, p^{3}\alpha^{-1}_{p}]=[\vert\cdot\vert^{-3/2}\chi^{2}\sigma(p), \vert\cdot\vert^{-1}\chi\sigma(p), \vert\cdot\vert^{-2}\chi\sigma(p), \vert\cdot\vert^{-3/2}\sigma(p)]
\end{equation*}
the \emph{Hecke parameters} of $\pi_{p}=\chi\mathrm{St}_{\GL_{2}}\rtimes \sigma$. The central character of $\pi_{p}$ is given by $(\chi\sigma)^{2}$ which we assume is trivial. Thus $\chi\sigma(p)=\rmu\in\{\pm 1\}$ and we have 
\begin{equation*}
[\alpha_{p}, \beta_{p}, p^{3}\beta^{-1}_{p}, p^{3}\alpha^{-1}_{p}]=[\alpha_{p}, \rmu p, \rmu p^{2}, p^{3}\alpha^{-1}_{p}].
\end{equation*}

We define the \emph{Atkin-Lehner operator} using the matrix
\begin{equation}\label{AL}
\rmu_{p}=
\begin{pmatrix}
&&1&\\
&&&-1\\
p&&&\\
&-p&&\\
\end{pmatrix}.
\end{equation}
This element normalizes the paramodular subgroup $\Pa(p)$ and thus will act on the space of $\Pa(p)$ fixed vectors of $\pi_{p}$. 
\begin{lemma}\label{Atkin-Lehner}
Let $\pi_{p}=\chi\mathrm{St}_{\GL_{2}}\rtimes \sigma$ be a representation of type $\mathrm{IIa}$ with $\chi$ and $\sigma$ as above. Then $\pi^{\Pa(p)}_{p}$ is one dimensional on which $u_{p}$ acts by the scalar $\chi\sigma(p)$. 
\end{lemma}
\begin{proof}
This is \cite[Lemma 2.2.3]{VH19}. 
\end{proof}

\subsection{Cohomology of Siegel modular threefold} In this subsection, we define the Siegel modular threefold and various cohomology theories related to automorphic forms. Let $h: \mathrm{Res}_{\mathbb{C}/\mathbb{R}}\mathbb{G}_{m}=\mathbb{C}^{\times}\rightarrow \GSp_{4}(\RR)$ be the map sending $x+iy$ to 
\begin{equation*}
\begin{pmatrix}x\rmI_{2}&y\rmS\\-y\rmS&x\rmI_{2}\end{pmatrix}
\end{equation*}
where $\rmI_{2}$ is the identity matrix of size $2$ and $\rmS=\mathrm{id}$. Let $\rmK^{h}$ be the centralizer of $h$ in $\GSp_{4}(\RR)$. Then $\rmK^{h}=\rmK_{\infty}\RR^{\times}$ where $\rmK_{\infty}$ is the maximal compact subgroup of $\GSp_{4}(\RR)$. Let $\rmU\subset \GSp_{4}(\mathbb{A}^{(\infty)})$ be an open compact subgroup and we obtain the Siegel modular threefold with level $\rmU$ denoted by $\rmX_{\rmU}$. This is a quasi-projective variety whose $\mathbb{C}$-points are given by
\begin{equation*}
\rmX_{\rmU}(\mathbb{C})=\GSp_{4}(\QQ)\backslash(\GSp_{4}(\RR)/\rmK^{h}\times \GSp_{4}(\mathbb{A}^{(\infty)}))/\rmU. 
\end{equation*}

Let $\mu=(a,b;c)\in \rmX^{*}(\rmT)$ and $\rmV_{\mu}$ be the corresponding irreducible representation of highest weight $\mu$. Let $\mathrm{V}_{\mu}(\calO)$ be a $\GSp_{4}$-stable $\calO$-lattice contained $\rmV_{\mu}$. We review the construction of an \'etale local system $\mathbb{V}_{\mu}$ on $\rmX_{\rmU}$.  For $m$ sufficiently large,  let $\rmU(m)$ be a normal subgroup in $\rmU$ that acts trivially on $\rmV_{\mu}(\calO)/\varpi^{m}$.  Then $\mathbb{V}_{\mu}$ is defined to be 
\begin{equation}\label{loc-sys}
\rmX_{\rmU(m)\rmU^{(l)}}\times^{\rmU/\rmU(m)} \rmV_{\mu}(\calO)/\varpi^{m}.
\end{equation}
We denote its associated $E_{\lambda}$-local system by $\mathbb{V}_{\mu, E_{\lambda}}$,  the $k$-local system by $\mathbb{V}_{\mu, k}$ and the $\CC$-local system by $\mathbb{V}_{\mu, \CC}$. In the following, we will denote by $\VV_{\mu, \ast}$ simply by $\VV_{\ast}$ if $\mu$ is fixed and where $\ast=E_{\lambda}, k, \emptyset$. One is usually interested in the following types of cohomology theories of $\rmX_{\rmU}$. 
\begin{itemize}
\item The \emph{intersection cohomology} $\mathrm{IH}^{*}(\rmX_{\rmU}(\mathbb{C}), \mathbb{V}_{\mu})$ with respect to the minimal compactification of $\rmX_{\rmU}$;
\item  The \emph{$L^{2}$-cohomology} $\mathrm{H}^{*}_{(2)}(\rmX_{\rmU}(\mathbb{C}), \mathbb{V}_{\mu, \CC})$ which can be computed by the Matsushima formula 
\begin{equation*}
\mathrm{H}^{*}_{(2)}(\rmX_{\rmU}(\mathbb{C}), \mathbb{V}_{\mu,\CC})=\bigoplus_{\pi}m(\pi)\pi^{\rmU}_{f}\otimes \mathrm{H}^{*}(\mathfrak{g}, \rmK_{\infty}, \mathbb{V}_{\mu, \mathbb{C}} \otimes \pi_{\infty})
\end{equation*}
 where $\pi$ runs through all the automorphic representations $\pi=\pi_{f}\otimes \pi_{\infty}$ that occurs in the discrete spectrum of $L^{2}(\GSp_{4}(\mathbb{Q})\backslash\GSp_{4}(\mathbb{A}))$ with multiplicity $m(\pi)$. There is a natural direct summand $\mathrm{H}^{*}_{\mathrm{cusp}}(\rmX_{\rmU}(\mathbb{C}),\mathbb{V}_{\mu, \CC} )$ of cuspidal cohomology which is the direct sum over cuspidal automorphic representations $\pi$. 
\item The \emph{interior cohomology} defined by
\begin{equation*}
 \mathrm{H}^{*}_{\mathrm{int}}(\rmX_{\rmU}(\mathbb{C}), \mathbb{V}_{\mu}):=\mathrm{Im}(\mathrm{H}^{*}_{\rmc}(\rmX_{\rmU}(\mathbb{C}), \mathbb{V}_{\mu})\rightarrow \mathrm{H}^{*}(\rmX_{\rmU}(\mathbb{C}), \mathbb{V}_{\mu})).
 \end{equation*}
\end{itemize}
Moreover we have an isomorphism $\mathrm{IH}^{*}(\rmX_{\rmU}(\mathbb{C}), \mathbb{V}_{\mu, \CC})= \mathrm{H}^{*}_{(2)}(\rmX_{\rmU}(\mathbb{C}), \mathbb{V}_{\mu,\CC})$.

\subsection{Galois representations attached to automorphic representations} 
Let $\pi$ be a cuspidal automorphic representation of $\GSp_{4}$ over $\QQ$. Throughout this article, we will assume that $\pi$ is of \emph{general type} in the sense of Arthur \cite{Arthur-GSp}. In particular $\pi$ is not an endoscopic or a (weakly) CAP representation. By twisting a power of $\vert\cdot\vert$, we will assume that $\pi$ is unitary with trivial central character. We will say $\pi$ is of \emph{weight} $(k_{1}, k_{2})$
if the its infinite component $\pi_{\infty}$ has Harish-Chandra parameter $(k_{1}-1, k_{2}-2; 0)$.  Given a cuspidal automorphic representation $\pi$ of $\GSp_{4}$ over $\QQ$, we let $\Sigma=\Sigma_{\pi}$ be the product of primes at which that $\pi$ is ramified. Then we define the unramified Hecke algebra to be the tensor product 
\begin{equation*}
\mathbb{T}^{\Sigma}= \bigotimes_{p\nmid \Sigma}\mathcal{H}_{p}. 
\end{equation*}
We will denote $\mathbb{T}^{\Sigma}$ by $\TT$ to simplify the notations. The morphism $\phi_{\pi_{v}}$ in \eqref{local-Hecke-map} patch to a morphism
$\phi_{\pi}: \TT\rightarrow \CC$. There is a number field $E$ such that $\phi_{\pi}: \TT\rightarrow \CC$ lands in $E$ called the \emph{coefficient field} of $\pi$.

We review the construction of the Galois representation attached to $\pi$ of weight $(k_{1}, k_{2})$ as above. We set $\mu=(k_{1}-3, k_{2}-3; 0)$ and consider $\VV=\VV_{\mu}$. We define 
$\mathrm{H}^{i}_{\mathrm{int}}(\rmX\otimes{\overline{\QQ}}, \mathbb{V}_{E_{\lambda}})=\varinjlim_{\rmU}\mathrm{H}^{i}_{\mathrm{int}}(\rmX_{\rmU}\otimes\overline{\QQ}, \mathbb{V}_{E_{\lambda}})$. In this case, it is known that $\pi$ is concentrated  in degree $3$ in $\mathrm{H}^{i}_{\mathrm{int}}(\rmX\otimes{\overline{\QQ}}, \mathbb{V}_{E_{\lambda}})$. Then we define the Galois module $\mathrm{W}_{\pi, \lambda}$ by
\begin{equation}\label{Gal-mod}
\mathrm{W}_{\pi, \lambda}= \Hom_{\GSp_{4}(\mathbb{A}^{(\infty)})}(\pi_{f}, \mathrm{H}^{3}_{\mathrm{int}}(\rmX\otimes{\overline{\QQ}}, \mathbb{V}_{E_{\lambda}})).
\end{equation}
In fact by \cite[Theorem 1.1]{Weis09}, we have an isomorphism
\begin{equation}\label{coh-same}
\mathrm{H}^{3}_{\mathrm{int}}(\rmX\otimes{\overline{\QQ}}, \mathbb{V}_{E_{\lambda}})[\pi_{f}]\cong \mathrm{H}_{\rmc}^{3}(\rmX\otimes{\overline{\QQ}}, \mathbb{V}_{E_{\lambda}})[\pi_{f}].
\end{equation}
This will allow us to pass from the cohomology with compact support  $\mathrm{H}_{\rmc}^{3}(\rmX\otimes{\overline{\QQ}}, \mathbb{V}_{E_{\lambda}})$ to the interior cohomology $\mathrm{H}_{!}^{3}(\rmX\otimes{\overline{\QQ}}, \mathbb{V}_{E_{\lambda}})$ after localizing at a non-Eisenstein maximal ideal in the Hecke algebra. The following theorem summarizes the properties of $\mathrm{W}_{\pi, \lambda}$ and the proof of it can be found in \cite{Taylor93}, \cite{Laum05}, \cite{Weis05}, \cite{Sor10} and \cite{Mo14}.

\begin{theorem}\label{Gal-GSp4}
Let $\pi$ be as above. Then the Galois module $\mathrm{W}_{\pi, \lambda}$ gives rise to a Galois representation
\begin{equation*}
\rho_{\pi, \iota_{l}}: \rmG_{\QQ}\rightarrow \GSp_{4}(\overline{\mathbb{Q}}_{l})
\end{equation*}
satisfying the following properties.
\begin{enumerate}
\item The representation $\rho_{\pi, \iota_{l}}$ is unramified at a prime $p$ for $p\neq l$ and $\pi_{p}$ is unramified. Moreover we have $$\det(1-\rho_{\pi, \iota_{l}}(\mathrm{Frob}_{p})\rmX)=\phi_{\pi}(\rmQ_{p}(\rmX)).$$
\item At all places $p\neq l$ we have the local-global compatibility of Langlands correspondence, this means
\begin{equation*}
\mathrm{WD}(\rho_{\pi, \iota_{l}}\vert _{\rmG_{\QQ_{p}}})^{\mathrm{ss}}\cong \mathrm{rec}_{\mathrm{GT}}(\pi_{p}\otimes \vert\cdot\vert^{-3/2})
\end{equation*}
where $\mathrm{rec}_{\mathrm{GT}}$ is the local Langlands reciprocity map of \cite{GT11} and $\mathrm{ss}$ denotes the Frobenius semi-simplification of the Weil--Deligne representation. 
\item If $\rho_{\pi, \iota_{l}}$ is irreducible, then for all finite place $v$, $\rho_{\pi, \iota_{l}}\vert_{\rmG_{\QQ_{v}}}$ is pure. 
\end{enumerate}
\end{theorem}

We say $E\subset \CC$ is a strong coefficient field for the automorphic representation $\pi$ if for every prime $\lambda$ of $E$ there exists a continuous homomorphism 
\begin{equation*}
\rho_{\pi,\lambda}: \rmG_{\QQ}\rightarrow \GSp_{4}(E_{\lambda})
\end{equation*}
such that $\rho_{\pi,\lambda}\otimes \overline{\QQ}_{l}$ is conjugate to $\rho_{\pi,\iota_{l}}$ for all $\iota_{l}:\CC\xrightarrow{\sim}\overline{\QQ}_{l}$.
Let $E$ be a strong coefficient field for $\pi$, we know that $\phi_{\pi}: \TT^{\Sigma}\rightarrow E$ in lands in $\calO_{E}$. Therefore we obtain a morphism $\phi_{\pi,\lambda}: \TT\rightarrow \calO_{\lambda}$
called the $\lambda$-adic avatar of $\phi_{\pi}$. We define the maximal ideal ${\frakm}$ by
\begin{equation}\label{maximal-ideal}
\frakm=\ker(\TT\xrightarrow{\phi_{\pi, \lambda}}\calO_{\lambda}\rightarrow \calO_{\lambda}/\lambda).
\end{equation}
We will say $\frakm$ is the maximal ideal associated to the residual Galois representation $\overline{\rho}_{\pi, \lambda}$.

\section{Siegel modular threefold with parahoric level structures}
\subsection{Siegel modular threefold with parahoric level} Let $\rmU^{p}$ be a fixed prime to $p$ open compact subgroup of $\GSp_{4}(\mathbb{A}^{(\infty)})$ which we assume is sufficiently small. Let $\rmH=\GSp_{4}(\ZZ_{p})$ be the hyperspecial subrgoup of $\GSp_{4}(\QQ_{p})$ and recall that $\Pa(p)$ is the paramodular subgroup and $\mathrm{Kl}(p)$ is the Klingen parahoric subgroup.  In this section, we will omit $\rmU^{p}$ from the notation for the Siegel modular threefold $\rmX_{\rmU}$. This means that we will write $\rmX_{\rmH}=\rmX_{\rmH \rmU^{p}}$, $\rmX_{\Pa}(p)=\rmX_{\Pa(p)\rmU^{p}}$ and $\rmX_{\Kl}(p)=\rmX_{\Kl(p)\rmU^{p}}$. These Shimura varieties admit natural integral models over $\ZZ_{(p)}$ which are well studied in the literature and therefore we will only recall their moduli interpretations. 
\begin{itemize}
\item $\rmX_{\rmH}$ classifies triples of the form $(\sfA, \lambda, \alpha^{p})$ up to isomorphism where $\sfA$ is an abelian scheme of relative dimension $2$ over a test scheme $S$ over $\ZZ_{p}$ equipped with a principal polarization $\lambda$ and a $\rmU^{p}$-level structure $\alpha^{p}$.

\item $\rmX_{\Pa}(p)$ classifies triples of the form $(\sfA, \lambda, \alpha^{p})$ up to isomorphism where $\sfA$ is an abelian scheme of relative dimension $2$ over a test scheme $S$ and $\lambda$ is a polarization on $\sfA$ such that  $\ker(\lambda)$ has rank $2$ and $\alpha^{p}$ is a $\rmU^{p}$-level structure. In this article, we will also refer to $\rmX_{\Pa}(p)$ as the \emph{paramodular Siegel threefold}.

\item $\rmX_{\Kl}(p)$ is the moduli space of isomorphism classes of the triple $(\sfA, \mathsf{H}\subset \sfA[p],\lambda, \alpha^{p})$ over $\ZZ_{p}$ where $\sfA$ is a principally polarized abelian scheme over $\ZZ_{p}$ of relative dimension $2$, $\sfH\subset \sfA[p]$ is subgroup of order $p$ and $\alpha^{p}$ is a $\rmU^{p}$ level structure. 
\end{itemize}

\subsection{Supersingular locus at paramodular level}

Next we explain the structure of the supersingular locus and the singular locus of the special fiber of  $\rmX_{\Pa}(p)$. First we recall a result of Yu \cite{Yu11a}, let $\overline{\rmX}_{\Pa}(p)$ be the special fiber of $\rmX_{\Pa}(p)$ and we also denote its base change to $\overline{\FF}_{p}$ by the same notation. 

\begin{theorem}
The scheme $\rmX_{\Pa}(p)$ is regular with isolated quadratic singularities concentrated on its special fiber. The singular locus $\Sigma_{p}(\Pa(p))$ consists of those $(\sfA,\lambda, \alpha^{p})$ such that $\ker(\lambda)\subset \sfA[p]$ and $\sfA$ is superspecial. 
\end{theorem}
\begin{proof}
The first assertion is proved by a local model computation in \cite{Yu11a}. For the second assertion, the singular locus consists of those $(\sfA,\lambda, \mu)$ such that $\ker(\lambda)\cong \alpha_{p}\times\alpha_{p}$ by \cite[Theorem 4.7]{Yu11a}. Thus $\ker(\lambda)\subset \sfA[p]$ and moreover it implies that $\sfA$ is superspecial. 
\end{proof}

Next we recall the description of the supersingular locus  $\overline{\rmX}^{\ss}_{\Pa}(p)$ of  $\overline{\rmX}_{\Pa}(p)$ using the Bruhat--Tits stratification in \cite{Wanga}. We fix a $p$-divisible group $\mathbb{X}$ of dimension $2$ which is isoclinic of slope $1/2$ equipped with a polarization of height $2$. We denote by $\rmN$ its associated isocrystal. The polarization induces an alternating form $(\cdot, \cdot)$ on $\rmN$. Let $b\in \rmB(\GSp_{4})$ be the element corresponds to $\rmN$ in the Kottwitz set $\rmB(\GSp_{4})$ of $\GSp_{4}$. Let $\rmJ_{b}$ be the group functor over $\QQ_{p}$ defined by assigning each $\QQ_{p}$-algebra $R$ the group
\begin{equation}\label{Jb}
\rmJ_{b}(R)=\{g\in \GSp_{4}(R\otimes_{\QQ_{p}} \rmK_{0}): g(b\sigma)=(b\sigma)g\}.
\end{equation}
It is well known that in this case $\rmJ_{b}\cong \GU_{2}(\rmD)$. In fact if we denote by $\rmC=\rmN^{\tau=1}$ with $\tau=\rmF^{-1}\rmV$, then $\rmC$ can be viewed as an $\rmD$-module equipped with an alternating form $(\cdot, \cdot)$ which is the restriction of $(\cdot, \cdot)$ on $\rmN$. Then we have $\rmJ_{b}\cong \GU_{2}(\rmC)$. 

We  consider the Rapoport--Zink space $\mathcal{M}_{\Pa}$ for $\GSp_{4}(\QQ_{p})$ with paramodular level structure. We denote by $(\mathrm{Nilp})$ the category of $\rmW_{0}$-schemes $S$ on which $p$ is locally nilpotent.  This is the set valued functor on $(\mathrm{Nilp})$ that classifies the data $(\sfX, \lambda_{\sfX}, \rho_{\sfX})$ where 
\begin{itemize}
\item $\mathsf{X}$ is a two dimensional $p$-divisible group over $S$ of height $4$;
\item $\lambda_{\mathsf{X}}: \mathsf{X}\rightarrow \mathsf{X}^{\vee}$ is a polarization of height $2$;
\item $\rho_{\mathsf{X}}: \mathbb{\sfX}\times_{\overline{\FF}_{p}}\overline{S}\rightarrow \sfX\times_{S}\overline{S}$ is a quasi-isogeny.
\end{itemize}
Here $\overline{S}$ is the special fiber of $S$ at $p$. This functor is representable by a formal scheme $\mathcal{M}_{\Pa}$ over $\Spf(\rmW_{0})$ and $\rmJ_{b}(\QQ_{p})$ acts on  $\mathcal{M}_{\Pa}$ naturally. The formal scheme $\mathcal{M}_{\Pa}$ can be decomposed into connected components 
\begin{equation}
\mathcal{M}_{\Pa}=\bigsqcup_{i\in\ZZ}\mathcal{M}^{(i)}_{\Pa}
\end{equation}
where each $\mathcal{M}^{(i)}_{\Pa}$ classifies those $(\sfX, \lambda_{\sfX}, \rho_{\sfX})$ such that $\rho_{\rmX}$ has height $i$. We denote by $\rmM_{\Pa}$ the underlying reduced subscheme of $\mathcal{M}^{(0)}_{\Pa}$. 

A \emph{vertex lattice} $L$ in the symplectic space $\rmC_{\QQ_{p^{2}}}$ over $\QQ_{p^{2}}$ is by definition a $\ZZ_{p^{2}}$-lattice $L\subset \rmC_{\QQ_{p^{2}}}$ such that 
\begin{equation*}
pL^{\vee}\subset L\subset L^{\vee}
\end{equation*}
where $L^{\vee}$ is the integral dual of $L$ in $\rmC_{\QQ_{p^{2}}}$. Let $\calL_{\{i\}}$ be the set of vertex lattices of type $i\in\{0, 1, 2\}$, they are by definition given by
\begin{itemize}
\item[] $\calL_{\{0\}}=\{\text{vertex lattice $L_{0}$ of type $0$ in $\rmC_{\QQ_{p^{2}}}$: } pL^{\vee}_{0}\subset^{4} L_{0}\subset^{0} L^{\vee}_{0}\}$;
\item[] $\calL_{\{2\}}=\{\text{vertex lattice $L_{2}$ of type $2$ in $\rmC_{\QQ_{p^{2}}}$: }  pL^{\vee}_{2}\subset^{0} L_{2}\subset^{4} L^{\vee}_{2}\}$;
\item[] $\calL_{\{1\}}=\{\text{vertex lattice $L_{1}$ of type $1$ in $\rmC_{\QQ_{p^{2}}}$: }  pL^{\vee}_{1}\subset^{2} L_{1}\subset^{2} L^{\vee}_{1}\}$.
\end{itemize}

Let $L$ be a pair $L=(L_{0}, L_{2})$ where $L_{0}$ and $L_{2}$ are vertex lattices of type $0$ and $2$ for $\rmC_{\QQ_{p^{2}}}$. Then $L=(L_{0}, L_{2})$ is called a vertex lattice of type $02$ for $\rmC$ and the collection of type $02$-lattices will be denoted by $\calL_{\{02\}}$. If $L_{1}$ is a vertex lattice of type $1$ for $\rmC_{\QQ_{p^{2}}}$, then it gives rise to a vertex lattice of type $1$ for $\rmC$. Vertex lattices of type $02$ and $1$ give rise to a maximal parahoric subgroups in $\GU_{2}(\rmD)$ by taking the stabilizer of these lattices. 
\begin{itemize}
\item Vertex lattices of type $02$ correspond to the Siegel parahoric subgroup $\rmK^{\prime}_{\{02\}}=\mathrm{Sie}_{\rmD}(p)$;
\item Vertex lattices of type $1$ correspond to the paramodular subgroup $\rmK^{\prime}_{\{1\}}=\Pa_{\rmD}(p)$.
\end{itemize}

For each vertex lattice $L_{02}=(L_{0}, L_{2})$ of type $02$ and each vertex lattice $L_{1}$ of type $1$, we have defined in \cite{Wanga} projective subschemes $\rmM_{\Pa}\{L_{02}\}$ and $\rmM_{\Pa}\{L_{1}\}$ of $\rmM_{\Pa}$. In fact, each $\rmM_{\Pa}\{L_{02}\}$ is isomorphic to a projective line $\PP^{1}$ and 
$\rmM_{\Pa}\{L_{1}\}$ is isomorphic to a reduced point which can be defined over $\FF_{p^{2}}$. We will refer to these projective schemes as the \emph{lattice strata}. We can also define the open lattice strata by
$\rmM^{\circ}_{\Pa}\{L_{02}\}=\rmM_{\Pa}\{L_{02}\}-\bigcup\limits_{L_{1}\in\calL_{\{1\}}}\rmM_{\Pa}\{L_{1}\}$. This allows us to define the \emph{Bruhat-Tits strata} of $\rmM_{\Pa}$ by
\begin{equation*}
\begin{aligned}
& \rmM^{\circ}_{\Pa}\{02\}=\bigcup\limits_{L_{02}\in\calL_{\{02\}}}\rmM^{\circ}_{\Pa}\{L_{02}\}; \\
& \rmM_{\Pa}\{1\}=\bigcup\limits_{L_{1}\in\calL_{\{1\}}}\rmM_{\Pa}\{L_{1}\}. \\
\end{aligned}
\end{equation*}
Then by \cite[Theorem 7.3]{Wanga} we have the following description of the scheme $\rmM_{\Pa}$.
\begin{proposition}\label{BT-para}
The scheme $\rmM_{\Pa}$ admits the Bruhat-Tits stratification
\begin{equation*}
\rmM_{\Pa}=\rmM^{\circ}_{\Pa}\{02\}\sqcup \rmM_{\Pa}\{1\}.
\end{equation*}
The irreducible components of $\rmM_{\Pa}$ are given by the lattice stratum of the form $\rmM_{\Pa}\{L_{02}\}$ for a vertex lattice $L_{02}$ in $\calL_{\{02\}}$. The singular locus of $\rmM_{\Pa}$ is given by $\rmM_{\Pa}\{1\}$. 
\end{proposition}

The Rapoport--Zink uniformization theorem \cite[Theorem 6.30]{RZ96} furnishes an isomorphism
\begin{equation}\label{RZ-unif}
\overline{\rmX}^{\ss}_{\Pa}(p)\cong \rmI(\QQ)\backslash \calM_{\Pa, \mathrm{red}}\times \GSp_{4}(\mathbb{A}^{(\infty p)})/\rmU^{p} .
\end{equation}
Here $\rmI(\QQ)\cong \GU_{2}(\rmB^{\prime})(\QQ)$ for the quaternion algebra $\rmB^{\prime}=\rmB_{p\infty}$ over $\QQ$ which ramifies at $p$ and infinity and $\calM_{\Pa, \mathrm{red}}$ is the underlying reduced subscheme of $\calM_{\Pa}$.

We define the quaternionic unitary Shimura set for the group $\rmG(\rmB^{\prime})=\GU_{2}({\rmB^{\prime}})$ by the double coset space
\begin{equation*}
\rmZ_{\rmU}({\rmB}^{\prime})=\rmG({\rmB}^{\prime})(\QQ)\backslash \rmG({\rmB}^{\prime})(\mathbb{A}^{(\infty)})/\rmU
\end{equation*}
where $\rmU\subset \rmG(\rmB^{\prime})(\mathbb{A}^{\infty})$ is an open compact subgroup.  These Shimura sets appear naturally when we discuss the structure of the supersingular locus of $\rmX_{\Pa}(p)$. In this case, we will use another set of notations.
\begin{itemize}
\item When $\rmU=\rmK^{\prime}_{\{i\}}\rmU^{(p)}$ for $i\in\{1, 02\}$, we denote by $\rmZ_{\{i\}}(\rmB^{\prime})$ the Shimura set $\rmZ_{\rmU}({\rmB^{\prime}})$ with this level $\rmU$ structure;

\item Since $\rmK^{\prime}_{\{02\}}$ is also denoted by $\mathrm{Sie}(p)$, we sometimes use the notation
$\rmZ_{\mathrm{Sie}}(\rmB^{\prime})$ for the Shimura set $\rmZ_{\{02\}}(\rmB^{\prime})$ and we refer to it as the Shimura set with Siegel parahoric level;

\item Since $\rmK^{\prime}_{\{1\}}$ is also denoted by $\mathrm{Pa}(p)$, we sometimes use the notation 
$\rmZ_{\mathrm{\Pa}}(\rmB^{\prime})$ for the Shimura set $\rmZ_{\{1\}}(\rmB^{\prime})$ and we refer to it as the Shimura set with paramodular level.
\end{itemize}
The same construction as in \S $3.5$ defines an \'etale local system $\VV_{\mu}$ on the Shimura set $\rmZ_{\rmU}(\rmB^{\prime})$ for each character $\mu\in\rmX^{\ast}(\rmT)$. The space of global sections $\rmH^{0}(\rmZ_{\rmU}(\rmB^{\prime}), \VV_{\mu})$ will be denoted by $\calC(\rmZ_{\rmU}(\rmB^{\prime}), \VV_{\mu})$. 

\begin{proposition}\label{Siegel-double}
The supersingular locus $\overline{\rmX}^{\ss}_{\Pa}(p)$ admits the following descriptions.
\begin{enumerate}
\item The irreducible components of $\overline{\rmX}^{\ss}_{\Pa}(p)$ are all one dimensional and are parametrized by the Shimura set
$\rmZ_{\mathrm{Sie}}(\rmB^{\prime})$ with Siegel parahoric level structure.

\item The singular locus of $\overline{\rmX}^{\sin}_{\Pa}(p)$ is parametrized by the Shimura set $\rmZ_{\Pa}(\rmB^{\prime})$ with paramodular level structure.

\end{enumerate}
\end{proposition}

\begin{proof}
By Proposition \ref{BT-para}, the irreducible components of the scheme $\calM_{\Pa}$ corresponds to vertex lattices of type $02$. The derived group of $\GU_{2}(\rmD)$ acts transitively on these lattices and the stabilizer corresponds to the Siegel parahoric $\mathrm{Sie}_{\rmD}(p)=\rmK^{\prime}_{\{02\}}$. Therefore the irreducible components of $\rmM_{\Pa}$ is parametrized by $\GU_{2}(\rmD)/\mathrm{Sie}_{\rmD}(p)$. Now we apply \eqref{RZ-unif} to see that the irreducible components of $\overline{\rmX}^{\ss}_{\Pa}(p)$ are parametrized by 
\begin{equation*}
\rmI(\QQ)\backslash \rmJ_{b}(\QQ_{p})/\rmK^{\prime}_{\{02\}}\times \rmG({\rmB}^{\prime})(\mathbb{A}^{(\infty p)})/\rmU^{(p)}
\end{equation*} 
but this is precisely $\rmZ_{\mathrm{Sie}}(\rmB^{\prime})$. 
By Proposition \ref{BT-para}, the singular locus of the scheme $\rmX_{\Pa}(p)$ corresponds to vertex lattices of type $1$. Then we proceed similarly as above to finish the proof of $(2)$.
\end{proof}
\begin{remark}
This result can also be deduced from  \cite{Yu-doc} although the method there is completely different.
\end{remark}

\section{Cohomological vanishing for paramodular Siegel threefold}
In order to prove the Mazur's principle, we need to analyze the Picard--Lefschetz formula for the nearby cycle cohomology 
\begin{equation*}
\rmH^{3}_{\rmc}(\overline{\rmX}_{\Pa}(p)\otimes\overline{\FF}_{p},\rmR\Psi (\VV)).
\end{equation*}
For this purpose, we will prove a vanishing result for the non-middle-degree cohomology of the special fiber $\overline{\rmX}_{\Pa}(p)$ when localized at a non-Eisenstein maximal ideal $\frakm$ of $\TT$. This result essentially follows from the work of van Hoften in \cite{VH19} but we give a slightly different exposition.  We define a maximal ideal $\frakm$ to be non-Eisenstein if the natural map
\begin{equation*}
\rmH^{i}_{\rmc}(\rmX\otimes\overline{\QQ}, \VV)_{\frakm}\rightarrow \rmH^{i}(\rmX\otimes\overline{\QQ}, \VV)_{\frakm}
\end{equation*}
is an isomorphism for $\rmX=\rmX_{\rmH}$ or $\rmX=\rmX_{\Pa}(p)$. 
\begin{remark}
Suppose $\frakm$ is associated to a residual Galois representation $\overline{\rho}_{\pi, \lambda}$. It is known from \cite[Corollary 3]{MT14} that if $\overline{\rho}_{\pi,\lambda}$ satisfies the assumption that the image of $\overline{\rho}_{\pi,\lambda}(\rmG_{\QQ})$ contains $\GSp_{4}(\FF_{l})$ and that $l+1\geq k_{1}+k_{2}$ for the weight $(k_{1}, k_{2})$ of $\pi$. Then $\frakm$ is indeed non-Eisenstein. 
\end{remark}
We also impose the following assumption on the maximal ideal $\frakm$ throughout this section. 
\begin{itemize}
\item[] The cohomology groups  $\rmH^{i}_{\rmc}(\rmX_{\Pa}(p)\otimes\overline{\QQ}, \VV)_{\frakm}$ and $\rmH^{i}_{\rmc}(\rmX_{\rmH}\otimes\overline{\QQ}, \VV)_{\frakm}$ are both concentrated in the middle degree. 
\end{itemize}
This assumption can be verified in the situations discussed in Remark \ref{CS-van}.

\subsection{Klingen type correspondence}
The Siegel modular threefold with Klingen level structure realizes a correspondence between $\rmX_{\rmH}$ and $\rmX_{\Pa}(p)$. More precisely, $\rmX_{\Kl}(p)$ fits in the following diagram
\begin{equation*}
\begin{tikzcd}
  &\rmX_{\Kl}(p)  \arrow[ld, "p_{1}" '] \arrow[rd, "p_{2}"] &   \\
    \rmX_{\rmH} &                         & \rmX_{\Pa}(p)
\end{tikzcd}
\end{equation*}
where $p_{1}: \rmX_{\Kl}(p)  \rightarrow \rmX_{\rmH}$ is the natural map forgetting the subgroup $\rmH$ and $p_{2}: \rmX_{\Kl}(p)  \rightarrow \rmX_{\Pa}(p)$ is the map sending $\sfA$ to $\sfA/\sfH$. Note that 
\begin{equation*}
\sfA\rightarrow \sfA/\mathsf{H}\rightarrow \sfA/\sfH^{\perp}\rightarrow \sfA^{\vee}
\end{equation*}
 is the map given by $p\lambda$. 

\begin{proposition}
The scheme $\rmX_{\Kl}(p)$ has semistable reduction over $\ZZ_{p}$. Its special fiber $\overline{\rmX}_{\Kl}(p)$ decomposes into $\overline{\rmX}^{\mathrm{e}}_{\Kl}(p)\cup \overline{\rmX}^{\mathrm{m}}_{\Kl}(p)$ where each component admits the following descriptions.
\begin{enumerate}
\item $\overline{\rmX}^{\mathrm{e}}_{\Kl}(p)$ is the Zariski closure of the locus $\overline{\rmX}^{\rme\cdot\circ}_{\Kl}(p)$ where $\sfH$ is \'{e}tale and similarly $\overline{\rmX}^{\mathrm{m}}_{\Kl}(p)$ is the Zariski closure of the locus $\overline{\rmX}^{\mathrm{m}\cdot\circ}_{\Kl}(p)$ where $\sfH$ is multiplicative.
\item Let $\overline{\rmX}^{a}_{\Kl}(p)=\overline{\rmX}^{\mathrm{e}}_{\Kl}(p)\cap \overline{\rmX}^{\mathrm{m}}_{\Kl}(p)$. It is the singular locus of $\overline{\rmX}_{\Kl}(p)$ which is also the locus where $\sfH$ is additive. 
\end{enumerate}
\end{proposition}
\begin{proof}
All of the statements follow from a local model computation. See for example \cite{GT-TaylorWiles}.
\end{proof}

Over $\overline{\rmX}_{\rmH}$, we have the Ekedhal--Oort stratification
\begin{equation*}
\overline{\rmX}_{\rmH}=\overline{\rmX}^{[=2]}_{\rmH}\sqcup \overline{\rmX}^{[=1]}_{\rmH} \sqcup \overline{\rmX}^{[\mathrm{sg}]}_{\rmH} \sqcup \overline{\rmX}^{[\mathrm{sp}]}_{\rmH}.
\end{equation*}
Here the stratum $\overline{\rmX}^{[=2]}_{\rmH}$ agree with the ordinary locus $\overline{\rmX}^{\ord}_{\rmH}$; the stratum $\overline{\rmX}^{[=1]}_{\rmH}$ is the $p$-rank $1$ locus; the stratum $\overline{\rmX}^{[\mathrm{sg}]}_{\rmH}$ is the supergeneral locus, in other words it is the supersingular but non-superspecial locus and the stratum $\overline{\rmX}^{[\mathrm{sp}]}_{\rmH}$ is the superspecial locus which consists of finite number of points.

\begin{proposition}\label{p1}
We can the describe the map $p_{1}: \overline{\rmX}^{\rm{e}}_{\Kl}(p)\rightarrow \overline{\rmX}_{\rmH}$ with respect to the stratification 
\begin{equation*}
\overline{\rmX}_{\rmH}=\overline{\rmX}^{[=2]}_{\rmH}\sqcup \overline{\rmX}^{[=1]}_{\rmH} \sqcup \overline{\rmX}^{[\mathrm{sg}]}_{\rmH} \sqcup \overline{\rmX}^{[\mathrm{sp}]}.
\end{equation*}
\begin{enumerate}
\item ${\overline{\rmX}^{\rme}_{\Kl}}(p)\vert_{ \overline{\rmX}^{[=2]}_{\rmH}}\rightarrow \overline{\rmX}^{[=2]}_{\rmH}$ can be factored into a map which is finite \'{e}tale map of degee $p+1$ and a purely inseperable map of degree $p^{2}$;
\item ${\overline{\rmX}^{\rme}_{\Kl}}(p)\vert_{ \overline{\rmX}^{[=1]}_{\rmH}}\rightarrow \overline{\rmX}^{[=1]}_{\rmH}$ is an isomorphism;
\item ${\overline{\rmX}^{\rme}_{\Kl}}(p)\vert_{ \overline{\rmX}^{[\mathrm{sg}]}_{\rmH}}\rightarrow \overline{\rmX}^{[\mathrm{sg}]}_{\rmH}$ is an isomorphism;
\item The fiber ${\overline{\rmX}^{\rme}_{\Kl}}(p)\vert_{ \overline{\rmX}^{[\mathrm{sp}]}_{\rmH}}$ is a $\PP^{1}$ bundle over a finite set and 
\begin{equation*}
{\overline{\rmX}^{\rme}_{\Kl}}(p)\vert_{ \overline{\rmX}^{[\mathrm{sp}]}_{\rmH}}\rightarrow \overline{\rmX}^{[\mathrm{sp}]}_{\rmH}
\end{equation*}
is the contraction of each $\PP^{1}$ to its base point.
\end{enumerate}
\end{proposition}
\begin{proof}
The first two point are  well-known by elementary computations. For example, we define the relative moduli space $\overline{\rmY}^{\rme}_{\Kl}(p)\vert_{ \overline{\rmX}^{[=2]}_{\rmH}}$ over $\overline{\rmX}^{[=2]}_{\rmH}$ by considering the connected-\'{e}tale exact sequence 
\begin{equation*}
0\rightarrow \sfA^{\circ}[p]\rightarrow \sfA[p]\rightarrow \sfA^{\mathrm{et}}[p]\rightarrow0
\end{equation*}
of the universal abelian scheme $\sfA$ over $\overline{\rmX}^{[=2]}_{\rmH}$ and by taking $\overline{\rmY}^{\rme}_{\Kl}(p)\vert_{ \overline{\rmX}^{[=2]}_{\rmH}}$ as the moduli of subgroup $\sfH$ of $\sfA^{\mathrm{et}}[p]$ of degree $p$. The second point is proved by a similar reasoning. The last two points are proved by an easy Di\'{e}udonne module computation, see \cite[Theorem 1.2]{Yu-doc}.
\end{proof}

We will consider the naive stratification of $\overline{\rmX}_{\Pa}(p)$ into its smooth locus and its singular locus
\begin{equation*}
\overline{\rmX}_{\Pa}(p)=\overline{\rmX}^{\mathrm{sm}}_{\Pa}(p)\sqcup \overline{\rmX}^{\sing}_{\Pa}(p).
\end{equation*}

\begin{proposition}\label{p2}
We can describe the map $p_{2}: \overline{\rmX}^{\mathrm{e}}_{\Kl}(p)\rightarrow \overline{\rmX}_{\Pa}(p)$ with respect to the stratification
\begin{equation*}
\overline{\rmX}_{\Pa}(p)=\overline{\rmX}^{\mathrm{sm}}_{\Pa}(p)\sqcup \overline{\rmX}^{\sing}_{\Pa}(p)
\end{equation*}
in the following way.
\begin{enumerate}
\item $\overline{\rmX}^{\mathrm{e}}_{\Kl}(p)\vert_{\overline{\rmX}^{\mathrm{sm}}_{\Pa}(p)}\rightarrow \overline{\rmX}^{\mathrm{sm}}_{\Pa}(p)$ is an isomorphism;
\item $\overline{\rmX}^{\mathrm{e}}_{\Kl}(p)\vert_{\overline{\rmX}^{\mathrm{sin}}_{\Pa}(p)}$ is a $\PP^{1}$ bundle over a finite set of points. The map 
\begin{equation*}
\overline{\rmX}^{\mathrm{e}}_{\Kl}(p)\vert_{\overline{\rmX}^{\mathrm{sin}}_{\Pa}(p)}\rightarrow \overline{\rmX}^{\mathrm{sin}}_{\Pa}(p)
\end{equation*}
is the contraction of each $\PP^{1}$ to its base point.
\end{enumerate}
\end{proposition}
\begin{proof}
This follows from a similar computation as in the previous proposition. A proof of these statements can be found in \cite[Lemma 5.3.2]{VH19}.
\end{proof}

\subsection{Comparison of cohomologies} We recall the definition of a \emph{(semi-)small morphism}. Let $f: \rmX\rightarrow \rmY$ be a morphism between two finitely generated schemes over a field $k$. Suppose that $\rmY$ admits a partition $\rmY=\sqcup^{r}_{\alpha=0} \rmY_{\alpha}$ where $\rmY_{0}$ is open and dense in $\rmY$. 
\begin{definition}
We say $f$ is semi-small with respect to the partition
$\rmY=\sqcup^{r}_{\alpha=0} \rmY_{\alpha}$ if 
\begin{equation*}
2\cdot\mathrm{dim}(f^{-1}(y))\leq n-\mathrm{dim}(\rmY_{\alpha}) 
\end{equation*}
holds for all $\alpha$ and all $y\in \rmY_{\alpha}$. We say $f$ is small if $f$ is semismall and 
\begin{equation*}
2\cdot\mathrm{dim}(f^{-1}(y))< n-\mathrm{dim}(\rmY_{\alpha}) 
\end{equation*}
for all $\alpha>0$ and $y\in \rmY_{\alpha}$.
\end{definition}

\begin{corollary}
The morphism $p_{1}: \overline{\rmX}^{\rm{e}}_{\Kl}(p)\rightarrow \overline{\rmX}_{\rmH}$ is small with respect to the stratification 
\begin{equation*}
\overline{\rmX}_{\rmH}=\overline{\rmX}^{[=2]}_{\rmH}\sqcup \overline{\rmX}^{[=1]}_{\rmH} \sqcup \overline{\rmX}^{[\mathrm{sg}]}_{\rmH} \sqcup \overline{\rmX}^{[\mathrm{sp}]}.
\end{equation*}
The morphism $p_{2}: \overline{\rmX}^{\mathrm{e}}_{\Kl}(p)\rightarrow \overline{\rmX}_{\Pa}(p)$ is small with respect to the stratification
\begin{equation*}
\overline{\rmX}_{\Pa}(p)=\overline{\rmX}^{\mathrm{sm}}_{\Pa}(p)\sqcup \overline{\rmX}^{\mathrm{sin}}_{\Pa}(p).
\end{equation*}
\end{corollary}
\begin{proof}
This follows immediately from Proposition \ref{p1} and Proposition \ref{p2}.
\end{proof}

Since $\overline{\rmX}^{\mathrm{e}}_{\Kl}(p)$ is smooth, $p_{2}$ is in fact a small resolution of $\overline{\rmX}_{\Pa}(p)$.  We have the following useful result regarding to the push-forward of a perverse sheaf along a small morphism \cite[Lemma 7.5]{KW-perverse}.

\begin{lemma}\label{inter}
Suppose $\rmX$ is smooth and equidimensional over the base field $k$ and $f: \rmX\rightarrow Y$ is small and proper with respect to the partition
$\rmY=\sqcup^{r}_{\alpha=0} \rmY_{\alpha}$. Let $j: \rmY_{0}\hookrightarrow \rmY$ be the open inclusion.
Let $A$ be a smooth perverse sheaf on $\rmX$. Then
\begin{equation*}
\rmR f_{\ast}A=j_{!\ast}(\rmR f_{\ast}A\vert_{\rmY_{0}}).
\end{equation*}
\end{lemma}

\begin{proposition}\label{Kl-Pa}
For $i\geq 4$, we have the following isomorphism
\begin{equation*}
\rmH^{i}(\overline{\rmX}^{\rme}_{\Kl}(p), \VV)\cong \rmH^{i}(\overline{\rmX}_{\Pa}(p), \VV).
\end{equation*}
\end{proposition}
\begin{proof}
Note that $p_{2}$ is a small resolution. The result follows from the following elementary computations
\begin{equation*}
\begin{aligned}
\rmH^{i}(\overline{\rmX}^{\rme}_{\Kl}(p), \VV)&\cong \rmH^{i}(\overline{\rmX}_{\Pa}(p), \rmR p_{1\ast}\VV)\\
&\cong \rmH^{i}(\overline{\rmX}_{\Pa}(p), j_{!\ast}\rmR p_{1\ast}\VV\mid_{\overline{\rmX}^{\mathrm{sm}}_{\Pa}(p)})\\ 
&\cong \rmH^{i}(\overline{\rmX}_{\Pa}(p), j_{!\ast}\VV)\\
&\cong \rmH^{i}(\overline{\rmX}_{\Pa}(p), \VV).\\ 
\end{aligned}
\end{equation*}
The second isomorphism follows from \cite[Lemma 7.5]{KW-perverse} and last isomorphism follows from the fact that $\overline{\rmX}_{\Pa}(p)$ has isolated singularities and simple calculations using excision long exact sequence. 
\end{proof}

Recall that the supersingular locus of  $\overline{\rmX}^{\rme}_{\Kl}(p)$ is one dimensional and is uniformized by the union of Shimura sets $\rmZ_{\mathrm{Sie}}(\rmB^{\prime})\cup \rmZ_{\mathrm{Pa}}(\rmB^{\prime})$ by Theorem \ref{Siegel-double}. Therefore we obtain the following Gysin map
\begin{equation*}
\mathrm{inc}^{\Kl}_{\mathrm{ss}!}: \mathcal{C}(\rmZ_{\mathrm{Sie}}(\rmB^{\prime}), \VV)\oplus \mathcal{C}(\rmZ_{\mathrm{Pa}}(\rmB^{\prime}), \VV)\rightarrow\rmH^{4}(\overline{\rmX}^{\rme}_{\Kl}(p), \VV(2)).
\end{equation*}

On the other hand, the supersingular locus of $\overline{\rmX}_{\rmH}$ is also one dimensional and is parametrized by the Shimura set $\rmZ_{\mathrm{Sie}}(\rmB^{\prime})$ and hence we also have a Gysin map
\begin{equation*}
\mathrm{inc}^{\rmH}_{\mathrm{ss}!}: \mathcal{C}(\rmZ_{\mathrm{Pa}}(\rmB^{\prime}), \VV)\rightarrow\rmH^{4}(\overline{\rmX}_{\rmH}, \VV(2)).
\end{equation*}
The following is due to to van Hoften \cite[Theorem 7.12, Lemma 7.1.3, 7.14]{VH19}.

\begin{proposition}\label{vanpa}
We have the following statements
\begin{enumerate}
\item The cohomology group
\begin{equation*}
\rmH^{4}(\overline{\rmX}_{\Pa}(p), \VV)_{\frakm}
\end{equation*}
vanishes.

\item Similarly, the cohomology group
\begin{equation*}
\rmH^{2}_{\rmc}(\overline{\rmX}_{\Pa}(p), \VV)_{\frakm}.
\end{equation*}
vanishes.
\end{enumerate}
\end{proposition}

\begin{proof}
We only treat the case of $i=4$. The other case is given by dualizing the argument given below. The proof essentially is given by following commutative diagram
\begin{equation*}
\begin{tikzcd}
\calC(\rmZ_{\mathrm{Pa}}(\rmB^{\prime}), \VV) \arrow[d] \arrow[r, "\mathrm{inc}^{\mathrm{H}}_{\mathrm{ss}!}"] &  \mathrm{H}^{4}(\overline{\rmX}_{\rmH}, \VV(2))  \arrow[d] \\
 \mathcal{C}(\rmZ_{\mathrm{Sie}}(\rmB^{\prime}), \VV)\oplus \mathcal{C}(\rmZ_{\mathrm{Pa}}(\rmB^{\prime}), \VV) \arrow[r, "\mathrm{inc}^{\Kl}_{\mathrm{ss}!}"] & \mathrm{H}^{4}(\overline{\rmX}^{\mathrm{e}}_{\Kl}(p), \VV(2))  \\
\calC(\rmZ_{\mathrm{Pa}}(\rmB^{\prime}), \VV) \arrow[u] \arrow[r, "\alpha"] &  \mathrm{H}^{4}(\overline{\rmX}_{\Pa}(p), \VV(2))  \arrow[u, "\cong"].
\end{tikzcd}
\end{equation*}
Note the boundary map 
\begin{equation*}
\bigoplus\limits_{\sigma\in\Sigma_{\Pa}(p)}\rmR^{3}\Phi_{\sigma}(\VV(2))_{\frakm}  \xrightarrow{\alpha} \mathrm{H}^{4}(\overline{\rmX}_{\Pa}(p), \VV(2))_{\frakm} 
\end{equation*}
in the exact sequence \ref{van-cycle-ext} agrees with the composite map 
\begin{equation*}
\calC(\rmZ_{\mathrm{Pa}}(\rmB^{\prime}), \VV)\rightarrow \mathrm{H}^{4}(\overline{\rmX}^{\mathrm{e}}_{\Kl}(p), \VV(2))
\end{equation*}
given in the above diagram if we identify $\bigoplus_{\sigma\in\Sigma_{\Pa}(p)}\rmR^{3}\Phi_{\sigma}(\VV(2))_{\frakm}$ with $\calC(\rmZ_{\mathrm{Pa}}(\rmB^{\prime}), \VV)$ using \ref{Siegel-double} and identify $\mathrm{H}^{4}(\overline{\rmX}^{\mathrm{e}}_{\Kl}(p), \VV(2))$ with $\mathrm{H}^{4}(\overline{\rmX}_{\Pa}(p), \VV(2))_{\frakm}$ using Proposition \ref{Kl-Pa}. This observation is clear from the description of the Klingen correspondence as in Proposition \ref{p1} and Proposition \ref{p2}. But note that $\mathrm{H}^{4}(\overline{\rmX}_{\rmH}, \VV(2))_{\frakm}\cong\mathrm{H}^{4}({\rmX}_{\rmH}\otimes\overline{\QQ}_{p}, \VV(2))_{\frakm}$ vanishes under our assumption. Therefore the image of $\alpha$ is zero. On the other hand, the map $\alpha$ is also surjective under our assumption and hence $\mathrm{H}^{4}(\overline{\rmX}_{\Pa}(p), \VV(2))_{\frakm}$ vanishes as desired. 
\end{proof}

\begin{remark}
One may naturally wonder why the contribution 
\begin{equation*}
\mathrm{inc}_{\mathrm{ss} !}: \mathcal{C}(\rmZ_{\mathrm{Sie}}(\rmB^{\prime}), \VV)\rightarrow\rmH^{4}(\overline{\rmX}_{\Pa}(p), \VV(2)).
\end{equation*}
of the Tate cycles on the supersingular locus also vanishes. We could not give a satisfactory answer to this observation. The reason might be the following: for a cuspidal automorphic representation $\pi$ of general type appears in $\mathcal{C}(\rmZ_{\mathrm{Sie}}(\rmB^{\prime}), \VV)$, the local component $\pi_{p}$ of $\pi$ must corresponds to the type $\mathrm{II}a$ representation under the Jacquet--Langlands correspondence by an examination of the table $3$ in \cite{Schmi05}. Hence $\mathcal{C}(\rmZ_{\mathrm{Sie}}(\rmB^{\prime}), \VV)$ does not provide new representations than the space $\calC(\rmZ_{\mathrm{Pa}}(\rmB^{\prime}), \VV)$ already does.
\end{remark}

\section{Mazur's principle for the paramodular Siegel threefold}

\subsection{Mazur's principle} We recall that we are concerned with a cuspidal automorphic representation $\pi$ for $\GSp_{4}(\mathbb{A})$ which is of general type and whose weight is $(k_{1}, k_{2})$ with $k_{1}\geq k_{2}\geq 3$ and $k_{1}\equiv k_{2}\mod 2$. We now assume that the local component $\pi_{p}$ of $\pi$ is a type IIa representation. In this case, its associated Galois representation $\rho_{\pi,\lambda}$ is realized in the Galois module $\rmH^{3}(\rmX_{\Pa}(p)\otimes\overline{\QQ}, \VV)$. We first give a list of assumptions that we will impose on the residual Galois representation for Mazur's principle.
\begin{assumption}\label{Pa(p)-assumptions}
We will make the following assumptions throughout this section.
\begin{enumerate}
\item The residual Galois representation $\overline{\rho}_{\pi, \lambda}$ is absolutely irreducible.

\item Let $\frakm$ be the maximal ideal corresponding to $\overline{\rho}_{\pi, \lambda}$ as explained in Definition \ref{maximal-ideal}, we assume that
\begin{equation*}
\mathrm{H}^{i}_{\rmc}(\rmX_{\Pa}(p)\otimes\overline{\QQ}, \mathbb{V})_{\frakm}
\end{equation*}
vanishes unless $i=3$.

\item The cohomology group
\begin{equation*} \label{semi}
\mathrm{H}^{3}_{\rmc}(\rmX_{\Pa}(p)\otimes\overline{\QQ}, \mathbb{V})_{\mathfrak{m}}\otimes k 
\end{equation*}
is semisimple as a Galois module. This implies that
\begin{equation*}
\mathrm{H}^{3}_{\rmc}(\rmX_{\Pa}(p)\otimes\overline{\QQ}, \mathbb{V})_{\mathfrak{m}}\otimes k=\overline{\rho}^{\oplus s}_{\pi,\lambda}
\end{equation*}
for some positive number $s$. 
\end{enumerate}
\end{assumption}

\begin{remark}\label{CS-van}
We remark on $(2)$ of the above assumption. This assumption is known in the following cases.
\begin{itemize}
\item Under the ordinary assumption of $\rho_{\pi, \lambda}$, the large image assumption on $\overline{\rho}_{\pi, \lambda}$  and the assumption $l+1\geq k_{1}+k_{2}$, then $(2)$ indeed holds true. This follows from \cite[Theorem 1]{MT14}. Under the assumption that $\overline{\rho}_{\pi,\lambda}$ has large image, $(2)$ also follows from an unpublished work of E. Urban using the fact that the complement of the Igusa divisor is affine. This is mentioned in  \cite{MT14}.

\item Under the assumption that $\mu$ is sufficiently regular, it follows from \cite[Theorem 10.1]{LS13} that $(2)$ holds. This is the assumption used in \cite{VH19}.
\item Under the assumption that there is a prime $v$ where $\overline{\rho}_{\pi,\lambda}\vert_{\rmG_{\QQ_{v}}}$ is unramified and the Hecke parameters  is generic in the sense that the ratio of them is not $v^{\pm1}$, $(2)$ follows from combining the work \cite{Kos-b} of Koshikawa and forthcoming work of M. Santos on the semi-perversity of the Hodge--Tate period map. This can be seen as a generalization of the pioneering work of Scholze--Caraiani \cite{CS-cpt, CS-noncpt} to more general Shimura varieties. 
\end{itemize}

We will refer to part $(3)$ of this assumption as the semisimple assumption. For Galois representations appearing on the Shimura curves, the analogues assumption holds by the main result of \cite{BLR91}. On the other hand, work of Nekov\'{a}\v{r} in \cite{Nek} should shed some lights on verifying this assumption.
\end{remark}

The main result of this section is the following theorem which we will refer to as the Mazur's principle. 

\begin{theorem}\label{Mz-principle}
Let $\pi$ be a cuspidal automorphic representation of $\GSp_{4}(\mathbb{A})$ of general type which is of weight $(k_{1}, k_{2})$ satisfying $k_{1}\geq k_{2}\geq 3$ and $k_{1}\equiv k_{2} \mod 2$. Let $p$ be a prime distinct from $l$ and such that $p\not\equiv 1 \mod l$. Let $\rmU=\Pa(p)\rmU^{p}\subset \GSp_{4}(\QQ_{p})\GSp_{4}(\mathbb{A}^{(\infty p)})$ be a neat open compact subgroup such that $\pi^{\rmU}\neq 0$. We assume that $\pi_{p}$ is ramified of type $\mathrm{II}a$ in the classification of Schmidt. Suppose the following assumptions hold. 
\begin{enumerate}
\item The residual Galois representation $\overline{\rho}_{\pi, \lambda}$ satisfies Assumption \ref{Pa(p)-assumptions}. 
\item The residual Galois representation $\overline{\rho}_{\pi, \lambda}\mid_{\rmG_{\QQ_{p}}}$ is unramified.
\end{enumerate}
Then there exists a cuspidal automorphic representation $\breve{\pi}$ of the same weight as $\pi$ such that 
\begin{equation*}
\overline{\rho}_{\pi, \lambda}\cong \overline{\rho}_{\breve{\pi}, \lambda}
\end{equation*}
and $\breve{\pi}_{p}$ is an unramified principal series.
\end{theorem} 

\subsection{Picard--Lefschetz formula} 
In this subsection, we are mainly concerned with analyzing the following diagram which we refer to as the Picard--Lefschetz formula for the Galois module $\rmH^{3}(\overline{\rmX}_{\Pa}(p)\otimes\overline{\FF}_{p}, \rmR\Psi(\VV))_{\frakm}$. We let $\Sigma_{\Pa}(p)$ be the singular locus of $\overline{\rmX}_{\Pa}(p)_{\overline{\FF}_{p}}$.
\begin{equation*}
\begin{tikzcd}
\mathrm{H}^{3}(\overline{\rmX}_{\Pa}(p)_{\overline{\FF}_{p}}, \rmR\Psi(\VV)(1))_{\frakm} \arrow[r, "\mathrm{ca}"] \arrow[d, "\mathrm{N}"] & \bigoplus\limits_{\sigma\in\Sigma_{\Pa}(p)}\rmR^{3}\Phi_{\sigma}(\VV(1))_{\frakm} \arrow[d, "\mathrm{\rmN}_{\Sigma}"]  \arrow[r, "\alpha"] & \mathrm{H}^{4}(\overline{\rmX}_{\Pa}(p)_{\overline{\FF}_{p}}, \VV(1))_{\frakm}\\
\mathrm{H}^{3}_{\rmc}(\overline{\rmX}_{\Pa}(p)_{\overline{\FF}_{p}}, \rmR\Psi(\VV))_{\frakm}   &\bigoplus\limits_{\sigma\in\Sigma_{\Pa}(p)}\rmH^{3}_{\{\sigma\}}(\overline{\rmX}_{\Pa}(p)_{\overline{\FF}_{p}}, \rmR\Psi(\VV))_{\frakm}  \arrow[l,"\mathrm{coca}"] & \mathrm{H}^{2}_{\rmc}(\overline{\rmX}_{\Pa}(p)_{\overline{\FF}_{p}}, \VV)_{\frakm}  \arrow[l, "\beta"]                
\end{tikzcd}          
\end{equation*}
By general theory, the map $\mathrm{ca}$ and $\mathrm{coca}$ are dual to each other and therefore $\mathrm{ca}$ is injective if and only if $\mathrm{coca}$ is surjective.  The following proposition is first due to van Hoften \cite[Proposition 7.1.2]{VH19}. 

\begin{proposition}\label{injbeta-para}
The canonical map $\mathrm{ca}$ localized at $\mathfrak{m}$ is injective
\begin{equation}
\mathrm{ca}:\mathrm{H}^{3}(\overline{\rmX}_{\Pa}(p)\otimes\overline{\FF}_{p},\rmR\Psi(\mathbb{V})(1))_{\frakm} \rightarrow  \bigoplus_{\sigma\in\Sigma_{p}(\Pa(p))} \rmR^{3}\Phi_{\sigma}(\mathbb{V}(1))_{\frakm}
\end{equation}
and dually
\begin{equation}
\mathrm{coca}: \bigoplus_{\sigma\in\Sigma_{p}(\Pa(p))} \mathrm{H}^{3}_{\{\sigma\}}(\overline{\rmX}_{\Pa}(p)\otimes\overline{\FF}_{p}, \mathbb{V})_{\frakm} \rightarrow \mathrm{H}^{3}_{\rmc}(\overline{\rmX}_{\Pa}(p)\otimes \overline{\FF}_{p},\rmR\Psi(\mathbb{V}))_{\frakm}
\end{equation}
is surjective. 
\end{proposition}
\begin{proof}
We have proved that $\mathrm{H}^{4}(\overline{\rmX}_{\Pa}(p), \VV(1))_{\frakm}$
vanishes in Proposition \ref{vanpa} and $\mathrm{ca}$ is surjective. Similarly, we have proved that $\mathrm{H}^{2}_{\rmc}(\overline{\rmX}_{\Pa}(p), \VV)_{\frakm}$ vanishes and therefore $\mathrm{coca}$ is injective.
\end{proof}

We have a decomposition of the Galois module
\begin{equation}\label{decomp}
\rmH^{3}_{\rmc}(\rmX_{\Pa}(p)\otimes{\overline{\QQ}}, \VV)_{\frakm}\otimes E_{\lambda}=\bigoplus_{\vec{\pi}}\rho^{\oplus m(\vec{\pi})}_{\vec{\pi}, \lambda}
\end{equation}
where $\vec{\pi}$ runs through all the autormorphic representation $\vec{\pi}$ whose Galois representation $\rho_{\vec{\pi}, \lambda}$ has the same residual Galois representation with $\rho_{\pi, \lambda}$ and $m(\vec{\pi})$ is the multiplicity of $\vec{\pi}$ which could be zero. Let $\vec{\pi}$ be an automorphic representation appearing in the above decomposition and such that $\vec{\pi}_{p}$ is ramified of type $\mathrm{IIa}$. Then we have the following description of $\rho_{\vec{\pi}, \lambda}$ using the Picard--Lefschetz formula.
\begin{proposition}\label{van-llc}
Let $\vec{\pi}$ be a cuspidal automorphic representation appearing in \eqref{decomp} such that $\vec{\pi}_{p}$ is ramified of type $\mathrm{IIa}$. 
\begin{enumerate}
\item The space 
\begin{equation*}
\mathrm{Gr}_{-1,\frakm}\otimes E_{\lambda}=\bigoplus\limits_{\sigma\in\Sigma_{p}(\Pa(p))}\mathrm{H}^{3}_{\{\sigma\}}(\overline{\rmX}_{\Pa}(p)\otimes{\overline{\FF}_{p}}, \mathbb{V})_{\frakm}\otimes E_{\lambda}
\end{equation*}
contributes to each $\rho_{\vec{\pi}, \lambda} \vert_{\rmG_{\QQ_{p}}}$ as a $1$-dimensional subspace;
\item The space 
\begin{equation*}
\mathrm{Gr}_{1,\frakm}\otimes {E_{\lambda}}=\bigoplus\limits_{\sigma\in\Sigma_{p}(\Pa(p))}\rmR^{3}\Phi_{\sigma}(\mathbb{V})_{\frakm}\otimes E_{\lambda}
\end{equation*}
contributes to each $\rho_{\vec{\pi}, \lambda}\vert_{\rmG_{\QQ_{p}}}$ as a $1$-dimensional quotient; 

\item  The quotient space 
\begin{equation*}
\mathrm{Gr}_{0,\frakm}\otimes E_{\lambda}=\frac{\rmH^{3}_{\rmc}(\overline{\rmX}_{\Pa}(p)\otimes{\overline{\FF}_{p}}, \VV)_{\frakm}\otimes E_{\lambda}}{\bigoplus\limits_{\sigma\in\Sigma_{p}(\Pa(p))}\mathrm{H}^{3}_{\{\sigma\}}(\overline{\rmX}_{\Pa}(p)\otimes{\overline{\FF}_{p}}, \mathbb{V})_{\frakm}\otimes E_{\lambda}}
\end{equation*} 
contributes to each $\rho_{\vec{\pi}, \lambda}\vert_{\rmG_{\QQ_{p}}}$ as a $2$-dimensional sub-quotient.
\end{enumerate}
\end{proposition}

\begin{proof}
The local-global compatibility of Langlands correspondence as in Theorem \ref{Gal-GSp4} implies that the monodromy operator $\rmN$ has one dimensional image. Thus the Picard--Lefschetz formula and Proposition \ref{injbeta-para} implies that 
\begin{equation*}
\bigoplus\limits_{\sigma\in\Sigma_{p}(\Pa(p))} \mathrm{H}^{3}_{\{\sigma\}}(\overline{\rmX}_{\Pa}(p)\otimes\overline{\FF}_{p}, \mathbb{V})\otimes E_{\lambda} 
\end{equation*}
has one dimensional contribution to $\rho_{\vec{\pi},\lambda}$. Since the local monodromy map $\rmN_{\Sigma}$ is an isomorphism, the space
\begin{equation*}
\bigoplus\limits_{\sigma\in\Sigma_{p}(\Pa(p))}\rmR^{3}\Phi_{\sigma}(\VV)_{\frakm}\otimes E_{\lambda}
\end{equation*}
contributes to $\rho_{\vec{\pi}, \lambda\mid \rmG_{\QQ_{p}}}$ as a $1$-dimensional quotient. Then the third claim is also clear from what we have proved.
\end{proof}

\subsection{Proof of the Mazur's principle} Now we come to the proof of Theorem \ref{Mz-principle}. We will prove the theorem by contradiction. Assume that we \emph{can not} find such a $\breve{\pi}$ in the statement of the theorem. This means each $\vec{\pi}$ appearing in \eqref{decomp} is ramified of type $\mathrm{IIa}$ at $p$. We first prove a Eichler--Shimura type congruence relation on the space of vanishing cycles 
\begin{equation*}
\bigoplus\limits_{\sigma\in\Sigma_{p}(\Pa(p))}\rmR^{3}\Phi_{\sigma}(\mathbb{V})_{\frakm} 
\end{equation*}
on $\overline{\rmX}_{\Pa}(p)$ under this assumption.

\begin{lemma}\label{cong-up-pa}
Suppose that all the $\vec{\pi}$ appears in \eqref{decomp} are ramified at $p$. 
\begin{enumerate}
\item  The Frobenius action $\mathrm{Frob}_{p}$ on the Galois submodule 
\begin{equation*}
\bigoplus\limits_{\sigma\in\Sigma_{p}(\Pa(p))}\mathrm{H}^{3}_{\{\sigma\}}(\overline{\rmX}_{\Pa}(p)\otimes{\overline{\FF}_{p}}, \rmR\Psi(\mathbb{V}))_{\frakm}
\end{equation*}
is given by the operator $u_{p}p$. 

\item The Frobenius action $\mathrm{Frob}_{p}$ on the Galois subquotient  
\begin{equation*}
\bigoplus_{\sigma\in\Sigma_{p}(\Pa(p))}\rmR^{3}\Phi_{\sigma}(\mathbb{V})_{\frakm}
\end{equation*}
is given by the operator $u_{p}$. 
\end{enumerate}
\end{lemma}
\begin{proof}
Examine the Weil--Deligne representation corresponding to $\rho_{\vec{\pi},\lambda}$ in \eqref{weil-del} and the shape of the monodromy operator \eqref{mono}, $\bigoplus\limits_{\sigma\in\Sigma_{p}(\Pa(p))}\mathrm{H}^{3}_{\{\sigma\}}(\rmX_{\Pa}(p)\otimes{\overline{\FF}_{p}}, \rmR\Psi(\mathbb{V}))_{\frakm}\otimes{E_{\lambda}}$ contributes to the Galois representation $\rho_{\vec{\pi},\lambda\mid_{\rmG_{\QQ_{p}}}}$ as the character $\chi\sigma\vert\cdot\vert^{-1}$. Since the local mondromy operator $\rmN_{\Sigma}$ induces an isomorphism between $\bigoplus\limits_{\sigma\in\Sigma_{p}(\Pa(p))}\rmR^{3}\Phi_{\sigma}(\mathbb{V})_{\frakm}(1)\otimes E_{\lambda}$ and $\bigoplus\limits_{\sigma\in\Sigma_{p}(\Pa(p))}\mathrm{H}^{3}_{\{\sigma\}}(\rmX_{\Pa}(p)_{\overline{\FF}_{p}}, \rmR\Psi(\mathbb{V}))_{\frakm}\otimes{E_{\lambda}}$. Then the result follows from Lemma \ref{Atkin-Lehner}. 
\end{proof}

\begin{myproof}{Theorem}{\ref{Mz-principle}}
We continue with the proof of the Mazur's principle. Recall that we have assumed that in the decomposition \ref{decomp}
\begin{equation}
\rmH^{3}_{\mathrm{c}}(\rmX_{\Pa}(p)\otimes\overline{\QQ}, \VV)_{\frakm}= \bigoplus_{\vec{\pi}}\rho^{\oplus m(\vec{\pi})}_{\vec{\pi}, \lambda},
\end{equation}
each $\vec{\pi}$ is ramified but $\overline{\rho}_{\vec{\pi}, \lambda}\vert_{\rmG_{\QQ_{p}}}\cong \overline{\rho}_{\pi, \lambda}\vert_{\rmG_{\QQ_{p}}}$ is unramified. We also have a decomposition 
\begin{equation}\label{mod-l-decomp}
\rmH^{3}_{\mathrm{c}}(\rmX_{\Pa}(p)\otimes\overline{\QQ}, \VV)_{\frakm}\otimes k= \overline{\rho}^{\oplus s}_{\pi, \lambda}
\end{equation}
for some positive integer $s$ by our semi-simple assumption. As the representation $\overline{\rho}_{\pi, \lambda}$ is unramified, the monodromy operator $\rmN\otimes k$ degenerates to $0$ on  $\rmH^{3}_{\mathrm{c}}(\overline{\rmX}_{\Pa}(p)\otimes\overline{\FF}_{p}, \rmR\Psi(\VV))_{\frakm}\otimes k$. The monodromy filtration 
\begin{equation*}
0\subset_{\mathrm{Gr}_{-1,\frakm}}\mathrm{F}_{-1}\mathrm{H}^{3}_{\rmc, \frakm}\subset_{\mathrm{Gr}_{0,\frakm}} \mathrm{F}_{0}\mathrm{H}^{3}_{\rmc,\frakm} \subset_{\mathrm{Gr}_{1,\frakm}} \mathrm{F}_{1}\mathrm{H}^{3}_{\rmc, \frakm}.
\end{equation*}
of  
\begin{equation*}
\mathrm{H}^{3}_{\rmc, \frakm}:=\mathrm{H}^{3}_{\rmc}(\overline{\rmX}_{\Pa}(p)\otimes{\overline{\FF}_{p}}, \rmR\Psi(\VV))_{\frakm}=\mathrm{H}^{3}_{\rmc}(\rmX_{\Pa}(p)\otimes{\overline{\QQ}_{p}}, \VV)
\end{equation*}
is determined by the Picard--Lefschetz formula and thus its successive quotient is given by
\begin{equation}\label{mono-fil}
\begin{aligned}
&\mathrm{Gr}_{-1,\frakm}= \Coker(\beta)=\bigoplus\limits_{\sigma\in\Sigma_{p}(\Pa(p))}\mathrm{H}^{3}_{\{\sigma\}}(\rmX_{\Pa}(p)\otimes{\overline{\FF}_{p}}, \rmR\Psi(\mathbb{V}))_{\frakm};\\
&\mathrm{Gr}_{0,\frakm}=\frac{\mathrm{H}^{3}_{\rmc}(\rmX_{\Pa}(p)\otimes{\overline{\FF}_{p}}, \VV)_{\frakm}}{\bigoplus\limits_{\sigma\in\Sigma_{p}(\Pa(p))}\mathrm{H}^{3}_{\{\sigma\}}(\rmX_{\Pa}(p)\otimes{\overline{\FF}_{p}}, \rmR\Psi(\mathbb{V}))_{\frakm}};\\
&\mathrm{Gr}_{1,\frakm}=\Ker(\alpha)=\bigoplus\limits_{\sigma\in\Sigma_{p}(\Pa(p))}\rmR^{3}\Phi_{\sigma}(\mathbb{V})_{\frakm}.\\
\end{aligned}
\end{equation}
 
By the torsion-freeness of $\mathrm{Gr}_{1,\frakm}=\bigoplus\limits_{\sigma\in\Sigma}\rmR^{3}\Phi_{\sigma}(\mathbb{V})_{\frakm}$, the natural exact sequence
\begin{equation*}
0\rightarrow \mathrm{Gr}_{0,\frakm}\rightarrow (\mathrm{H}^{3}_{\rmc,\frakm}/\mathrm{F}_{-1}\mathrm{H}^{3}_{\rmc, \frakm})\rightarrow \mathrm{Gr}_{1,\frakm}\rightarrow 0
\end{equation*}
induces the exact sequence
\begin{equation*}
0\rightarrow \mathrm{Gr}_{0,\frakm}\otimes k\rightarrow (\mathrm{H}^{3}_{\rmc,\frakm}/\mathrm{F}_{-1}\mathrm{H}^{3}_{\rmc, \frakm})\otimes k\rightarrow \mathrm{Gr}_{1,\frakm}\otimes k\rightarrow 0.
\end{equation*}
Since $\rmN\otimes k$ degenerates to zero,  $\mathrm{F}_{-1}\mathrm{H}^{3}_{\rmc, \frakm}$ is contained in $\lambda\mathrm{H}^{3}_{\rmc, \frakm}$. It follows then $\mathrm{Gr}_{1,\frakm}\otimes k$ contributes to each $\overline{\rho}_{\pi, \lambda} \vert_{\rmG_{\QQ_{p}}}$ in \ref{mod-l-decomp} as the two dimensional space where $\mathrm{Frob}_{p}$ acts by $\chi\sigma(p)p$ and $\chi\sigma(p)p^{2}$ (which are the Hecke parameters of the twisted Steinberg part $\chi\sigma\mathrm{St}_{\GL_{2}}$ of $\pi_{p}$). On the other hand, by Lemma \ref{cong-up-pa}, we known that $\mathrm{Frob}_{p}$ acts on  $\mathrm{Gr}_{1,\frakm}\otimes k$ by $\chi\sigma(p)p^{2}$ modulo $\lambda$. It then follows that $p\equiv 1 \mod l$ which is a contradiction. This finishes the proof of Theorem \ref{Mz-principle}.
\end{myproof}

\section{Quaternionic unitary Shimura variety of paramodular level}

\subsection{Quaternionic unitary Shimura variety} Let $\rmB=\rmB_{pq}$ be the indefinite quaternion algebra with discriminant $pq$ over $\QQ$. Here $p$ and $q$ are two distinct primes that are disctinct from $l$. Let $*$ be a \emph{neben-involution} on $\rmB$ defined as in \cite[A.4]{KR00} and $\mathcal{O}_{\rmB}$ be a maximal order fixed by $*$. We set $\rmV=\rmB\oplus \rmB$ and let 
\begin{equation*}
(\cdot, \cdot): \rmV\times \rmV\rightarrow \QQ
\end{equation*}
be the alternating form defined in \eqref{qua-herm}.  This form satisfies the following equation $$(av_{1}, v_{2})=(v_{1}, a^{*}v_{2})$$ for all $v_{1}, v_{2}\in \rmV$ and $a\in \rmB$.  Let $\rmG(\rmB)=\GU_{2}(\rmB)$ be the quaternionic unitrary group of degree $2$ defined as in \eqref{Global-quauni}. Since $\rmB$ splits over $\RR$. We have the identification $\rmG(\RR)\cong \GSp_{4}(\RR)$. Let $\mathbb{S}=\mathrm{Res}_{\CC/\RR}\mathbb{\rmG}_{m}$ be the Deligne torus and let $h: \mathbb{S}\rightarrow \rmG_{\RR}$ be the Hodge cocharacter that it induces the miniscule cocharacter 
\begin{equation*}
\mu: \mathbb{G}_{m, \CC}\rightarrow \mathrm{G}_{\CC}\cong \GSp(4)_{\CC}
\end{equation*}
that sends $z\in \mathbb{G}_{m, \CC}$ to  $\text{diag}(z,z,1,1)\in  \GSp(4)_{\CC}$. Let $\rmU$ be a neat open compact subgroup of $\GU_{2}(\rmB)(\mathbb{A}^{(\infty)})$. Consider the quaternionic unitary Shimura variety $\rmX_{\rmU}(\rmB)$  whose $\CC$-points is given by the symmetric space
\begin{equation}
\rmX_{\rmU}(\rmB)(\mathbb{C})=\rmG(\rmB)(\QQ)\backslash(\GSp_{4}(\RR)/\rmK^{h}\times \rmG(\rmB)(\mathbb{A}^{(\infty)}))/\rmU. 
\end{equation}
Then $\rmX_{\rmU}(\rmB)$ has a canonical model over $\QQ$.  Since $\GU_{2}(\rmB)(L)=\GSp_{4}(L)$ for a suitable quadratic extension $L/\QQ$, we can define a $\lambda$-adic local system $\VV_{\mu}$ exactly the same way as in \eqref{loc-sys} for each $\mu=(a, b; c)\in \rmX^{\ast}(\rmT)$ on $\rmX_{\rmU}(\rmB)$.

\subsection{Integral model of the quternionic unitary Shimura variety}
Let $\rmU^{pq}\subset \GSp_{4}(\mathbb{A}^{pq}_{f})$ be a sufficiently small subgroup. Let $\mathcal{L}_{p}=\mathcal{O}_{\rmB_{p}}\oplus\mathcal{O}_{\rmB_{p}}$ and let $\mathcal{L}_{q}=\mathcal{O}_{\rmB_{q}}\oplus\mathcal{O}_{\rmB_{q}}$. Then the stabilizers of these lattices in $\rmG(\QQ_{p})$ and $\rmG(\QQ_{q})$ are the paramodular subgroups denoted by $\Pa_{\rmD}(p)$ and $\Pa_{\rmD^{\prime}}(q)$ respectively in \S \ref{inner_form}. Then the tuple
\begin{equation}
(\rmB, \mathcal{O}_{\rmB}, \rmV, (\cdot, \cdot), \mu, \mathcal{L}_{p}, \rmU^{p}=\Pa_{\rmD^{\prime}}(q)\rmU^{pq})
\end{equation}
gives a PEL-data over $\ZZ_{p}$ and the tuple
\begin{equation}
(\rmB, \mathcal{O}_{\rmB}, \rmV, (\cdot, \cdot), \mu, \mathcal{L}_{q}, \rmU^{q}=\Pa_{\rmD}(p)\rmU^{pq})
\end{equation}
gives a PEL-data over $\ZZ_{q}$. They define a PEL-type moduli problem $\rmX(\rmB)$ over $\ZZ_{p}$ and over $\ZZ_{q}$. As the role of $p$ and $q$ are completely symmetric in this section, we will only state the moduli problem $\rmX(\rmB)$ over $\ZZ_{p}$. For each $\ZZ_{p}$ scheme $S$, an $S$-valued point of the functor $\rmX(\rmB)$ is given by the following data up to isomorphism:
\begin{itemize}
\item $\sfA$ is an abelian scheme over $S$ of dimension $4$;
\item $\iota: \mathcal{O}_{\rmB}\rightarrow \End_{S}(\sfA)$ is a ring homomorphism;
\item $\lambda: \sfA\rightarrow \sfA^{\vee}$ is a prime to $p$ polarization;
\item $\eta^{p}: \mathrm{H}_{1}(\sfA, \mathbb{A}^{(p\infty)})\cong \rmV\otimes \mathbb{A}^{(p\infty)} \mod \rmU^{p}$ is a $\rmU^{p}$-level structure that respects the bilinear forms on both sides up to a constant in $(\mathbb{A}^{(p\infty)})^{\times}$.
\end{itemize}
The quadruple $(\sfA, \iota, \lambda, \eta)$ is required to satisfy the following additional conditions:
\begin{itemize}
\item The Kottwitz condition: 
\begin{equation*}
\det(\rmT-\iota(a); \Lie(\sfA))=(\rmT^{2}-\Trd(a)\rmT+\Nrd(a))^{2}
\end{equation*} 
for all $a\in \calO_{\rmB}$;
\item The polarization $\lambda$ and $\iota$ are compatible in the sense 
\begin{equation*}
\lambda\circ\iota(a)\circ \lambda^{-1}=\iota(a^{*})^{\vee}.
\end{equation*}
\end{itemize}
This functor is representable by a quasi-projective scheme $\rmX(\rmB)$ over $\ZZ_{p}$ of relative dimension $3$. The special fiber of $\rmX(\rmB)$ will be denoted by $\overline{\rmX}(\rmB)$ and the same notation will be used for its base change to $\overline{\FF}_{p}$. We remark that $\rmG(\rmB)$ has semi-simple $\QQ$-rank one and therefore $\rmX(\rmB)$ is not proper. The boundary of the minimal compactification of  $\rmX(\rmB)$ consists of cusps. The scheme $\rmX(\rmB)$ is regular but not smooth over $\Spec(\ZZ_{p})$. Since the group $\rmG(\rmB)$ can be identified with $\GSp_{4}$ over $\ZZ_{p^{2}}$ and the parahoric subgroup $\Pa_{\rmD}(p)$ can be identified with $\Pa(p)$ over the same extension, the local model for $\rmX(\rmB)$ agrees with that of $\rmX_{\Pa}(p)$ after the base change to $\ZZ_{p^{2}}$. Therefore $\rmX(\rmB)$ also has isolated singularities which are ordinary and quadratic. The set of singular points $\overline{\rmX}^{\sing}(\rmB)$ is concentrated in the special fiber. In fact, we have the following lemma. 
\begin{lemma}\label{sing-quat}
The singular locus $\overline{\rmX}^{\sing}(\rmB)$ consists of those $(\sfA, \iota, \lambda, \eta^{p})\in \overline{\rmX}(\rmB)$ such that $\sfA$ is superspecial.
\end{lemma}
\begin{proof}
This follows from \cite[Lemma 2.1]{Wanga} and the discussion on local model following it. 
\end{proof}

\subsection{Supersingular locus of the quaternionic unitary Shimura variety}
Next we recall the description of the supersingular locus of $\overline{\rmX}(\rmB)$ using the Bruhat--Tits stratification in \cite{Wanga}. The description of the supersingular locus is also obtained by Oki \cite{Oki} independently using a different method. Let $\overline{\rmX}^{\ss}(\rmB)$ be the supersingular locus considered as a reduced closed subscheme of $\overline{\rmX}(\rmB)$. We fix a $p$-divisible group $\XX$ over $\overline{\FF}_{p}$ with an action $\iota: \calO_{\rmD}\rightarrow \End(\XX)$ and a polarization $\lambda: \XX\rightarrow \XX^{\vee}$. We assume that $\XX$ is isoclinic of slope $\frac{1}{2}$ and we denote by $\mathsf{N}$ its associated isocrystal. Let $b\in \rmB(\rmG(\rmB))$ be the element associated to $\mathsf{N}$ in the Kottwitz set $\rmB(\rmG(\rmB))$. Then the group $\rmJ_{b}$ defined as in \eqref{Jb} can be identified with the split $\GSp_{4}$ over $\QQ_{p}$ in this case. In fact, let $\rmC=\mathsf{N}^{\tau=1}$ with $\tau=\Pi^{-1}\rmF$ equipped with the restriction of the alternating form $(\cdot, \cdot)$ to $\rmC$. Then we have $\rmJ_{b}=\GSp_{4}(\rmC)$.

We consider the Rapoport--Zink space $\calM(\rmB)$ for the group $\rmG(\rmB)$. This is the set valued functor on $(\Nilp)$ which assigns each $S\in (\Nilp)$ the set of isomorphism classes of quadruples $(\sfX, \iota_{\sfX}, \lambda_{\sfX}, \rho_{\sfX})$ where:
\begin{itemize}
\item $\sfX$ is a $4$-dimensional $p$-divisible group over $S$;
\item $\iota_{\sfX}: \calO_{\rmD}\rightarrow \End_{S}(\sfX)$ is an action of $\calO_{\rmD}$ on $\sfX$;
\item $\lambda_{\sfX}: \sfX\rightarrow \sfX^{\vee}$ is a principal polarization;
\item $\rho_{\sfX}:  \mathbb{X}\times_{\overline{\FF}_{p}}\overline{S} \rightarrow \sfX\times_{S}\overline{S}$ is a quasi-isogeny where $\overline{S}$ is the special fiber of $S$ at $p$. 
\end{itemize}
This functor $\calM(\rmB)$ is representable by a formal scheme locally formally of finite type over $\Spf(\rmW_{0})$. It carries an action of $\rmJ_{b}(\QQ_{p})\cong \GSp_{4}(\QQ_{p})$.  Again $\calM(\rmB)$ can be decomposed into connected components 
\begin{equation}
\calM(\rmB)=\bigsqcup_{i\in\ZZ}\calM^{(i)}(\rmB)
\end{equation}
where $\calM^{(i)}(\rmB)$ classifies those $(\sfX, \iota_{\sfX}, \lambda_{\sfX}, \rho_{\sfX})$ such that $\rho_{\sfX}$ has height $i$. We denote by $\rmM(\rmB)$  the reduced subscheme of the formal scheme $\calM^{(0)}(\rmB)$. The scheme $\rmM(\rmB)$ can be described by the Bruhat--Tits stratification. A \emph{vertex lattice} $L$ in the symplectic space $\rmC$ is by definition a $\ZZ_{p}$-lattice $L\subset \rmC$ such that 
$pL^{\vee}\subset L\subset L^{\vee}$
where $L^{\vee}$ is the integral dual of $L$ in $\rmC$. Let $\calL_{\{i\}}$ be the set of vertex lattices of type $i\in\{0, 1, 2\}$, they are by definition given by
\begin{itemize}
\item[] $\calL_{\{0\}}=\{\text{vertex lattice $L_{0}$ of type $0$ in $\rmC$: } pL^{\vee}_{0}\subset^{4} L_{0}\subset^{0} L^{\vee}_{0}\}$;
\item[] $\calL_{\{2\}}=\{\text{vertex lattice $L_{2}$ of type $2$ in $\rmC$: }  pL^{\vee}_{2}\subset^{0} L_{2}\subset^{4} L^{\vee}_{2}\}$;
\item[] $\calL_{\{1\}}=\{\text{vertex lattice $L_{1}$ of type $1$ in $\rmC$: }  pL^{\vee}_{1}\subset^{2} L_{1}\subset^{2} L^{\vee}_{1}\}$.
\end{itemize}

Each vertex lattice gives rise to a maximal parahoric subgroup in $\GSp_{4}(\QQ_{p})$ by taking the stabilizer of the lattice. 
\begin{itemize}
\item A vertex lattice $L_{0}$ of type $0$ corresponds to a hyperspecial subgroup $\rmK_{\{0\}}$;
\item A vertex lattice $L_{2}$ of type $2$ corresponds to a hyperspecial subgroup $\rmK_{\{2\}}$. It is conjugate to $\rmK_{\{0\}}$. When we need to identify them, we write them commonly as $\rmH$;
\item A vertex lattices $L_{1}$ of type $1$ corresponds to the paramodular subgroup $\rmK_{\{1\}}=\Pa(p)$; 
\item A pair $(L_{0}, L_{2})$ of vertex lattices of type $0$ and type $2$ corresponds to the Siegel parahoric subgroup $\mathrm{Sie}(p)$; 
\item A pair $(L_{0}, L_{1})$ of vertex lattices of type $0$ and $1$ corresponds to the Klingen parahoric subgroup $\Kl(p)$.
\end{itemize}
Each vertex lattice gives rise to a \emph{lattice stratum} which is a projective subscheme of $\rmM(\rmB)$, see \cite[Theorem 5.1]{Wanga}.
\begin{itemize}
\item For a vertex lattice $L_{0}$ of type $0$, the scheme $\rmM(\rmB)\{L_{0}\}$ is a projective surface of the form 
\begin{equation*}
\rmZ^{p}_{3}\rmZ_{0}-\rmZ^{p}_{0}\rmZ_{3}+\rmZ^{p}_{2}\rmZ_{1}-\rmZ^{p}_{1}\rmZ_{2}=0
\end{equation*}
for a projective coordinate $[\rmZ_{0}: \rmZ_{1}: \rmZ_{2}: \rmZ_{3}]$ of $\PP^{3}$;

\item For a vertex lattice $L_{2}$ of type $2$, the scheme $\rmM(\rmB)\{L_{2}\}$ is a projective surface of the form 
\begin{equation*}
\rmZ^{p}_{3}\rmZ_{0}-\rmZ^{p}_{0}\rmZ_{3}+\rmZ^{p}_{2}\rmZ_{1}-\rmZ^{p}_{1}\rmZ_{2}=0
\end{equation*}
for a projective coordinate $[\rmZ_{0}:\rmZ_{1}:\rmZ_{2}:\rmZ_{3}]$ of $\PP^{3}$;

\item For a vertex lattice $L_{1}$ of type $1$, the scheme $\rmM(\rmB)\{L_{1}\}$ consists of a superspecial point. 
\end{itemize}

Therefore there are two types of irreducible components of $\rmM(\rmB)$. One of them is given by lattice strata $\rmM(\rmB)\{L_{0}\}$ of type $0$ and the other is given by lattice strata $\rmM(\rmB)\{L_{2}\}$ of type $2$. 

The Rapoport--Zink uniformization theorem \cite[Theorem 6.30]{RZ96} in this case gives rise to an isomorphism
\begin{equation}\label{RZ-unif-quat}
\rmX^{\ss}(\rmB)\cong \rmI(\QQ)\backslash \calM(\rmB)_{\mathrm{red}}\times \GU_{2}(\rmB)(\mathbb{A}^{(\infty p)})/\rmU^{p} .
\end{equation}
Here $\rmI(\QQ)\cong \GU_{2}(\overline{\rmB})(\QQ)$ and $\overline{\rmB}=\rmB_{q\infty}$ is the quaternion algebra over $\QQ$ which ramifies at $q$ and infinity. 

We define the {quaternionic unitary Shimura set} for the group $\rmG(\overline{\rmB})=\GU_{2}(\overline{\rmB})$ by the double coset space
\begin{equation*}
\rmZ_{\rmU}(\overline{\rmB})=\rmG(\overline{\rmB})(\QQ)\backslash \rmG(\overline{\rmB})(\mathbb{A}^{(\infty)})/\rmU
\end{equation*}
where $\rmU$ is an open compact subgroup of $\rmG(\overline{\rmB})(\mathbb{A}^{(\infty)})$.  These Shimura sets appear naturally when we discuss the structure of the supersingular locus of $\overline{\rmX}(\rmB)$. In this case, we will use another set of notations.

\begin{itemize}
\item  When $\rmU=\rmK_{\{i\}}\rmU^{p}$ for $i\in\{0,1,2\}$, we will denote the Shimura set $\rmZ_{\rmU}(\overline{\rmB})$ by $\rmZ_{\{i\}}(\overline{\rmB})$ for this level structure;
\item  We have an isomorphism $\rmZ_{\{0\}}(\overline{\rmB})\cong\rmZ_{\{2\}}(\overline{\rmB})$ and we will sometimes denote them commonly by $\rmZ_{\rmU}(\overline{\rmB})$;
\item Since $\rmK_{\{1\}}$ is also denoted by $\Pa(p)$, we will sometimes denote the Shimura set $\rmZ_{\{1\}}(\overline{\rmB})$ by $\rmZ_{\mathrm{\Pa}}(\overline{\rmB})$. 
\end{itemize}
The same construction \S $3.5$ defines an \'etale local system $\VV_{\mu}$ on the Shimura set $\rmZ_{\rmU}(\overline{\rmB})$ for each character $\mu\in\rmX^{\ast}(\rmT)$. The space of global sections $\rmH^{0}(\rmZ_{\rmU}(\overline{\rmB}), \VV_{\mu})$ will be denoted as $\calC(\rmZ_{\rmU}(\overline{\rmB}), \VV_{\mu})$. 

\begin{proposition}\label{B-double}
We have the following description of the supersingular  locus $\overline{\rmX}^{\ss}(\rmB)$ and the singular locus $\overline{\rmX}^{\sin}(\rmB)$ of $\overline{\rmX}(\rmB)$.
\begin{enumerate}
\item The irreducible components of $\overline{\rmX}^{\ss}(\rmB)$ can be parametrized by the union of the two Shimura sets $\rmZ_{\{0\}}(\overline{\rmB})$ and $\rmZ_{\{2\}}(\overline{\rmB})$.
\item The singular locus $\Sigma_{p}(\rmB)$ can be parametrized by the Shimura set $\rmZ_{\{1\}}(\overline{\rmB})$.
\end{enumerate}
\end{proposition}
\begin{proof}
Recall there are two kinds of irreducible components of $\rmM(\rmB)$ and they correspond to vertex lattices of type $0$ and vertex lattice of type $2$. Therefore by \eqref{RZ-unif-quat} the irreducible components of type $0$ can be parametrized by 
\begin{equation}
\rmI(\QQ)\backslash \rmJ_{b}(\QQ_{p})/{\rmK_{\{0\}}}\times \rmG(\rmB)(\mathbb{A}^{(\infty p)})/\rmU^{p}.
\end{equation}
One then sees that this is exactly 
\begin{equation*}
\rmZ_{\{0\}}(\overline{\rmB})=\rmG(\overline{\rmB})(\QQ)\backslash \rmG(\overline{\rmB})(\mathbb{A}^{(\infty)})/\rmK_{\{0\}}\rmU^{p}. 
\end{equation*}
Similarly, the irreducible components of type $2$ can be parametrized by 
\begin{equation}
\rmI(\QQ)\backslash \rmJ_{b}(\QQ_{p})/{\rmK_{\{2\}}}\times \rmG(\rmB)(\mathbb{A}^{(\infty p)})/\rmU^{p}.
\end{equation}
One then sees that this is exactly 
\begin{equation*}
\rmZ_{\{2\}}(\overline{\rmB})=\rmG(\overline{\rmB})(\QQ)\backslash \rmG(\overline{\rmB})(\mathbb{A}^{(\infty)})/\rmK_{\{2\}}\rmU^{p}.
\end{equation*}

By Lemma \ref{sing-quat}, the singular locus is exactly the superspecial locus which correspond to vertex lattices of type $1$. Therefore by  \eqref{RZ-unif-quat}, it is parametrized by the double coset
\begin{equation}
\rmI(\QQ)\backslash \rmJ_{b}(\QQ_{p})/{\Pa(p)}\times \rmG(\rmB)(\mathbb{A}^{(\infty p)})/\rmU^{p}.
\end{equation}
One then sees that this is exactly 
\begin{equation*}
\rmZ_{\{1\}}(\overline{\rmB})=\rmG(\overline{\rmB})(\QQ)\backslash \rmG(\overline{\rmB})(\mathbb{A}^{(\infty)})/\Pa(p)\rmU^{p}.
\end{equation*}
\end{proof}

Recall our notations for the quaternion algebras: $\rmB=\rmB_{pq}$, $\overline{\rmB}=\rmB_{q\infty}$ and $\rmB^{\prime}=\rmB_{p\infty}$. The following observation will be important in the subsequent discussions.
\begin{corollary}\label{pq-sing}
We have the following identifications of the singular loci of the Shimura varieties $\rmX_{\Pa}(pq)$ and $\rmX(\rmB)$ at the primes $p$ and $q$.
\begin{enumerate}
\item The singular locus $\Sigma_{p}(\Pa(pq))$ of $\overline{\rmX}_{\Pa}(pq)_{\overline{\FF}_{p}}$ is parametrized by the same set as the singular locus $\Sigma_{q}(\rmB)$ of $\overline{\rmX}(\rmB)_{\overline{\FF}_{q}}$. More precisely, we have
\begin{equation*}
\Sigma_{p}(\Pa(pq))=\rmZ_{\Pa}(\overline{\rmB})=\Sigma_{q}(\rmB).
\end{equation*}
\item The singular locus $\Sigma_{q}(\Pa(pq))$ of $\overline{\rmX}_{\Pa}(pq)_{\overline{\FF}_{q}}$ is parametrized by the same set as the singular locus $\Sigma_{p}(\rmB)$ of $\overline{\rmX}(\rmB)_{\overline{\FF}_{p}}$. More precisely, we have
\begin{equation*}
\Sigma_{q}(\Pa(pq))=\rmZ_{\Pa}(\rmB^{\prime})=\Sigma_{p}(\rmB).
\end{equation*}
\end{enumerate}
\end{corollary}

\begin{proof}
This follows easily from Proposition \ref{Siegel-double} and Proposition \ref{B-double}. 
\end{proof}

\section{Ribet's principle for paramodular Siegel threefold}
\subsection{The Ribet's principle} Now we come to state the second main theorem of this article which we call the Ribet's principle for the paramodular Siegel threefold.  We will make some similar assumptions as in the Mazur's principle concerning the residual Galois representation $\overline{\rho}_{\pi, \lambda}$. We fix a $\mu\in \rmX^{\ast}(\rmT)$ and denote $\VV_{\mu}$ simply by $\VV$. 

\begin{assumption}\label{B-assumptions}
We will make the following assumptions throughout this section.
\begin{enumerate}
\item The residual Galois representation $\overline{\rho}_{\pi, \lambda}$ is absolutely irreducible and the image of $\overline{\rho}_{\pi,\lambda}(\rmG_{\QQ})$ contains $\GSp_{4}(\FF_{l})$ .

\item Let $\frakm$ be the maximal ideal associated to $\overline{\rho}_{\pi, \lambda}$ as explained in \eqref{maximal-ideal}, we assume that
\begin{equation*}
\mathrm{H}^{i}_{\rmc}(\rmX_{\Pa}(pq)\otimes\overline{\QQ}, \mathbb{V})_{\frakm}
\end{equation*}
and
\begin{equation*}
\mathrm{H}^{i}_{\rmc}(\rmX(\rmB)\otimes\overline{\QQ}, \mathbb{V})_{\frakm}
\end{equation*}
vanishes unless $i=3$.

\item The cohomology group
$\mathrm{H}^{3}_{\rmc}(\rmX_{\Pa}(pq)\otimes\overline{\QQ}, \mathbb{V})_{\mathfrak{m}}\otimes k$
is semisimple as a Galois module. This implies that we have an identification of Galois modules
\begin{equation*}
\mathrm{H}^{3}_{\rmc}(\rmX_{\Pa}(pq)\otimes\overline{\QQ}, \mathbb{V})_{\mathfrak{m}}\otimes k=\overline{\rho}^{\oplus s}_{\pi,\lambda}
\end{equation*}
for some positive number $s$. 
\end{enumerate}
\end{assumption}

We introduce the following notion of generic Hecke parameters at $p$. This will be used when we apply some geometric results in \cite{Wang}. 

\begin{definition}\label{generic-at-p}
Let $\frakm$ be the maximal ideal corresponding to $\overline{\rho}_{\pi,\lambda}$. Then we say $\frakm$ is generic at $p$ if the Hecke parameters $[\alpha_{p}, \beta_{p}, p^{3}\beta^{-1}, p^{3}\alpha^{-1}_{p}]$ of $\pi_{p}$ satisfy 
\begin{equation*}
\alpha_{p}+\beta_{p}+p^{3}\alpha^{-1}_{p}+p^{3}\beta^{-1}_{p}\pm 2p(p+1) \not\equiv 0\mod \lambda.
\end{equation*}
If this is the case, we also say $\overline{\rho}_{\pi,\lambda}\vert_{\rmG_{\QQ_{p}}}$ is generic. 
\end{definition}

 \begin{theorem}\label{Rib-principle}
Let $\pi$ be a cuspidal automorphic representation of $\GSp_{4}(\mathbb{A})$ which is of general type with weight $(k_{1}, k_{2})$ satisfying $k_{1}\geq k_{2}\geq 3$ and $k_{1}\equiv k_{2}\mod 2$. Let $p, q$ be two distinct primes different from $l$. Let $\rmU=\Pa(p)\Pa(q)\rmU^{(pq)}$ with $\rmU^{(pq)}\subset \GSp_{4}(\mathbb{A}^{\infty pq})$ such that $\rmU$ is a neat open compact subgroup such that $\pi^{\rmU}\neq 0$. Suppose that $\pi_{p}$ and $\pi_{q}$ are both ramified of type $\mathrm{IIa}$ in the table of Schmidt. Suppose also that the residual Galois representation $\overline{\rho}_{\pi, \lambda}$ satisfies the following assumptions:
\begin{enumerate}
\item $\overline{\rho}_{\pi, \lambda}$ satisfies Assumption \ref{B-assumptions};
\item $\overline{\rho}_{\pi, \lambda}\vert_{\rmG_{\QQ_{p}}}$ is unramified and generic; 
\item $\overline{\rho}_{\pi, \lambda}\vert_{\rmG_{\QQ_{q}}}$ is ramified.
\end{enumerate}
Then there exists a cuspidal automorphic representation $\breve{\pi}$of general type with the same weight as $\pi$ such that 
\begin{equation*}
\overline{\rho}_{\pi, \lambda}\cong \overline{\rho}_{\breve{\pi}, \lambda} 
\end{equation*}
and such that $\breve{\pi}_{p}$ is an unramified principal series.
\end{theorem} 

\subsection{Tate cycles from the supersingular locus } In this subsection, we review a key ingredient in \cite{Wang} that the cohomology group $\rmH^{2}_{\rmc}(\overline{\rmX}(\rmB)\otimes \overline{\FF}_{p}, \VV(1))_{\frakm}$  is generated by Tate cycles coming from the supersingular locus.  We first review some notations that leads to the definition of the supersingular matrix as in \cite{Wang}. Recall that the irreducible components of the supersingular locus of $\overline{\rmX}(\rmB)_{\overline{\FF}_{p}}$ are parametrized by 
the Shimura sets $\rmZ_{\{0\}}(\overline{\rmB})$ and $\rmZ_{\{2\}}(\overline{\rmB})$ as in Proposition \ref{B-double}. Therefore we have two natural Gysin morphisms
\begin{equation}\label{Tate-inj}
\begin{aligned}
&\mathrm{inc}^{\{0\}}_{!}: \calC(\rmZ_{\{0\}}(\overline{\rmB}),\VV)\cong\rmH^{0}(\overline{\rmX}^{\mathrm{ss}}_{\{0\},\overline{\FF}_{p}}(\rmB), \VV)\rightarrow  \rmH^{2}_{\rmc}(\overline{\rmX}(\rmB)\otimes{\overline{\FF}_{p}}, \VV(1))\\
&\mathrm{inc}^{\{2\}}_{!}: \calC(\rmZ_{\{2\}}(\overline{\rmB}), \VV)\cong \rmH^{0}(\overline{\rmX}^{\mathrm{ss}}_{\{2\}, \overline{\FF}_{p}}(\rmB), \VV) \rightarrow  \rmH^{2}_{\rmc}(\overline{\rmX}(\rmB)\otimes{\overline{\FF}_{p}}, \VV(1))\\
\end{aligned}
\end{equation}
and the two natural restriction morphisms
\begin{equation}\label{Tate-surj}
\begin{aligned}
&\mathrm{inc}^{\ast}_{\{0\}}: \rmH^{4}(\overline{\rmX}(\rmB)\otimes \overline{\FF}_{p}, \VV(2))\rightarrow  \rmH^{4}(\overline{\rmX}^{\mathrm{ss}}_{\{0\},\overline{\FF}_{p}}(\rmB), \VV(2)) \cong \calC(\rmZ_{\{0\}}(\overline{\rmB}), \VV)\\
&\mathrm{inc}^{\ast}_{\{2\}}: \rmH^{4}(\overline{\rmX}(\rmB)\otimes \overline{\FF}_{p}, \VV(2))\rightarrow  \rmH^{4}(\overline{\rmX}^{\mathrm{ss}}_{\{2\},\overline{\FF}_{p}}(\rmB), \VV(2))\cong \calC(\rmZ_{\{2\}}(\overline{\rmB}), \VV).\\
\end{aligned}
\end{equation}

We have a natural restriction map
\begin{equation*}
\mathrm{inc}^{\ast}_{\{\mathrm{ss}\}}: \rmH^{2}_{\rmc}(\overline{\rmX}(\rmB)\otimes\overline{\FF}_{p}, \VV(1))\rightarrow \rmH^{2}(\overline{\rmX}^{\mathrm{ss}}(\rmB)\otimes\overline{\FF}_{p}, \VV(1)).
\end{equation*}

We also have the natural Gysin map
\begin{equation*}
\mathrm{inc}^{\{\mathrm{ss}\}}_{!}: \rmH^{2}(\overline{\rmX}^{\mathrm{ss}}(\rmB)\otimes\overline{\FF}_{p}, \VV(1))\rightarrow \rmH^{4}(\overline{\rmX}_{\Pa}(\rmB)\otimes\overline{\FF}_{p}, \VV(2)).
\end{equation*}

We can fit all the maps in the above discussions in the diagram below
\begin{equation*}
\begin{tikzcd}
 \calC(\rmZ_{\{0\}}(\overline{\rmB}), \VV) \arrow[rd, "\mathrm{inc}^{\{0\}}_{!}"'] &                                    &  \calC(\rmZ_{\{2\}}(\overline{\rmB}), \VV)\arrow[ld, "\mathrm{inc}^{\{2\}}_{!}"] \\
                   & \rmH^{2}_{\rmc}(\overline{\rmX}(\rmB)\otimes\overline{\FF}_{p}, \VV(1)) \arrow[d, "\mathrm{inc}^{\ast}_{\{\mathrm{ss}\}}"]                   &                   \\
                   & \rmH^{2}(\overline{\rmX}^{\mathrm{ss}}(\rmB)\otimes\overline{\FF}_{p}, \VV(1)) \arrow[d, "\mathrm{inc}^{\{\mathrm{ss}\}}_{!}"]                   &                   \\
                   &  \rmH^{4}(\overline{\rmX}(\rmB)\otimes\overline{\FF}_{p}, \VV(2)) \arrow[ld, "\mathrm{inc}^{\ast}_{\{0\}}"'] \arrow[rd, "\mathrm{inc}^{\ast}_{\{2\}}"] &                   \\
 \calC(\rmZ_{\{0\}}(\overline{\rmB}), \VV)          &                                    &   \calC(\rmZ_{\{2\}}(\overline{\rmB}), \VV)          
\end{tikzcd}
\end{equation*}
and obtain from the above diagram the following four maps
\begin{equation*}
\begin{aligned}
&\calT_{\mathrm{ss},\{00\}}= \mathrm{inc}^{\ast}_{\{0\}}\circ \mathrm{inc}^{\{\mathrm{ss}\}}_{!}\circ \mathrm{inc}^{\ast}_{\{\mathrm{ss}\}}\circ\mathrm{inc}^{\{0\}}_{!}\\
&\calT_{\mathrm{ss},\{02\}}= \mathrm{inc}^{\ast}_{\{2\}}\circ  \mathrm{inc}^{\{\mathrm{ss}\}}_{!}\circ \mathrm{inc}^{\ast}_{\{\mathrm{ss}\}}\circ\circ\mathrm{inc}^{\{0\}}_{!}\\
&\calT_{\mathrm{ss},\{20\}}= \mathrm{inc}^{\ast}_{\{0\}}\circ \mathrm{inc}^{\{\mathrm{ss}\}}_{!}\circ \mathrm{inc}^{\ast}_{\{\mathrm{ss}\}}\circ\mathrm{inc}^{\{2\}}_{!}\\
&\calT_{\mathrm{ss},\{22\}}= \mathrm{inc}^{\ast}_{\{2\}} \circ \mathrm{inc}^{\{\mathrm{ss}\}}_{!}\circ \mathrm{inc}^{\ast}_{\{\mathrm{ss}\}}\circ\circ\mathrm{inc}^{\{2\}}_{!}.\\
\end{aligned}
\end{equation*}
We will call the resulting matrix
\begin{equation*}
\calT_{\mathrm{ss}}=\begin{pmatrix} &\calT_{\mathrm{ss},\{00\}} & \calT_{\mathrm{ss}, \{02\}}\\ &\calT_{\mathrm{ss}, \{20\}}& \calT_{\mathrm{ss}, \{22\}}\\ \end{pmatrix}
\end{equation*}
the supersingular matrix. If we need to consider this matrix modulo $\frakm$, then we will denote it by $\calT_{\mathrm{ss}/\frakm}$. The entries of this matrix will be denoted without referring to $\frakm$.

\begin{proposition}\label{ss-matrix}
Suppose that $[\alpha_{p}, \beta_{p}, p^{3}\beta^{-1}, p^{3}\alpha^{-1}_{p}]$ is the Hecke parameter of $\pi$ at $p$. Then determinant of the supersingular matrix modulo $\frakm$ is given by
\begin{equation*}
\det\phantom{.}\calT_{\mathrm{ss}/\frakm}= \prod_{\rmu=\pm1}(\alpha_{p}+\beta_{p}+p^{3}\alpha^{-1}_{p}+p^{3}\beta^{-1}_{p}- 2\rmu p(p+1))^{2}.
\end{equation*}
\end{proposition}

\begin{proof}
This follows from \cite[Proposition 7.5]{Wang} with trivial modification to handle the higher weight. Note that in \cite{Wang}, we impose the condition $l\nmid p^{2}-1$ throughout, however this assumption is not needed in the proof here.
\end{proof}

\begin{theorem}\label{Tate}
Let $\frakm$ be a generic maximal ideal as in the above definition, we have the following statements.
\begin{enumerate}
\item The Gysin maps
\begin{equation*}
(\mathrm{inc}^{\{0\}}_{!, \frakm}+\mathrm{inc}^{\{2\}}_{!, \frakm}):\calC(\rmZ_{\{0\}}(\overline{\rmB}), \VV)_{\frakm}\oplus\calC(\rmZ_{\{2\}}(\overline{\rmB}), \VV)_{\frakm} \rightarrow  \rmH^{2}_{\rmc}(\overline{\rmX}_{\Pa}(\rmB)\otimes\overline{\FF}_{p}, \VV(1))_{\frakm}
\end{equation*}
is an isomorphism.
\item The restriction maps
\begin{equation*}
(\mathrm{inc}^{\ast}_{\{0\}, \frakm}+\mathrm{inc}^{\ast}_{\{2\}, \frakm}): \rmH^{4}(\overline{\rmX}_{\Pa}(\rmB)\otimes\overline{\FF}_{p}, \VV(2))_{\frakm}\rightarrow\calC(\rmZ_{\{0\}}(\overline{\rmB}), \VV)_{\frakm}\oplus\calC(\rmZ_{\{2\}}(\overline{\rmB}), \VV)_{\frakm}
\end{equation*}
is an isomorphism up to the torsion in  $\rmH^{4}(\overline{\rmX}_{\Pa}(\rmB)\otimes\overline{\FF}_{p}, \VV(2))_{\frakm}$.
\end{enumerate}
\end{theorem}
\begin{proof}
The injectivity of $(\mathrm{inc}^{\{0\}}_{!, \frakm}+\mathrm{inc}^{\{2\}}_{!, \frakm})$ follows from Proposition \ref{ss-matrix} and Nakayama's lemma under our generic assumption. The rest of the argument is the same as that of the first part of \cite[Proposition 7.9]{Wang} comparing the dimensions of newforms and oldforms with the dimensions of the target and the source of $(\mathrm{inc}^{\{0\}}_{!, \frakm}+\mathrm{inc}^{\{2\}}_{!, \frakm})$. Note again that $p\equiv 1\mod l$ is allowed in this argument but that the Hecke parameters being generic is really needed.
\end{proof}

\subsection{Proof of Ribet's principle} Let $\frakm$ be the maximal ideal corresponding to $\overline{\rho}_{\pi, \lambda}$ as in \eqref{maximal-ideal}. Then by our semisimple assumption, we have 
\begin{equation}\label{pq-semi-a}
\rmH^{3}_{\rmc}(\rmX_{\Pa}(pq)\otimes\overline{\QQ}, \VV)_{\frakm}\otimes k=\overline{\rho}^{\oplus s}_{\pi, \lambda}
\end{equation}
for some positive integer $s$. By the global Jacquet--Langlands correspondence for $\GSp_{4}$ \cite{RW19}, we have a non-canonical injection
\begin{equation*}
\rmH^{3}_{\rmc}(\rmX(\rmB)\otimes{\overline{\QQ}}, \VV_{E_{\lambda}})\hooklongrightarrow \rmH^{3}_{\rmc}(\rmX_{\Pa}(pq)\otimes \overline{\QQ}, \VV_{E_{\lambda}}). 
\end{equation*}
The subspace $\rmH^{3}_{\rmc}(\rmX(\rmB)\otimes{\overline{\QQ}}, \VV_{E_{\lambda}})$ is preserved by the Hecke algebra $\mathbb{T}$, the Atkin--Lehner operators $\rmu_{p}, \rmu_{q}$ and the Galois action. By Assumption \ref{B-assumptions}, the cohomology groups $\rmH^{3}_{\rmc}(\rmX(\rmB)\otimes{\overline{\QQ}}, \VV)$ and  $\rmH^{3}_{\rmc}(\rmX_{\Pa}(pq)\otimes \overline{\QQ}, \VV)$ are both torsion free. Thus the above injection upgrades to
\begin{equation*}
\rmH^{3}_{\rmc}(\rmX(\rmB)\otimes{\overline{\QQ}}, \VV)\hooklongrightarrow \rmH^{3}_{\rmc}(\rmX_{\Pa}(pq)\otimes \overline{\QQ}, \VV). 
\end{equation*}

Therefore the semisimple assumption of $\rmH^{3}_{\rmc}(\rmX_{\Pa}(pq)\otimes \overline{\QQ}, \VV)\otimes k$ also implies that of $\rmH^{3}_{\rmc}(\rmX(\rmB)\otimes\overline{\QQ}, \VV)_{\frakm}\otimes k$ and thus 
\begin{equation}\label{multi-t-a}
\rmH^{3}_{\rmc}(\rmX(\rmB)\otimes\overline{\QQ}, \VV)_{\frakm}\otimes k=\overline{\rho}_{\pi, \lambda}^{\oplus t}
\end{equation}
for some non-negative integer $t$. Note that $t$ is positive in our case, since in the decomposition 
\begin{equation}\label{decomp-B-ribet}
\rmH^{3}_{\rmc}(\rmX(\rmB)\otimes{\overline{\QQ}}, \VV)_{\frakm}\otimes E_{\lambda}=\bigoplus_{\vec{\pi}}\rho^{m(\vec{\pi})}_{\vec{\pi}, \lambda}
\end{equation}
where $\vec{\pi}$ runs through all the representations that is congruent to $\pi$, there is at least one $\vec{\pi}$ such that $\vec{\pi}_{p}$ and $\vec{\pi}_{q}$ are both ramified. Our strategy is to prove the theorem by contradiction
and therefore we assume that all the representation $\vec{\pi}$ appear in the following equations
\begin{equation}\label{decomp-B-para}
\begin{split}
&\rmH^{3}_{\rmc}(\rmX(\rmB)\otimes\overline{\QQ}, \VV)_{\frakm}\otimes E_{\lambda}=\bigoplus_{\vec{\pi}}\rho^{n(\vec{\pi})}_{\vec{\pi}, \lambda}\\
&\rmH^{3}_{\rmc}(\rmX_{\Pa}(pq)\otimes\overline{\QQ}, \VV)_{\frakm}\otimes E_{\lambda}=\bigoplus_{\vec{\pi}}\rho^{m(\vec{\pi})}_{\vec{\pi}, \lambda}\\
\end{split}
\end{equation}
are \emph{ramified} at $p$.  As we have assumed that $\overline{\rho}_{\pi, \lambda}$ is ramified at $q$, all $\vec{\pi}_{q}$ is necessarily ramified. Therefore it follows that in particular $t$ in \eqref{multi-t-a} is positive.

\begin{proposition}
Suppose that all the $\vec{\pi}$ are ramified in \eqref{decomp-B-para}. The map $\mathrm{ca}$
\begin{equation*}
\mathrm{ca}: \mathrm{H}^{3}(\overline{\rmX}(\rmB)\otimes{\overline{\FF}_{p}}, \rmR\Psi(\mathbb{V}))_{\frakm}\otimes k \rightarrow  \bigoplus_{\sigma\in\Sigma_{p}(\rmB)} \rmR^{3}\Phi_{\sigma}(\mathbb{V})_{\frakm}\otimes k
\end{equation*}
localized at $\mathfrak{m}$ is surjective and by duality the map $\mathrm{coca}$
\begin{equation*}
\mathrm{coca}: \bigoplus_{\sigma\in\Sigma_{p}(\rmB)} \mathrm{H}^{3}_{\{\sigma\}}(\overline{\rmX}(\rmB)\otimes{\overline{\FF}_{p}}, \rmR\Psi(\mathbb{V}))_{\frakm} \otimes k \rightarrow \mathrm{H}^{3}_{\rmc}(\overline{\rmX}(\rmB)\otimes{\overline{\FF}_{p}}, \rmR\Psi(\mathbb{V}))_{\frakm}\otimes k
\end{equation*}
localized at $\mathfrak{m}$ is injective. 
\end{proposition}

\begin{proof}
Note that the kernel of
\begin{equation*}
\mathrm{coca}: \bigoplus_{\sigma\in\Sigma_{p}(\rmB)} \mathrm{H}^{3}_{\{\sigma\}}(\overline{\rmX}(\rmB)\otimes{\overline{\FF}_{p}}, \rmR\Psi(\mathbb{V}))_{\frakm}  \rightarrow \mathrm{H}^{3}_{\rmc}(\overline{\rmX}(\rmB)\otimes{\overline{\FF}_{p}}, \rmR\Psi(\mathbb{V}))_{\frakm},
\end{equation*}
is given by $\rmH^{2}_{\rmc}(\overline{\rmX}_{\Pa}(\rmB)\otimes \overline{\FF}_{p}, \VV)_{\frakm}$ and we have an isomorphism 
\begin{equation*}
 \rmH^{2}_{\rmc}(\overline{\rmX}_{\Pa}(\rmB)\otimes\overline{\FF}_{p}, \VV(1))_{\frakm}\cong \calC(\rmZ_{\rmH}(\overline{\rmB}), \VV)_{\frakm}\oplus\calC(\rmZ_{\rmH}(\overline{\rmB}), \VV)_{\frakm} 
\end{equation*}
by Theorem \ref{Tate}.
But the right hand side supports only unramified representations and therefore vanishes by our assumption in the statement of this proposition. Since 
\begin{equation*}
\bigoplus\limits_{\sigma\in\Sigma_{p}(\rmB)} \mathrm{H}^{3}_{\{\sigma\}}(\overline{\rmX}(\rmB)\otimes{\overline{\FF}_{p}}, \rmR\Psi(\mathbb{V}))_{\frakm}
\end{equation*}
and $\mathrm{H}^{3}_{\rmc}(\overline{\rmX}(\rmB)\otimes{\overline{\FF}_{p}}, \rmR\Psi(\mathbb{V}))_{\frakm}$ are both torsion free, 
\begin{equation*}
\mathrm{coca}: \bigoplus_{\sigma\in\Sigma_{p}(\rmB)} \mathrm{H}^{3}_{\{\sigma\}}(\overline{\rmX}(\rmB)\otimes{\overline{\FF}_{p}}, \rmR\Psi(\mathbb{V}))_{\frakm}\otimes k  \rightarrow \mathrm{H}^{3}_{\rmc}(\overline{\rmX}(\rmB)\otimes{\overline{\FF}_{p}}, \rmR\Psi(\mathbb{V}))_{\frakm}\otimes k
\end{equation*}
is also an injection. Similarly, the cokernel of 
\begin{equation*}
\mathrm{ca}: \mathrm{H}^{3}(\overline{\rmX}(\rmB)\otimes{\overline{\FF}_{p}}, \rmR\Psi(\mathbb{V}))_{\frakm}\otimes k \rightarrow  \bigoplus\limits_{\sigma\in\Sigma_{p}(\rmB)} \rmR^{3}\Phi_{\sigma}(\mathbb{V})_{\frakm}\otimes k 
\end{equation*}
is given by $\rmH^{4}(\overline{\rmX}_{\Pa}(\rmB), \VV_{k})_{\frakm}$.  Note that, we have 
\begin{equation*}
\begin{aligned}
&\rmH^{4}(\overline{\rmX}(\rmB)\otimes\overline{\FF}_{p}, \VV_{k})\xrightarrow{\sim} \rmH^{4}(\overline{\rmX}^{\mathrm{sm}}(\rmB)\otimes\overline{\FF}_{p}, \VV_{k})\\
&\rmH^{2}_{\rmc}(\overline{\rmX}^{\mathrm{sm}}(\rmB)\otimes\overline{\FF}_{p}, \VV_{k})\xrightarrow{\sim} \rmH^{2}_{\rmc}(\overline{\rmX}(\rmB)\otimes\overline{\FF}_{p}, \VV_{k})\\
\end{aligned}
\end{equation*}
by considering the excision long exact sequence. Then Poincar\'e duality for $\overline{\rmX}^{\mathrm{sm}}(\rmB)$ implies that $\rmH^{4}(\overline{\rmX}_{\Pa}(\rmB)\otimes\overline{\FF}_{p}, \VV_{k})_{\frakm}$ also vanishes.
\end{proof}

\begin{lemma}\label{cong-up}
Suppose that each $\vec{\pi}$ appearing in \eqref{decomp-B-ribet} is ramified at $p$. 
\begin{enumerate}
\item  The Frobenius action $\mathrm{Frob}_{p}$ on the Galois submodule 
\begin{equation*}
\bigoplus\limits_{\sigma\in\Sigma_{p}(\rmB)}\mathrm{H}^{3}_{\{\sigma\}}(\rmX(\rmB)\otimes{\overline{\FF}_{p}}, \rmR\Psi(\mathbb{V}))_{\frakm}
\end{equation*}
is given by $\rmu_{p}p$ where $\rmu_{p}$ is the Atkin--Lehner operator as in \eqref{AL}. 

\item The Frobenius action $\mathrm{Frob}_{p}$ on the Galois subquotient  
\begin{equation*}
\bigoplus_{\sigma\in\Sigma_{p}(\rmB)}\rmR^{3}\Phi_{\sigma}(\mathbb{V})_{\frakm}
\end{equation*}
is given by the Atkin--Lehner operator $\rmu_{p}$. 
\end{enumerate}
\end{lemma}
\begin{proof}
This follows by exactly the same proof for Lemma \ref{cong-up-pa}.
\end{proof}

\begin{proposition}\label{2t-degen}
Suppose that each representation $\vec{\pi}$ appearing in \eqref{decomp-B-ribet} is ramified at $p$, then we have
\begin{equation*}
\dim_{k} \bigoplus\limits_{\sigma\in\Sigma_{p}(\rmB)}\rmR^{3}\Phi_{\sigma}(\mathbb{V})_{\frakm}\otimes k=2t.
\end{equation*}
\end{proposition}

\begin{proof}
Recall that we have 
\begin{equation}\label{multi-t}
\rmH^{3}_{\rmc}(\rmX(\rmB)\otimes\overline{\QQ}, \VV)_{\frakm}\otimes k=\overline{\rho}_{\pi, \lambda}^{\oplus t}
\end{equation}
and $\overline{\rho}_{\pi, \lambda} \vert_{\rmG_{\QQ_{p}}}$ is unramified by our assumption. Therefore the monodromy operator $\rmN\otimes k$ degenerates to $0$ on $\mathrm{H}^{3}_{\rmc}(\rmX(\rmB)\otimes{\overline{\FF}_{p}}, \rmR\Psi(\VV))_{\frakm}$. The monodromy filtration 
\begin{equation*}
0\subset_{\mathrm{Gr}_{-1,\frakm}}\mathrm{F}_{-1}\mathrm{H}^{3}_{\rmc, \frakm}\subset_{\mathrm{Gr}_{0,\frakm}} \mathrm{F}_{0}\mathrm{H}^{3}_{\rmc,\frakm} \subset_{\mathrm{Gr}_{1,\frakm}} \mathrm{F}_{1}\mathrm{H}^{3}_{\rmc, \frakm}.
\end{equation*}
of  $\mathrm{H}^{3}_{\rmc, \frakm}:=\mathrm{H}^{3}_{\rmc}(\rmX(\rmB)\otimes{\overline{\FF}_{p}}, \rmR\Psi(\VV))_{\frakm}=\mathrm{H}^{3}_{\rmc}(\rmX(\rmB)\otimes{\overline{\QQ}_{p}}, \VV)_{\frakm}$ is determined by the Picard--Lefschetz formula
and the successive quotient of this filtration is given by
\begin{equation}\label{mono-fil}
\begin{aligned}
&\mathrm{Gr}_{-1,\frakm}= \Coker(\beta)=\bigoplus\limits_{\sigma\in\Sigma_{p}(\rmB)}\mathrm{H}^{3}_{\{\sigma\}}(\rmX(\rmB)\otimes{\overline{\FF}_{p}}, \rmR\Psi(\mathbb{V}))_{\frakm}\\
&\mathrm{Gr}_{0,\frakm}= \frac{\mathrm{H}^{3}_{\rmc}(\rmX(\rmB)\otimes{\overline{\FF}_{p}}, \VV)_{\frakm}}{\bigoplus_{\sigma\in\Sigma_{p}(\rmB)}\mathrm{H}^{3}_{\{\sigma\}}(\rmX(\rmB)\otimes{\overline{\FF}_{p}}, \rmR\Psi(\mathbb{V}))_{\frakm}}\\
&\mathrm{Gr}_{1,\frakm}=\Ker(\alpha)(-1)=\bigoplus\limits_{\sigma\in\Sigma_{p}(\rmB)}\rmR^{3}\Phi_{\sigma}(\mathbb{V})_{\frakm}.\\
\end{aligned}
\end{equation}
By definition and the torsion-freeness of $\mathrm{Gr}_{1, \frakm}$, we have an exact sequence
\begin{equation*}
0\rightarrow \mathrm{Gr}_{0,\frakm}\otimes k\rightarrow (\mathrm{H}^{3}_{\rmc,\frakm}/\mathrm{F}_{-1}\mathrm{H}^{3}_{\rmc, \frakm})\otimes k\rightarrow \mathrm{Gr}_{1,\frakm}\otimes k\rightarrow 0.
\end{equation*}
Since $\rmN\otimes k$ is zero on $\rmH^{3}_{\rmc, \frakm}$, $\mathrm{F}_{-1}\mathrm{H}^{3}_{\rmc, \frakm}$ is contained in $\lambda\mathrm{H}^{3}_{\rmc, \frakm}$. Then $\mathrm{Gr}_{1,\frakm}\otimes k$  has to contribute to each $\overline{\rho}_{\pi, \lambda }\vert_{\rmG_{\QQ_{p}}}$ in \ref{multi-t-a} as a two dimensional space.
It then follows that 
\begin{equation}
\dim_{k} \bigoplus_{\sigma\in\Sigma_{p}(\rmB)}\rmR^{3}\Phi_{\sigma}(\mathbb{V})_{\frakm}\otimes k=2t.
\end{equation}
\end{proof}

\begin{proposition}
We have an equality
\begin{equation}
\dim_{k} \bigoplus_{\sigma\in\Sigma_{q}(\Pa(pq))}\rmR^{3}\Phi_{\sigma}(\mathbb{V})_{\frakm}\otimes k=s
\end{equation}
where $s$ is the multiplicity in \ref{pq-semi-a}.
\end{proposition}
\begin{proof}
Consider the isomorphism \eqref{pq-semi-a} 
\begin{equation}
\rmH^{3}_{\mathrm{c}}(\rmX_{\Pa}(pq)\otimes\overline{\QQ}_{q}, \VV)_{\frakm}\otimes k= \overline{\rho}^{\oplus s}_{\pi, \lambda}\vert_{\rmG_{\QQ_{q}}}
\end{equation}
where all $\overline{\rho}_{\pi, \lambda} \vert_{\rmG_{\QQ_{q}}}$ is ramified by our assumption. 
By the same argument as above we find that
\begin{equation}
\begin{aligned}
&\mathrm{Gr}_{-1,\frakm}\otimes k= \Coker(\beta\otimes k)=\bigoplus\limits_{\sigma\in\Sigma_{q}(\Pa(pq))}\mathrm{H}^{3}_{\{\sigma\}}(\rmX_{\Pa}(pq)\otimes{\overline{\FF}_{q}}, \rmR\Psi(\mathbb{V}))_{\frakm}\otimes k\\
&\mathrm{Gr}_{0,\frakm}\otimes k= \frac{\mathrm{H}^{3}_{\rmc}(\rmX_{\Pa}(pq)\otimes{\overline{\FF}_{q}}, \VV)_{\frakm}\otimes k}{\bigoplus\limits_{\sigma\in\Sigma_{q}(\Pa(pq))}\mathrm{H}^{3}_{\{\sigma\}}(\rmX_{\Pa}(pq)\otimes{\overline{\FF}_{q}}, \rmR\Psi(\mathbb{V}))_{\frakm}\otimes k}\\
&\mathrm{Gr}_{1,\frakm}\otimes k=\Ker(\alpha\otimes k)(-1)=\bigoplus\limits_{\sigma\in\Sigma_{q}(\Pa(pq))}\rmR^{3}\Phi_{\sigma}(\mathbb{V})_{\frakm}\otimes k\\
\end{aligned}
\end{equation}
give the graded piece of the monodromy filtration for $\rmH^{3}_{\mathrm{c}}(\rmX_{\Pa}(pq)\otimes{\overline{\FF}_{q}}, \rmR\Psi(\VV))_{\frakm}\otimes k$. 

Since $\overline{\rho}_{\pi, \lambda}\vert_ {\rmG_{\QQ_{q}}}$ is ramified, $\rmN\otimes k$ does not degenerate and has the same rank as $\rmN$ which is one. Therefore 
\begin{equation*}
\bigoplus\limits_{\sigma\in\Sigma_{q}(\Pa(pq))}\mathrm{H}^{3}_{\{\sigma\}}(\rmX_{\Pa}(pq)\otimes{\overline{\FF}_{q}}, \rmR\Psi(\mathbb{V}))_{\frakm}\otimes k
\end{equation*}
contributes to each $\overline{\rho}_{\pi, \lambda} \vert_{\rmG_{\QQ_{q}}}$ as an one dimensional subspace. It follows then that 
\begin{equation*}
\bigoplus_{\sigma\in\Sigma_{q}(\Pa(pq))}\rmR^{3}\Phi_{\sigma}(\mathbb{V})_{\frakm}\otimes k 
\end{equation*}
contributes to $\overline{\rho}_{\pi, \lambda}\vert_{\rmG_{\QQ_{q}}}$ as an one dimensional quotient. Then we have
\begin{equation*}
\dim_{k} \bigoplus_{\sigma\in\Sigma_{q}(\Pa(pq))}\rmR^{3}\Phi_{\sigma}(\mathbb{V})_{\frakm}\otimes k=s.
\end{equation*}
\end{proof}
\begin{myproof}{Theorem}{\ref{Rib-principle}}
Recall that in Corollary \ref{pq-sing}, we have proved that $\Sigma_{p}(\rmB)$ and $\Sigma_{q}(\Pa(pq))$ can be identified. Thus we have $s=2t$ by the previous two propositions. Note also that if we switch the role of $p$ and $q$ in the previous discussions, we will obtain that
\begin{equation*}
\begin{split}
&\dim_{k} \bigoplus_{\sigma\in\Sigma_{p}(\Pa(pq))}\rmR^{3}\Phi_{\sigma}(\mathbb{V})_{\frakm}\otimes k=2s;\\
&\dim_{k} \bigoplus_{\sigma\in\Sigma_{q}(\rmB)}\rmR^{3}\Phi_{\sigma}(\mathbb{V})_{\frakm}\otimes k=t.\\
\end{split}
\end{equation*}
By Corollary \ref{pq-sing} again, we have $t=2s$. Therefore we obtain that $s=t=0$ which is obviously impossible. Thus Ribet's principle is proved. 
\end{myproof}

\section{Main level lowering theorem}
In this section we will prove our main level lowering theorem.  We will first review an arithmetic level raising theorem for a cuspidal automorphic representation of $\GSp_{4}(\mathbb{A})$ of general type proved in \cite{Wang}.

\subsection{Arithmetic level raising for $\GSp_{4}$}
Let $\pi$ be a cuspidal automorphic representation of $\GSp_{4}(\mathbb{A})$ of general type which is of weight $(k_{1}, k_{2})$ satisfying $k_{1}\geq k_{2}\geq 3$ and $k_{1}\equiv k_{2}\mod 2$. Let $\Sigma_{\pi}$ be the set of non-archimedean places at which $\pi$ is ramified.  We fix a set of non-archimedean places $\Sigma_{\mathrm{min}}$ containing $\Sigma_{\pi}-\{p, q\}$. Let $\Sigma_{\mathrm{lr}}$ be another set of non-achimedean places that is disjoint from $\Sigma_{\mathrm{min}}$ but contains $\{p, q\}$. To state our level raising result, we first recall the definition of the rigidity of the residual Galois representation $\overline{\rho}_{\pi,\lambda}$. 

\begin{definition}\label{rigid}
We say the residual Galois representation $\overline{\rho}_{\pi, \lambda}: \rmG_{\QQ}\rightarrow\GSp_{4}(k)$ is rigid for $(\Sigma_{\mathrm{min}}, \Sigma_{\mathrm{lr}})$ if the following conditions are satisfied.
\begin{itemize}
\item For every $v\in\Sigma_{\min}$, every lifting of $\overline{\rho}_{\pi,\lambda}\vert_{\rmG_{\QQ_{v}}}$ is minimally ramified in the sense of \cite[Definition 3.4]{Wangd};
\item For every $v\in\Sigma_{\lr}$, the generalized eigenvalues of $\overline{\rho}_{\pi,\lambda}(\mathrm{Fr}_{v})$ contain the pair $\{v, v^{2}\}$ exactly once;
\item At $v=l$, $\overline{\rho}_{\pi,\lambda}\vert_{\rmG_{\QQ_{v}}}$ is regular Fontaine-Laffaille crystalline as in \cite[Definition 3.10]{Wangd};
\item For $v\not\in\Sigma_{\min}\cup\Sigma_{\lr}\cup\{l\}$, the representation $\overline{\rho}_{\pi,\lambda}\vert_{\rmG_{\QQ_{v}}}$ is unramified.
\end{itemize}
\end{definition}

Given $\Sigma_{\mathrm{lr}}\cup\Sigma_{\mathrm{min}}$, we assume in this section that the level $\rmU$ for the Shimura variety $\rmX(\rmB)$ is given in the following way.
\begin{itemize}
\item For $v\not\in\Sigma_{\lr}\cup\Sigma_{\min}$ or $v=l$, then $\rmU_{v}$ is hyperspecial;
\item For $v\in\Sigma_{\lr}-\{p, q\}$, then $\rmU_{v}$ is the paramodular subgroup of $\GSp_{4}(\ZZ_{v})$;
\item For $v\in\Sigma_{\min}$, then $\rmU_{v}$ is contained in the pro-$v$ Iwahori subgroup $\Iw_{1}(v)$;
\item For $v=p, q$, then  $\rmU_{v}$ is the paramodular subgroup  $\mathrm{Pa}_{\rmD}(p)$ of $\GU_{2}(\rmD)$ and $\mathrm{Pa}_{\rmD^{\prime}}(q)$ of $\GU_{2}(\rmD^{\prime})$ respectively.
\end{itemize}
We also choose the level $\rmU$ for $\rmX_{\Pa}(pq)$ in an obvious compatible way as above: for the  last point, $\rmU_{v}$ is the paramodular subgroup  $\mathrm{Pa}(v)$ for $v=p, q$ and we keep the rest of the choices. 

Given $\Sigma_{\mathrm{lr}}\cup\Sigma_{\mathrm{min}}$, we assume in this section that the level $\rmU$ for the Shimura set $\rmZ_{\rmH}(\rmB)$ is given in the following way.
\begin{itemize}
\item For $v\not\in\Sigma_{\lr}\cup\Sigma_{\min}$ or $v=l$ or $v=p$, then $\rmU_{v}$ is hyperspecial;
\item For $v\in\Sigma_{\lr}-\{pq\}$, then $\rmU_{v}$ is the paramodular subgroup of $\GSp(\ZZ_{v})$;
\item For $v\in\Sigma_{\min}$, then $\rmU_{v}$ is contained in the pro-$v$ Iwahori subgroup $\Iw_{1}(v)$;
\item For $v=q$, then  $\rmU_{v}$ is the paramodular subgroup $\mathrm{Pa}_{\rmD^{\prime}}(q)$ of $\GU_{2}(\rmD^{\prime})$.
\end{itemize}
We also choose the level $\rmU$ for $\rmZ_{\rmH}(\rmB^{\prime})$ in an obvious compatible way as above: For $v=q$ we choose $\rmU_{v}$ to be the hyperspecial subroup and for $v=p$ we choose $\rmU_{v}$ to be the paramodular subgroup $\mathrm{Pa}_{\rmD}(p)$. 

\begin{definition}
Suppose $p$ is a prime such that $\pi_{p}$ is an unramified principal series representation and $m$ is a positive integer. We say $p$ is level raising special for $\pi$ of depth $m$ if 
\begin{enumerate}
\item we have $l\nmid p^{2}-1$;
\item the Hecke parameters $[\alpha_{p}, \beta_{p}, p^{3}\beta^{-1}, p^{3}\alpha^{-1}_{p}]$ of $\pi_{p}$ satisfy simultaneously the following condiitons, 
\begin{equation*}
\begin{aligned}
&\beta_{p}+p^{3}\beta^{-1}_{p}\equiv \rmu (p+ p^{2})\mod \lambda \text{ for some $\rmu\in\{\pm1\}$}\\
&\alpha_{p}+p^{3}\alpha^{-1}_{p}\not\equiv \pm (p+ p^{2})\mod \lambda\\
&\alpha_{p}+\beta_{p}+p^{3}\beta^{-1}_{p}+p^{3}\alpha^{-1}_{p}\not\equiv \pm2(p+ p^{2}) \mod \lambda\\
\end{aligned}
\end{equation*} 
and such that the $\lambda$-adic valuation of $\beta_{p}+p^{3}\beta^{-1}_{p}-\rmu (p+ p^{2})$ is exactly $m$.
\end{enumerate}
\end{definition}

Now we can state our arithmetic level raising theorem for $\pi$ in \cite{Wang} at a level raising special prime as above. 

\begin{theorem}\label{arithmetic-level-raising}
Suppose that $\pi$ is a cuspidal automorphic representation of $\GSp_{4}(\mathbb{A})$ of general type with weight $(k_{1}, k_{2})$ satisfying $k_{1}\geq k_{2}\geq 3$ and $k_{1}\equiv k_{2}\mod 2$ and with trivial central character. 
Let $p$ be a level raising special prime for $\pi$ of depth $m$. We assume further that
\begin{enumerate}
\item the prime $l+1\geq k_{1}+k_{2}$;
\item $\overline{\rho}_{\pi,\lambda}$ is rigid for $(\Sigma_{\min}, \Sigma_{\lr})$ as in Definition \ref{rigid};
\item $\overline{\rho}_{\pi,\lambda}$ is absolutely irreducible and the image of $\overline{\rho}_{\pi,\lambda}(\rmG_{\QQ})$ contains $\GSp_{4}(\FF_{l})$;
\item The cohomology group  $\rmH^{i}_{\rmc}(\overline{\rmX}(\rmB)\otimes\overline{\FF}_{p}, \rmR\Psi(\VV))_{\frakm}$ of the Shimura variety $\rmX(\rmB)$ is concentrated in degree $3$;
\item $\frakm$ is in the support of $\calC(\rmZ_{\rmH}(\overline{\rmB}), \VV)$.
\end{enumerate}
Then we have the following isomorphism
\begin{equation*}
\rmH^{1}_{\sing}(\QQ_{p^{2}}, \rmH^{3}_{\rmc}(\overline{\rmX}(\rmB)\otimes\overline{\FF}_{p}, \rmR\Psi(\VV)(1))_{\frakm})\xrightarrow{\sim}
\calC(\rmZ_{\rmH}(\overline{\rmB}), \VV)/\lambda^{m}.
\end{equation*}
\end{theorem}
\begin{proof}
This follows from the same proof as \cite[Theorem 9.9]{Wang} with only trivial modifications to handle the non-trivial coefficient.
\end{proof}

As an application of the arithmetic level raising theorem, we can deduce the following level raising theorem.
\begin{corollary}\label{level-raise}
Let $\pi$ be an automorphic representation of $\GSp_{4}(\mathbb{A})$ as above. Suppose that $\pi$ and $\overline{\rho}_{\pi, \lambda}$ satisfy all the assumptions in Theorem \ref{arithmetic-level-raising}. Let $p$ be a level raising special prime for $\pi$ of depth $m$. Then there exists an automorphic representation $\pi^{\prime}$ of $\GSp_{4}(\mathbb{A})$ of general type with the same weight as $\pi^{\prime}$ such that
\begin{enumerate}
\item the component $\pi^{\prime}_{p}$ at $p$ is of type $\mathrm{IIa}$;
\item we have an isomorphism of the residual Galois representation
\begin{equation*}
\overline{\rho}_{\pi, \lambda}\cong \overline{\rho}_{\pi^{\prime}, \lambda}.
\end{equation*}
\end{enumerate}
In this case we will say $\pi^{\prime}$ is a level raising representation of $\pi$. 
\end{corollary}

\subsection{The main level lowering theorem}
The goal of this subsection is to prove the following theorem which is an analogue of Serre's epsilon conjecture for $\GSp_{4}$ and is the main level lowering theorem in this article.
\begin{theorem}\label{main-thm}
Let $\pi$ be a cuspidal automorphic representation of $\GSp_{4}(\mathbb{A})$ of general type which is of weight $(k_{1}, k_{2})$ satisfying  $k_{1}\geq k_{2}\geq 3$ and $k_{1}\equiv k_{2}\mod 2$.  Suppose $l+1\geq k_{1}+k_{2}$. Let $p$ be a prime different from $l$. Let $\rmU=\Pa(p)\rmU^{(p)}$ with $\rmU^{(p)}\subset \GSp_{4}(\mathbb{A}^{(\infty p)})$ such that $\rmU$ is a neat open compact subgroup such that $\pi^{\rmU}\neq 0$. Suppose that $\pi_{p}$ is ramified of type $\mathrm{IIa}$ in the table of Schmidt. Suppose also that the residual Galois representation $\overline{\rho}_{\pi, \lambda}$ satisfies the following assumptions.
\begin{enumerate}
\item  $\overline{\rho}_{\pi, \lambda}$ satisfies Assumption \ref{B-assumptions};
\item  $\overline{\rho}_{\pi,\lambda}$ is rigid for $(\Sigma_{\min}, \Sigma_{\lr})$ as in Definition \ref{rigid};
\item  $\overline{\rho}_{\pi,\lambda}$ is absolutely irreducible and the image of $\overline{\rho}_{\pi,\lambda}(\rmG_{\QQ})$ contains $\GSp_{4}(\FF_{l})$;
\end{enumerate}
Then there exists a cuspidal automorphic representation $\breve{\pi}$of general type with the same weight as $\pi$ such that 
\begin{equation*}
\overline{\rho}_{\pi, \lambda}\cong \overline{\rho}_{\breve{\pi}, \lambda} 
\end{equation*}
and $\breve{\pi}_{p}$ is an unramified principal series representation.
\end{theorem} 

The strategy to prove this theorem is similar to that of Ribet's proof of the original epsilon conjecture. First by using Chebotarev density theorem and the large image assumption, we can find a prime $q$ which is level raising special for $\pi$. We apply Corollary \ref{level-raise} to $\pi$ and obtain an cuspidal automorphic representation $\pi^{\prime}$ of general type occurring in the cohomology of $\rmX_{\Pa}(pq)$ as well as in the cohomology of $\rmX(\rmB)$ by the Jacquet--Langlands correspondence. Then we assume that all the representations appear in the cohomologies of $\rmX_{\Pa}(pq)$ and $\rmX(\rmB)$ localized at the maximal ideal corresponding to $\overline{\rho}_{\pi,\lambda}$ are ramified at $p$. Under this assumption, one can proceed as before to calculate the dimension of the space of vanishing cycles on $\rmX_{\Pa}(pq)$ and $\rmX(\rmB)$ at $p$ and $q$ and derive a contradiction.

\subsection{Proof of the main level lowering theorem}
Following the above strategy, we assume that all $\vec{\pi}$ appear in the decompositions below
\begin{equation}\label{decomp-2}
\begin{split}
&\rmH^{3}_{\rmc}(\rmX(\rmB)\otimes\overline{\QQ}, \VV)_{\frakm}\otimes E_{\lambda}=\bigoplus_{\vec{\pi}}\rho^{n(\vec{\pi})}_{\vec{\pi}, \lambda}\\
&\rmH^{3}_{\rmc}(\rmX_{\Pa}(pq)\otimes\overline{\QQ}, \VV)_{\frakm}\otimes E_{\lambda}=\bigoplus_{\vec{\pi}}\rho^{m(\vec{\pi})}_{\vec{\pi}, \lambda}\\
\end{split}
\end{equation}
are \emph{ramified} at $p$. Again by our semisimple assumption and the fact that 
$\overline{\rho}_{\pi, \lambda}\cong \overline{\rho}_{\vec{\pi}, \lambda}$,
we have
\begin{equation}\label{pq-semi}
\rmH^{3}_{\rmc}(\rmX_{\Pa}(pq)\otimes\overline{\QQ}, \VV)_{\frakm}\otimes k=\overline{\rho}^{\oplus s}_{\pi, \lambda}
\end{equation}
for some positive integer $s$ and arguing similarly as before we deduce that 
\begin{equation}\label{multi-t}
\rmH^{3}_{\rmc}(\rmX(\rmB)\otimes\overline{\QQ}, \VV)_{\frakm}\otimes k=\overline{\rho}_{\pi, \lambda}^{\oplus t}
\end{equation}
using the Jacquet--Langlands correspondence. By the existence of $\pi^{\prime}$ discussed before, $t$ must also be non-zero.

\begin{proposition}
We have the following equalities 
\begin{equation*}
\dim_{k} \bigoplus\limits_{\sigma\in\Sigma_{q}(\Pa(pq))}\rmR^{3}\Phi_{\sigma}(\mathbb{V})_{\frakm}\otimes k=s.
\end{equation*}
and 
\begin{equation*}
\dim_{k} \bigoplus\limits_{\sigma\in\Sigma_{q}(\rmB)}\rmR^{3}\Phi_{\sigma}(\mathbb{V})_{\frakm}\otimes k=3t
\end{equation*}
\end{proposition}
\begin{proof}
To prove the first equality, we consider the specialization exact sequence 
\begin{equation*}
0\rightarrow\mathrm{H}^{3}(\overline{\rmX}_{\Pa}(pq)_{\overline{\FF}_{q}}, \VV)_{\frakm}\otimes k\rightarrow \mathrm{H}^{3}(\overline{\rmX}_{\Pa}(pq)_{\overline{\FF}_{q}}, \rmR\Psi(\VV))_{\frakm}\otimes k \rightarrow \bigoplus\limits_{\sigma\in\Sigma_{q}(\Pa(pq))}\rmR^{3}\Phi_{\sigma}(\mathbb{V})_{\frakm}\otimes k\rightarrow 0.
\end{equation*}
Note that $q$ is a level raising special prime, the Hecke parameters of $\pi$ are distinct modulo $\lambda$, thus the Frobenius eigenvalues are distinct modulo $\lambda$. Since the three Frobenius eigenvalues already appear in $\mathrm{H}^{3}(\overline{\rmX}_{\Pa}(pq)\otimes\overline{\FF}_{q}, \VV)_{\frakm}\otimes k$, the dimension equality follows.

We now prove the second equality. The same argument as above proves the equality 
\begin{equation*}
\dim_{k}\ker(\alpha)_{\frakm}\otimes k=t.
\end{equation*}
By definition and torsion freeness of $\rmH^{4}(\overline{\rmX}(\rmB)\otimes{\overline{\FF}_{q}}, \VV)_{\frakm}$ in \cite[Corollary 9.8]{Wang}, we have an exact sequence
\begin{equation*}
0\rightarrow\ker(\alpha)\otimes k\rightarrow \bigoplus\limits_{\sigma\in\Sigma_{q}(\rmB)}\rmR^{3}\Phi_{\sigma}(\mathbb{V})_{\frakm}\otimes k\rightarrow \mathrm{H}^{4}(\overline{\rmX}(\rmB)\otimes\overline{\FF}_{q}, \VV)_{\frakm}\otimes k \rightarrow 0.
\end{equation*}
Since we have an isomorphism 
\begin{equation*}
\mathrm{H}^{4}(\overline{\rmX}(\rmB)\otimes\overline{\FF}_{q}, \VV_{k})_{\frakm}\cong \calC(\rmZ_{\rmH}({\rmB^{\prime}}), \VV_{k})^{\oplus 2}_{\frakm}
\end{equation*}
by \cite[Corollary 9.8]{Wang} and our main result of arithmetic level raising gives the isomorphism
\begin{equation*}
\rmH^{1}_{\sing}(\QQ_{q^{2}}, \rmH^{3}_{\rmc}(\overline{\rmX}(\rmB)\otimes\overline{\FF}_{q}, \rmR\Psi(\VV_{k})(1))_{\frakm})\xrightarrow{\sim}
\calC(\rmZ_{\rmH}({\rmB^{\prime}}), \VV_{k}),
\end{equation*}
note that the role of the two primes $p$ and $q$ are switched in the statement of Theorem \ref{arithmetic-level-raising}.
A simple calculation shows 
\begin{equation*}
\rmH^{1}_{\sing}(\QQ_{q^{2}}, \rmH^{3}_{\rmc}(\overline{\rmX}(\rmB)\otimes\overline{\FF}_{q}, \rmR\Psi(\VV_{k})(1))_{\frakm})
\end{equation*}
is of dimension $t$ and it follows that 
\begin{equation*}
\dim_{k} \mathrm{H}^{4}(\overline{\rmX}(\rmB)\otimes\overline{\FF}_{q}, \VV)_{\frakm}\otimes k=2t. 
\end{equation*}
Hence we have 
\begin{equation*}
\dim_{k}\bigoplus\limits_{\sigma\in\Sigma_{q}(\rmB)}\rmR^{3}\Phi_{\sigma}(\mathbb{V})_{\frakm}\otimes k =3t.
\end{equation*}
\end{proof}

\begin{proposition}
Suppose that all those representation $\vec{\pi}$ appearing in \eqref{decomp-2} are all ramified at $p$. Then we have
\begin{equation*}
\dim_{k} \bigoplus\limits_{\sigma\in\Sigma_{p}(\rmB)}\rmR^{3}\Phi_{\sigma}(\mathbb{V})_{\frakm}\otimes k=2t
\end{equation*}
and 
\begin{equation*}
\dim_{k} \bigoplus\limits_{\sigma\in\Sigma_{p}(\Pa(pq))}\rmR^{3}\Phi_{\sigma}(\mathbb{V})_{\frakm}\otimes k=2s.
\end{equation*}
\end{proposition}
\begin{proof}
We note that by our assumption that $\vec{\pi}$ are all ramified in \eqref{decomp-2}, these equalities  follow from the same analysis as in Proposition \ref{2t-degen}. 
\end{proof}

\begin{myproof}{Theorem}{\ref{main-thm}}
Recall that in Corollary \ref{pq-sing}, we have proved that $\Sigma_{p}(\rmB)$ and $\Sigma_{q}(\Pa(pq))$ can be identified and similarly $\Sigma_{q}(\rmB)$ and $\Sigma_{p}(\Pa(pq))$ can be identified. It follows immediately that $2s=3t$ and $s=2t$ by the previous two propositions under our assumption that $\vec{\pi}_{p}$ is ramified for each $\vec{\pi}$ in \eqref{decomp-2}. Hence $s=t=0$ and we have a contradiction. This implies we can drop the level at $p$ from the level raised representation $\pi^{\prime}$. Finally we apply Mazur's principle to drop the level at $q$ to obtain the representation $\breve{\pi}$. 
\end{myproof}

\end{document}